\documentclass[10pt]{amsart}

\usepackage{amsmath,amsthm}
\usepackage{graphics}
\usepackage{graphicx}

\numberwithin{equation}{section}

\usepackage[a4paper, left=35mm,right=35mm,top=34mm,bottom=34mm]{geometry}
\usepackage[utf8]{inputenc}
\usepackage[T1]{fontenc}
\usepackage[english]{babel}

\usepackage{enumerate}
\usepackage{graphicx}
\usepackage{hyperref}
\hypersetup{
    colorlinks=true,
    linkcolor=blue,
    filecolor=magenta,      
    urlcolor=cyan,
}
\usepackage{lipsum}
\usepackage{epstopdf}
\usepackage{algorithmic}
\usepackage[mathscr]{euscript}

 \usepackage{dsfont}
 \usepackage{stmaryrd}
 \usepackage{psfrag}
 \usepackage{color}
 \usepackage{tikz}
 \usepackage{tikz-cd}
 \usepackage{subfigure}
\usepackage{mathtools,amsthm,amssymb,amsfonts}
\usepackage{amsmath}
\usepackage{cite}
\usepackage{hyperref}
\usepackage{psfrag}
\usepackage{bbm}
\usepackage{algorithm}
\usepackage{inputenc}
\usepackage{csquotes}

\makeatother
\theoremstyle{plain}
\def\bord{\partial\Omega}
\def\dTTr{\operatorname{\dot{T}r}}
\def\TTr{\operatorname{Tr}}
\def\Tr{\operatorname{tr}}
\def\dTr{\operatorname{\dot{t}r}}
\def\loc{\operatorname{loc}}
\def\B{\mathcal{B}}
\def\Hp{\B(\bord)}
\def\Hm{\B'(\bord)}
\def\Hpo{\dot{\B}(\bord)}
\def\Hmo{\dot{\B}'(\bord)}

\def\Nscr{\mathscr{N}}
\def\Ascr{\mathscr{A}}
\def\Dcal{\mathcal{D}}
\def\Scal{\mathcal{S}}
\def\Vcal{\mathcal{V}}
\def\Wcal{\mathcal{W}}

\def\Kcal{\mathcal{K}}
\def\Gcal{\mathcal{G}}
\def\Ical{\mathcal{I}}
\def\Rcal{\mathcal{R}}
\def\NOt{\dot{\Nscr}}

\def\del{\partial}

\def\R{\mathbb{R}}

\def\C{\mathbb{C}}

\def\dx{{\rm d}x}
\def\dy{{\rm d}y}

\def\nbord{\R^n\backslash\bord}

\DeclareMathOperator{\supp}{supp}
\DeclareMathOperator{\cpct}{Cap}

\newcommand{\ddn}[1]{\frac{\partial #1}{\partial\nu}}
\newcommand{\ddny}[2]{\frac{\partial_{#2} #1}{\partial\nu}}

\newtheorem{theorem}{Theorem}[section]

\newtheorem{proposition}[theorem]{Proposition}

\newtheorem{corollary}[theorem]{Corollary}
\newtheorem{lemma}[theorem]{Lemma}
\newtheorem{remark}[theorem]{Remark}

\newtheorem{examples}[theorem]{Examples}

\usetikzlibrary{cd}

\usepackage{caption} 

\usepackage{fancyhdr}

\providecommand{\keyword}[1]{\noindent\textbf{Keywords:} #1}
\providecommand{\MSC}[1]{\noindent\textbf{2020 MSC:} #1}

\author{Gabriel Claret}\thanks{GC: CentraleSup\'elec, Universit\'e Paris-Saclay, France (gabriel.claret@centralesupelec.fr) Research partially included in the \enquote{Parcours Recherche} teaching program at CentraleSupélec.}
\author{Michael Hinz}\thanks{MH: University of Bielefeld, Germany
(mhinz@math.uni-bielefeld.de) Research supported in part by the DFG IRTG 2235: \enquote{Searching for the regular in the irregular: Analysis of singular and random systems} and by the DFG CRC 1283: \enquote{Taming uncertainty and profiting from randomness and low regularity in analysis, stochastics and their applications}. }
\author{Anna Rozanova-Pierrat}\thanks{ARP: CentraleSup\'elec, Universit\'e Paris-Saclay, France
(anna.rozanova-pierrat@centralesupelec.fr) Research supported in part by CNRS INSMI IEA (International Emerging Actions 2022) \enquote{Functional and applied analysis with fractal or
non-Lipschitz boundaries}}
\author
{Alexander Teplyaev}\thanks{AT: University of Connecticut, USA
(teplyaev@uconn.edu) Research supported in part by the NSF 
DMS 1613025, 
1950543, 2349433 
and the Simons and Fulbright Foundations. The project was partially conducted during a stay at CentraleSup\'elec, Universit\'e Paris-Saclay, whose kind hospitality is gratefully acknowledged.}

\title{Layer potential operators for transmission problems on extension domains}
\date{\today}

\begin{document}

\begin{abstract}
We use the well-posedness of transmission problems on classes of two-sided Sobolev extension domains to give variational definitions for (boundary) layer potential operators and Neumann-Poincaré operators. These classes of domains contain Lipschitz domains, and also domains with fractal boundaries. 
Although our variational formulation does not involve any measures on the boundary, we recover the classical results in smooth domains by considering the surface measure on the boundary. We discuss properties of these operators and generalize basic results in imaging beyond the Lipschitz case.
\end{abstract}

\maketitle

\keyword{fractal boundaries; layer potentials; Poincaré-Steklov operator; Neumann-Poincaré operator; Calder\'on projector; elliptic transmission problems; inverse problems.}

\MSC{31E05; 35J25; 31B10; 31A10; 31B20; 31A25.}

\tableofcontents
\section{Introduction}

The aim of our work is to introduce a purely variational approach in order to extend the   
classical theory of layer potential operators, Neumann-Poincar\'e operators, 
boundary layer operators and transmission problems to a class of rough (two-sided) extension domains, 
with possibly non-Lipschitz or even fractal boundaries. For rectifiable boundaries a deep study of layer potential operators and related topics has been carried out in the fundamental recent books \cite{MITREAx3-I,MITREAx3-II,MITREAx3-III,MITREAx3-IV,MITREAx3-V}. Here we extend several basic results to a much wider class of boundaries with no rectifiability assumption.
We work on \emph{extension domains} \cite{JONES-1981,HAJLASZ-2008}. They are not necessarily Lipschitz and may have fractal boundaries.
 
In the classical theory of partial differential equations on smooth domains, the (boundary) layer potential and Neumann-Poincar\'e operators are defined as convolutions with Green's functions and their first derivatives, see for instance \cite{FABES-1978, KHAVINSON-2006, MAZYA-1991, TRIEBEL-1992, VLADIMIROV-1971} or \cite{ANDO-2021}. They provide explicit representation formulas for solutions to boundary value problems, and they are fundamental tools in inverse methods, numerical analysis and certain areas of spectral theory. 

A systematic study of the classical approach to layer potentials in the context of Lipschitz domains was provided in \cite{VERCHOTA-1984}, following seminal studies on boundary value problems \cite{JERISON-1982}, and singular integrals \cite{CALDERON-1977, COIFMAN-1982}. Since then, the Lipschitz case has become a standard level of generality for many applications, see for instance \cite{AMMARI-2004,AMMARI-2009, COSTABEL-1988,COSTABEL-WENDLAND-1986, DAHLBERG-1988, DAHLBERG-1990, HSIAO-2008,MCLEAN-2000,NEDELEC-2001, STEINBACH-2001}. 
A different approach was adopted in \cite{BARTON-2014}, 
where the weak well-posedness of transmission problems 
was used to define the layer potentials for Lipschitz domains.
In view of the well-established use of Hilbert space methods in potential theory --- 
see \cite{COURANT-HILBERT-1937, CONSTANTINESCU-CORNEA-1963, DENY-LIONS-1954,
WEYL-1940} for classical references and \cite{BOULEAU-HIRSCH-1991, CHEN-FUKUSHIMA-2012, FOT94, LEJAN-1978} for later developments --- the variational definitions in \cite{BARTON-2014} are very natural.

Boundary layer operators on piecewise smooth, not necessarily Lipschitz domains were already studied in \cite[Chapter 5]{MAZYA-1991}. Results for layer potentials in the context of Riemannian manifolds were obtained in \cite{MITREA-TAYLOR-1999},  results for layer potentials on half-spaces with boundary data in Besov spaces in \cite{BARTON-MAYBORODA-2016}. The research on the boundedness of singular integral operators in \cite{CALDERON-1977, COIFMAN-1982, VERCHOTA-1984} culminated in the comprehensive study \cite{MITREAx3-I,MITREAx3-II,MITREAx3-III,MITREAx3-IV,MITREAx3-V} of such operators on uniformly rectifiable sets \cite{DAVID-SEMMES-1991, DAVID-SEMMES-1993}. 
Uniformly rectifiable sets in $\R^n$ are Ahlfors $(n-1)$-regular closed subsets having \enquote{big pieces of Lipschitz images}, see \cite[Definitions 5.10.1 and 5.10.2]{MITREAx3-I}, and the class of such sets is basically characterized by the $L^2$-boundedness of singular integral operators, \cite[Theorem 5.10.2]{MITREAx3-I}. A different major stream of research focused on the behaviour of harmonic functions on a domain $\Omega$ and harmonic measures on its boundary $\partial\Omega$, see \cite{Dahlberg-1977, Jerison-1981, JERISON-1982, Jones-1982} and the later references \cite{Capogna-2005, David-1990, Bass-Burdzy-1991, Aikawa-2007}. 

We are interested in boundary value problems involving rough boundaries that may be fractal and may even have different parts of different Hausdorff dimensions. Well-known and more specific references on linear elliptic partial differential equations on fractal domains are \cite{NYSTROM-1996, JONSSON-1997,  LAPIDUS-1991, LAPIDUS-1995}, applications to the heat equation were studied in~\cite{LEVITIN-1996,VANDENBERG-2000}. References closely related to our work are the articles \cite{BANNISTER-2022, Caetano-2021, CHANDLER-WILDE-2017, ARXIV-CAETANO-2022, Gibbs-2023,MAGOULES-2021} on Helmholtz models, \cite{DEKKERS-2022, DEKKERS-2022-1} on the non-linear wave equation (the Westervelt equation) and \cite{HINZ-2021-1, HINZ-2021, HINZ-2023} on varying domains and the existence of optimal shapes.

The highly active research around domains with rough boundaries motivates an extension of the variational approach to layer potentials beyond the Lipschitz case, and follow-up questions about convergence and optimization make it desirable to have a generalization which is flexible and rather easy to handle.

The purpose of  our work is to propose a lightweight and streamlined generalization of basic results on layer potential operators. The domains we consider are \enquote{first order} Sobolev extension domains $\Omega$ in $\R^n$ \cite{JONES-1981, HAJLASZ-2008}, having a boundary $\partial\Omega$ of positive capacity \cite{MAZ'JA-1985}. Basic examples in the planar case are quasidisks or complements of Cantor sets of positive Hausdorff dimension. Domains having sharp inward or outward cusps or \enquote{collapsing} boundaries are not covered; in particular, certain fractal trees \cite{ACHDOU-2008} do not fall within the scope of our work.
 
We use established results \cite{BIEGERT-2009} to introduce trace and extension operators between the domain $\Omega$ and its boundary $\partial\Omega$ in the context of the Sobolev space $H^1(\Omega)$ (respectively, the homogeneous Sobolev space $\dot{H}^1(\Omega)$). On the boundary $\partial\Omega$, we use the corresponding abstract trace space $\Hp$ (respectively, $\Hpo$), endowed with natural trace norms. 
That approach is \emph{measure free} in the sense that we do not specify any \enquote{surface measure} on $\partial\Omega$, nor do we consider $L^2$-spaces on $\partial\Omega$, unlike \cite{VERCHOTA-1984}. Instead, we make a systematic use of the trace spaces $\Hp$ and $\Hpo$ and their respective duals $\Hm$ and $\Hmo$. As a consequence, we observe a variety of \emph{natural isometries}; the preservation of those isometries might be considered a guiding theme of 
our work.

To set  notations, we briefly survey Dirichlet and Neumann problems in the weak (variational) sense for one-sided domains in Section \ref{SecBoundCase}. 
In Section \ref{Sec:P-S}, we briefly discuss basic features of the related Poincar\'e-Steklov operators. Suitable two-sided domains are then introduced in Section \ref{Sec:TrProb}. 
Roughly speaking, we require that both the inner domain $\Omega$ and the outer domain $\R^n\backslash \overline{\Omega}$ are non-empty and Sobolev extension domains, and we assume that the separating boundary $\partial\Omega$ has zero Lebesgue measure.
In the homogeneous case, we additionally assume that $\Omega$ is bounded. 
In Section~\ref{Sec:Resolvent}, we link the layer potential operators we define to resolvent representations and recover the classical integral formulas.
In Section~\ref{Sec:Neumann-Poincaré}, we introduce Neumann-Poincar\'e operators, boundary layer potentials and Calder\'on projectors and study some of their properties.
In Section~\ref{Sec:Imaging}, we discuss an application to a problem in imaging, formerly understood in the Lipschitz case only. In \ref{PSecBoundCase} we briefly recall some potential theoretic notions and provide some details for the proofs of results of Section~\ref{SecBoundCase}.

\subsection*{Acknowledgements}

The authors express their sincere gratitude to 
David Hewett and Simon N. Chandler-Wilde 
for their invaluable advice and insightful discussions, 
which significantly contributed to the advancement of our work.
They also thank the anonymous referees whose comments
and suggestions helped to improve the manuscript.

\section{Boundary value problems on admissible domains}\label{SecBoundCase}

\subsection{Admissible domains and traces}\label{SS:admissible}

Let $\Omega$ be a nonempty open subset of $\R^n$. As usual, we write $H^1(\Omega)$ for the Hilbert space of all $u\in L^2(\Omega)=L^2(\Omega,\R)$ such that $\nabla u \in L^2(\Omega,\R^n)$ and having the scalar product 
\begin{equation}\label{E:sp}
\left\langle u,v\right\rangle_{H^1(\Omega)}=\int_\Omega \nabla u\cdot\nabla v~\dx+\int_\Omega uv~\dx;
\end{equation}
here $\nabla u$ is interpreted in distributional sense.
We write $\dot{H}^1(\Omega)$ for the Hilbert space formed by the vector space of all $u\in L^2_{\loc}(\Omega)$ with $\nabla u\in L^2(\Omega,\R^n)$ modulo locally constant functions, endowed with the scalar product 
\begin{equation}\label{E:spdot}
\left\langle u,v\right\rangle_{\dot{H}^1(\Omega)}=\int_\Omega \nabla u\cdot\nabla v~\dx.
\end{equation}
Details on the space $\dot{H}^1(\Omega)$ in the case of connected $\Omega$ can be found in \cite[Section 2.2.4]{CHEN-FUKUSHIMA-2012}, \cite{DENY-LIONS-1954} or \cite[Section 1.1.13]{MAZ'JA-1985}; a generalization to nonempty open subsets with multiple connected components is straightforward. We agree to use the notations \eqref{E:sp} and \eqref{E:spdot} whenever the right-hand side makes sense. As usual, the Hilbert space norms induced by \eqref{E:sp} respectively \eqref{E:spdot} are denoted by $\|\cdot\|_{H^1(\Omega)}$ respectively $\|\cdot\|_{\dot{H}^1(\Omega)}$.

We use the notions of capacity, quasi continuous representatives and quasi everywhere (q.e.) valid statements exclusively with respect to the space $H^1(\R^n)$. For the convenience of the reader, some background on these notions is collected in \ref{PSecBoundCase}. As explained there, if $u\in H^1(\mathbb{R}^n)$ or $u\in \dot{H}^1(\mathbb{R}^n)$, then $u$ has a quasi continuous representative $\widetilde{u}$.

We call a connected nonempty open set $\Omega\subset\mathbb{R}^n$ an \textit{$H^1$-extension domain} \cite{JONES-1981,HAJLASZ-2008}, if there is a bounded linear extension operator $\mathrm{E}_\Omega:H^1(\Omega)\to H^1(\R^n)$. 
If $\Omega$ is an $H^1$-extension domain and its boundary $\partial\Omega$ has positive capacity, then we call it an \emph{$H^1$-admissible} domain. 

Assume that $\Omega$ is $H^1$-admissible. We write $\mathcal{B}(\partial\Omega)$ for the vector space of all q.e. equivalence classes of pointwise restrictions $\widetilde{w}|_{\partial\Omega}$ of quasi continuous representatives $\widetilde{w}$ of classes $w\in H^1(\R^n)$. Given $u\in H^1(\Omega)$, we choose an arbitrary element $w$ of $H^1(\mathbb{R}^n)$ such that $w=u$ a.e. in $\Omega$ and define $\TTr_i u:=\widetilde{w}|_{\partial\Omega}$,  where $\widetilde{w}$ is an arbitrary quasi continuous representative $\widetilde{w}$ of $w$. By the following consequence of \cite[Theorem 6.1 and Remark 6.2]{BIEGERT-2009} we may regard $\TTr_i$ as the natural trace operator from $H^1(\Omega)$ onto $\mathcal{B}(\partial\Omega)$.

\begin{proposition}\label{P:trace}
Let $\Omega$ be an $H^1$-admissible domain. Then $u\mapsto \TTr_i u$ gives a linear surjection $\TTr_i:H^1(\Omega)\to \mathcal{B}(\partial\Omega)$, well-defined in the sense that given $u\in H^1(\Omega)$, its trace $\TTr_i u$ on $\partial\Omega$ does not depend on the particular choice of $w$ or $\widetilde{w}$. 
\end{proposition}

In particular, we have $\TTr_i u=(\mathrm E_\Omega u)^\sim|_{\partial\Omega}$, $u\in H^1(\Omega)$, for any bounded linear extension operator $\mathrm{E}_\Omega:H^1(\Omega)\to H^1(\R^n)$, see \cite[Corollary 6.3]{BIEGERT-2009}.

We use the subscript $i$ in $\TTr_i$ because we will later discuss two-sided domains and $\TTr_i$ as defined here will play the role of an \emph{interior} 
trace operator with respect to $\Omega$.

We call a connected nonempty open set $\Omega\subset \R^n$ an \textit{$\dot{H}^1$-extension domain} if there is a bounded linear extension operator $\dot{\mathrm{E}}_\Omega:\dot{H}^1(\Omega)\to \dot{H}^1(\R^n)$, and we call it \emph{$\dot{H}^1$-admissible} if it is an $\dot{H}^1$-extension domain and $\partial\Omega$ is compact and of positive capacity. If so, we write  $\dot{\mathcal{B}}(\partial\Omega)$ for the vector space {\color{black} of all q.e. equivalence classes modulo constants of pointwise restrictions $\widetilde{w}|_{\partial\Omega}$ of quasi continuous representatives $\widetilde{w}$ of elements $w$ of $\dot{H}^1(\R^n)$.} We point out that, as said before, capacities, q.e. notions and quasi continuity are all with respect to $H^1(\R^n)$, not $\dot{H}^1(\R^n)$; see \ref{PSecBoundCase}. {\color{black} Given an element $u$ of $\dot{H}^1(\Omega)$, we choose an arbitrary element $w$ of $\dot{H}^1(\mathbb{R}^n)$ extending $u$ beyond $\Omega$ to $\mathbb{R}^n$ and define $\dTTr_i u$ to be the q.e. equivalence class of $\widetilde{w}|_{\partial\Omega}$ modulo constants of an arbitrary quasi continuous representative $\widetilde{w}$ of $w$.}
 A variant of Proposition \ref{P:trace} shows that $\dTTr_i$ is the natural trace operator from $\dot{H}^1(\Omega)$ onto $\dot{\mathcal{B}}(\partial\Omega)$. 

\begin{proposition}\label{P:traceo}\mbox{}
Let $\Omega$ be an $\dot{H}^1$-admissible domain. Then $u\mapsto \dTTr_i u$ gives a linear surjection $\dTTr_i:\dot{H}^1(\Omega)\to \dot{\mathcal{B}}(\partial\Omega)$, well-defined in the sense that given $u\in \dot{H}^1(\Omega)$, {\color{black} its trace $\dTTr_i u$ on $\partial\Omega$ does not depend on the particular choice of the extension $w$ of $u$, nor on the choice of the representative $\widetilde{w}$.}
\end{proposition}

\begin{proof}
Given $u\in \dot{H}^1(\Omega)$, let ${\color{black} w}\in \dot{H}^1(\mathbb{R}^n)$ be such that ${\color{black} w}=u$ a.e. on $\Omega$ modulo constants and choose 
 a representative ${\color{black} v\in w}$ modulo constants. Let $U$ be a relatively compact open neighbourhood of $\partial\Omega$ and $\chi\in C_c^\infty(U)$ be a nonnegative function such that $0\leq \chi\leq 1$ and $\chi\equiv 1$ on a neighbourhood of $\partial\Omega$. Then $\chi {\color{black} v}$ is in $H^1(\mathbb{R}^n)$. By \cite[Theorem 6.1 and Remark 6.2]{BIEGERT-2009} the restriction ${\color{black}\widetilde{v}}|_{\partial\Omega}=(\chi {\color{black} v})^\sim|_{\partial\Omega}$ is uniquely determined in the q.e. sense. Its class modulo constants in $\dot{\mathcal{B}}(\partial\Omega)$ does not depend on the choice of {\color{black} $v$}.
\end{proof}

\begin{remark}
Classical references on the spaces $\dot{H}^1(\Omega)$ are \cite{DENY-LIONS-1954} and \cite[Sections 1.1.2 and 1.1.13]{MAZ'JA-1985}; a more recent discussion may be found in \cite[Section II.6]{GALDI-2011}. In those references, different symbols are used to denote these spaces. For the case $\Omega=\R^n$, the \enquote{dot}-notation $\dot{H}^1(\R^n)$ is established, cf.  \cite[Chapter 5]{TRIEBEL-2010}. Since the domains we consider are $H^1$- respectively $\dot{H}^1$-extension domains, we follow the notation for the $\R^n$-case and write $\dot{H}^1(\Omega)$.
\end{remark}

\begin{examples} \mbox{} 
\begin{enumerate}
\item[(i)] A rich class of examples for extension domains is provided in \cite{JONES-1981}. For $n\geq 2$ any $(\varepsilon,\delta)$-domain $\Omega\subset \R^n$ is an $H^1$-extension domain \cite[Theorem 1]{JONES-1981}, and any $(\varepsilon,\infty)$-domain $\Omega\subset \R^n$ is an $\dot{H}^1$-extension domain \cite[Theorem 2]{JONES-1981}. Uniform domains \cite{MARTIO-1979, VAISALA-1988} are $(\varepsilon,\delta)$-domains.
\item[(ii)] For $n\geq 2$ any $(\varepsilon,\infty)$-domain $\Omega\subset \R^n$ with $\R^n\backslash \overline{\Omega}$ nonempty is $H^1$-admissible, and if one of the two open sets is bounded, it is also $\dot{H}^1$-admissible.
\item[(iii)] For $n=1$ any interval $(a,b)\subset [-\infty,+\infty]$ with $a$ or $b$ finite is $H^1$-admissible, and if both are finite, also $\dot{H}^1$-admissible. For $n\geq 2$ the domain $\Omega=\R^n\backslash\{0\}$ is not $H^1$-admissible. 
\end{enumerate}
\end{examples}

\begin{remark}
A prominent class of domains, highly relevant in the study of harmonic measures,  is the class of NTA domains \cite[Section 3]{JERISON-1982}, see \cite{NYSTROM-1996,JONES-1980,Azzam-2017,Toro-2017} for discussions. It contains  all Lipschitz domains, but also domains with possibly fractal boundaries such as quasidisks. Any NTA domain is a uniform domain, see for instance \cite[Theorem 2.15]{Azzam-2017}. Any uniform domain with uniformly rectifiable boundary is an NTA domain 
\cite[Theorem 1.1]{Azzam-2017}. In $\R^2$ there is an equivalence~\cite[Theorem 2]{DEKKERS-2022},~\cite{JONES-1980} for bounded simply connected NTA domains with uniform domains. 
\end{remark}

\begin{remark}\label{R:trace}\mbox{}
\begin{enumerate}
\item[(i)] Suppose that $\Omega$ is a bounded $H^1$-extension domain. Then the vector space $\dot{H}^1(\Omega)$ is isomorphic to the space of all $u\in H^1(\Omega)$ with $\int_\Omega u(x)\:\dx=0$, and by Poincar\'e's inequality, \eqref{E:sp} and \eqref{E:spdot} are equivalent scalar products on this space. In particular, the vector spaces $H^1(\Omega)$ and $\dot{H}^1(\Omega)\oplus \R$ are isomorphic.
\item[(ii)] For a bounded domain $\Omega$ which is both $H^1$- and $\dot{H}^1$-admissible, the vector spaces $\mathcal{B}(\partial\Omega)$ and $\dot{\mathcal{B}}(\partial\Omega)\oplus \R$ are isomorphic.
\end{enumerate}
\end{remark}

\begin{remark}\label{R:boundary_spaces}
If $\Omega$ is a bounded Lipschitz domain, then, up to equivalent norms, $\Hp$ equals $H^{1/2}(\partial\Omega)$ and $\Hm$ equals the dual $H^{-1/2}(\partial\Omega)$ of $H^{1/2}(\partial\Omega)$, cf. ~\cite[Chapter IV, Appendix]{DUTRAY-LIONS1988}.
The \enquote{dot} versions equal the homogeneous counterparts $\dot{H}^{1/2}(\partial\Omega)$ respectively $\dot{H}^{-1/2}(\partial\Omega)$ of these spaces. The spaces $H^{1/2}(\partial\Omega)$ can be endowed with explicit norms of fractional Sobolev type; these norms involve the surface measure $\sigma$ on $\partial\Omega$.

If, more generally, $\partial\Omega$ is the support of a measure $\mu$ satisfying certain scaling conditions, then the trace space $\Hp$ is a Besov type space. It can be endowed with an explicit norm involving $\mu$, see \cite{JONSSON-1984, JONSSON-1994}. 

Here we do not require $\partial\Omega$ to be Lipschitz or to carry any measure. Moreover, $\partial\Omega$ may have parts of different Hausdorff dimensions. The present formulation does not give any explicit norm representation for $\Hp$ and may therefore not be sufficient to discuss regularity features. But it works under minimal assumptions, which is useful in view of  convergence and compactness properties, cf. \cite{HINZ-2021-1,HINZ-2021,HINZ-2023}. 
\end{remark}

\begin{remark}\label{R:multiplecomponents}
The spaces $\mathcal{B}(\partial\Omega)$ and $\dot{\mathcal{B}}(\partial\Omega)$ are defined under the assumption that $\Omega$ is connected. A generalization of the spaces $\mathcal{B}(\partial\Omega)$ to the case of finite unions $\Omega$ of mutually disjoint $H^1$-admissible domains does not pose any problem. Care is needed for the spaces $\dot{\mathcal{B}}(\partial\Omega)$. Suppose that $\Omega=\Omega_1\cup ...\cup \Omega_N$ with $\dot{H}^1$-admissible $\Omega_j$ having mutually disjoint closures $\overline{\Omega}_j$. The space $\dot{\mathcal{B}}(\partial\Omega)$ can be defined as a space of classes of q.e. defined functions on $\partial\Omega$ modulo locally constant functions constant on each $\partial\Omega_j$. The argument of Proposition \ref{P:traceo} can be applied separately to each connected component $\Omega_j$ by taking mutually disjoint neighbourhoods $U_j$ of the $\partial\Omega_j$, respectively. Similarly as before, this gives a natural trace operator. However, the situation may call for a refined notation, in particular, when discussing complements. In $\mathbb{R}^2$ both $\Omega:=B(0,2)\setminus \overline{B(0,1)}$ and $\mathbb{R}^2\setminus \overline{\Omega}=B(0,1)\cup (\mathbb{R}^2\setminus \overline{B(0,2)})$ are $\dot{H}^1$-admissible and have the same boundary $\partial\Omega$. The space $\mathcal{B}(\partial\Omega)$, based on $\Omega$, and the space $\mathcal{B}(\partial\Omega)$, based on $\mathbb{R}^2\setminus\overline{\Omega}$, differ: while the elements of the former are defined modulo a single constant, the elements of the latter have to be understood \enquote{modulo two constants}, one on $\partial B(0,1)$ and another on $\partial B(0,2)$. 
\end{remark}

\subsection{Orthogonality, harmonic extensions and isometries}\label{SS:ortho}

We center our discussion of boundary value problems around restrictions, duals and inverses of trace operators.

Let $\Omega\subset \R^n$ be an $H^1$-admissible domain. We write $H^1_0(\Omega)$ for the closure in $H^1(\Omega)$ of the set $C^\infty_c(\Omega)$ of infinitely differentiable functions with compact support in $\Omega$. Let $V_1(\Omega)$ denote the orthogonal complement of $H^1_0(\Omega)$ in $H^1(\Omega)$,
\begin{equation}\label{V1}
H^1(\Omega)=H^1_0(\Omega)\oplus V_1(\Omega).
\end{equation}

Given $f\in \Hp$, an element $u$ of $H^1(\Omega)$ is called a \emph{weak solution} of the Dirichlet problem
\begin{equation}\label{DO}
\begin{cases}
-\Delta u+u&=0\quad \text{in $\Omega$}\\
\qquad u|_{\partial\Omega}&=f
\end{cases}
\end{equation}
if $\TTr_i u=f$ and $\left\langle u,v\right\rangle_{H^1(\Omega)}=0$ for all $v\in C_c^\infty(\Omega)$.
The symbol $u|_{\partial\Omega}$ in the formal problem \eqref{DO} stands for the restriction of $u$ to the boundary; in the context of weak solutions it is made rigorous through the trace condition.
For any $f\in \Hp$, the Dirichlet problem \eqref{DO} has a unique weak solution $u^f$; this is well known and immediate from \eqref{V1}. By the first line in \eqref{DO}, a solution $u^f$ is called \emph{$1$-harmonic} on $\Omega$. It is in $V_1(\Omega)$, which is the space of all elements of $H^1(\Omega)$ that are $1$-harmonic on $\Omega$.
 
We write $\Tr_i:=\TTr_i|_{V_1(\Omega)}$ for the restriction of $\TTr_i$ defined in Proposition \ref{P:trace} to $V_1(\Omega)$. The following theorem generalizes well-known results to the framework of $H^1$-admissible domains.

\begin{theorem}\label{Trisom}
Let $\Omega\subset \R^n$ be an $H^1$-admissible domain. Then the following statements are true:
\begin{enumerate}
\item[(i)] The space $H^1_0(\Omega)$ is the kernel of $\TTr_i$, that is, $H^1_0(\Omega)=\ker\TTr_i$.
\item[(ii)]  Endowed with the norm 
\begin{equation}\label{normTri}
\left\|f\right\|_{\Hp}:=\min\{ \left\|v\right\|_{H^1(\Omega)}|\ \text{$v\in H^1(\Omega)$ and $\TTr_i\:v=f$}\},
\end{equation}	
the space $\Hp$ is a Hilbert space. 
\item[(iii)] With respect to $\left\|\cdot\right\|_{\Hp}$, the trace operator trace $\TTr_i$ is bounded with operator norm one. Its restriction $\Tr_i:V_1(\Omega)\to \mathcal{B}(\partial\Omega)$ to $V_1(\Omega)$ is an isometry and onto.
\end{enumerate}
\end{theorem}

Details on statement (i) can be found in \ref{PSecBoundCase}, statements (ii) and (iii) are direct consequences.

\begin{remark}
By Theorem \ref{Trisom} the $1$-harmonic extension operator $\Tr_i^{-1}:\mathcal{B}(\partial\Omega)\to V_1(\Omega)$ is an isometry. 
For any $f\in \mathcal{B}(\partial\Omega)$ we have $u^f=\Tr_i^{-1}f$. 
\end{remark}

We write  $\Hm$ and $(H^1(\Omega))'$ for the dual spaces of $\Hp$ and $H^1(\Omega)$, and we use the notation 
\[\langle \cdot, \cdot \rangle_{\Hm,\Hp}\quad\text{and}\quad \langle \cdot, \cdot\rangle_{(H^1(\Omega))',H^1(\Omega)}\] 
for the corresponding dual pairings. 

\begin{remark}
Since $\Hp$ is a Hilbert space, for any $f\in \Hp$, the assignment $\iota(f)(h):=\left\langle f,h\right\rangle_{\Hp}$, $h\in \Hp$, defines an isometric isomorphism $\iota$ from $\Hp$ onto $\Hm$.
The dual pairing can be expressed as
$$\langle g, f \rangle_{\Hm,\B(\del \Omega)}=\langle \iota^{-1}(g),f\rangle_{\Hp}=\left\langle g,\iota(f)\right\rangle_{\Hm}, \;  f\in \Hp, \; g\in \Hm.$$
We may identify $\Hp$ with its image $\iota(\Hp)\subset \Hm$ under $\iota$.
\end{remark} 

\begin{remark}\label{R:extendbyzero}
If $V_1'(\Omega)$ denotes the dual of the closed subspace $V_1(\Omega)$ of $H^1(\Omega)$, then $(H^{1}(\Omega))'\subset V_1'(\Omega)$ by restriction. However, by the Riesz representation theorem, any $w\in V_1'(\Omega)$ is represented as $w=\left\langle v,\cdot\right\rangle_{H^1(\Omega)}$ with some suitable $v\in V_1(\Omega)$. The orthogonal decomposition \eqref{V1} then implies that $w$ automatically extends to a unique bounded linear functional $w'\in (H^{1}(\Omega))'$ on all of $H^1(\Omega)$ and zero on $H_0^1(\Omega)$, and that extension is an isometry, $\left\|w'\right\|_{(H^{1}(\Omega))'}=\left\|w\right\|_{V_1'(\Omega)}$. We agree to make silent use of this extension: we write $w$ to denote $w'$ and use $\left\|\cdot\right\|_{(H^{1}(\Omega))'}$ in place of $\left\|\cdot\right\|_{V_1'(\Omega)}$ on $V_1'(\Omega)$.
\end{remark}
By Remark \ref{R:extendbyzero}, the dual $\Tr_i^\ast:\Hm\to V_1'(\Omega)$ of $\Tr_i$ is seen to be characterized by 
\begin{equation}\label{DefDualTraceOp}
\langle g, \TTr_i v\rangle_{\Hm, \B(\del \Omega)}=\langle \Tr_i^\ast g, v\rangle_{(H^{1}(\Omega))',H^1(\Omega)}, \quad v\in H^1(\Omega), \quad  g\in \Hm.
 \end{equation}

\begin{corollary}\label{CorDualTraceOp}
Let $\Omega\subset \R^n$ be $H^1$-admissible. Then the operator $\Tr_i^\ast:\Hm\to V_1'(\Omega)$ is an isometry, $\left\|\Tr_i^\ast g\right\|_{(H^{1}(\Omega))'}=\left\|g\right\|_{\Hm}$, $g\in \Hm$, and
onto.
\end{corollary}

Now let $\Omega\subset \R^n$ be an $\dot{H}^1$-admissible domain. The space
\begin{equation}\label{VlOr}
\dot{V}_{0}(\Omega):=\bigg\{u\in \dot{H}^1(\Omega)\ \bigg|\  \text{$\int_\Omega \nabla u\cdot\nabla v~\dx=0$ for all $v\in C_c^\infty(\Omega)$}\bigg\}
\end{equation}
is a closed subspace of $\dot{H}^1(\Omega)$.

Given $f\in \Hpo$, we call an element $u$ of $\dot{H}^1(\Omega)$ a \emph{weak solution in the $\dot{H}^1$-sense} of the Dirichlet problem
\begin{equation}\label{DOo}
\begin{cases}
-\Delta u&=0\quad \text{in $\Omega$}\\
u|_{\partial\Omega}&=f
\end{cases}
\end{equation}
if $\dTTr_i u=f$ and $\left\langle u,v\right\rangle_{\dot{H}^1(\Omega)}=0$ for all $v\in C_c^\infty(\Omega)$. For any $f\in \Hpo$ the Dirichlet problem \eqref{DOo} has a unique weak solution $u^f$ in the $\dot{H}^1$-sense; it is an element of the space $\dot{V}_0(\Omega)$ of all elements of $\dot{H}^1(\Omega)$ \emph{harmonic} on $\Omega$.

We write $\dTr_i:=\dTTr_i|_{\dot{V}_0(\Omega)}$ for the restriction of $\dTTr_i$ to $\dot{V}_0(\Omega)$. The following counterpart of Theorem \ref{Trisom} holds; a proof of (i) is given in \ref{PSecBoundCase}.

\begin{theorem}\label{Tr0isom}
Let $\Omega\subset \R^n$ be an $\dot{H}^1$-admissible domain. Then the following assertions hold:
\begin{enumerate}
\item[(i)] The kernel $\ker \dTTr_i|_{\dot{H}^1(\Omega)}$ is the orthogonal complement $(\dot{V}_0(\Omega))^\bot$ of the space $\dot{V}_0(\Omega)$ in $\dot{H}^1(\Omega)$.
\item[(ii)]  Endowed with the norm
$$\left\|f\right\|_{\Hpo}:=\min\{ \left\|v\right\|_{\dot{H}^1(\Omega)}\big|\ \text{$v\in \dot{H}^1(\Omega)$ and $\dTTr_i\:v=f$}\},$$
the space $\Hpo$ is a Hilbert space. 
\item[(iii)] With respect to $\left\|\cdot\right\|_{\Hpo}$, the operator $\dTTr_i$ is bounded with operator norm one. 
Its restriction $\dTr_i:\dot{V}_0(\Omega)\to \Hpo$ to $\dot{V}_0(\Omega)$ is an isometry and~onto.
\end{enumerate}
\end{theorem}

\begin{remark}
By Theorem \ref{Tr0isom} the harmonic extension operator $\dTr_i^{-1}$ is an isometry. For any $f\in \dot{\mathcal{B}}(\partial\Omega)$ we have 
$u^f=\dTr_i^{-1}f\in \dot{V}_0(\Omega)$. 
\end{remark}

We write $(\dot{H}^{1}(\Omega))'$, $\dot{V}_0'(\Omega)$ and $\Hmo$ for the dual spaces of $\dot{H}^1(\Omega)$, $\dot{V}_0(\Omega)$ and $\Hpo$ respectively. With a similar agreement as in Remark \ref{R:extendbyzero}, the dual $\dTr_i^\ast:\Hmo\to \dot{V}_0'(\Omega)$ of $\dTr_i$ is now seen to be  characterized by 
\[\langle g, \dTTr_i v\rangle_{\Hmo,\Hpo}=\langle \dTr_i^\ast g, v\rangle_{(\dot{H}^{1}(\Omega))',\dot{H}^1(\Omega)}, \quad v\in \dot{H}^1(\Omega), \quad  g\in \Hmo.\]

\begin{corollary}\label{ThIsomTrAdj}
Let $\Omega\subset \R^n$ be $\dot{H}^1$-admissible. Then the operator $\dTr_i^\ast:\Hmo\to \dot{V}_0'(\Omega)$ is an isometry, $\big\|\dTr_i^\ast g\big\|_{(\dot{H}^{1}(\Omega))'}=\left\|g\right\|_{\Hmo}$, $g\in \Hmo$, and
onto. 
\end{corollary}

\begin{remark}\label{R:representative}\mbox{}
\begin{enumerate}
\item[(i)] If $\Omega$ is $H^1$-admissible and bounded, and $f\in \Hp$, then $w\in H^1(\Omega)$ is called a weak solution to \eqref{DOo} if  $\TTr_i w=f$ and $\left\langle w,v\right\rangle_{\dot{H}^1(\Omega)}=0$ for all $v\in C_c^\infty(\Omega)$. It is well known that for any $f\in \Hp$, there is a unique weak solution $w^f$ in this sense of \eqref{DOo} and that $w^f$ is an element of the space $V_0(\Omega)$ of all $u\in H^1(\Omega)$ with $\left\langle u,v\right\rangle_{\dot{H}^1(\Omega)}=0$ for all $v\in H^1_0(\Omega)$. 
\item[(ii)] If $\Omega$ is $H^1$-admissible, $\dot{H}^1$-admissible and bounded, and $f\in \Hp$, then the equivalence class $[f]$ of $f$ modulo constants is in $\Hpo$, cf. Remark \ref{R:trace} (ii), and the unique weak solution $u^{[f]} \in \dot{H}^1(\Omega)$ in the $\dot{H}^1$-sense of \eqref{DOo} with this class $[f]$ in place of $f$ contains exactly one representative $w\in u^{[f]}$ modulo constants such that $\widetilde{w}|_{\partial\Omega}=f$ q.e. This particular representative $w$ is exactly the $w^f$ from (i).
\item[(iii)] In \cite{SIMADER-SOHR-1996} a different type of function spaces based on \cite{DENY-LIONS-1954} was used in order to handle Dirichlet boundary conditions for unbounded domains without losing constants. However, to avoid technicalities in later sections, we decided to accept a loss of constants. 
\end{enumerate}
\end{remark}

\begin{remark}
Recall Remark \ref{R:multiplecomponents}. Consider the more general situation where $\Omega=\Omega_1\cup...\cup\Omega_N$ with $\dot{H}^1$-admissible $\Omega_j$ having mutually disjoint closures. Then elements of $\dot{V}_0(\Omega)$, defined as in  \eqref{VlOr}, are classes modulo locally constant functions. Given a class $f\in \dot{\mathcal{B}}(\partial\Omega)$ modulo locally constant functions as in Remark \ref{R:multiplecomponents}, the Dirichlet problem (\ref{DOo}) has again a unique weak solution $u^f$ in the $\dot{H}^1$-sense, and clearly $u^f\in \dot{V}_0(\Omega)$.

If all $\Omega_j$ are bounded and $f\in \mathcal{B}(\partial\Omega)$, then there is a unique weak solution $w^f$ as in Remark \ref{R:representative} (i). If $[f]$ denotes the class of $f$ in $\dot{\mathcal{B}}(\partial\Omega)$, then $w^f$ is the unique representative modulo locally constant functions in $u^{[f]}$ that coincides with $f$ q.e. on all $\partial\Omega_j$.
\end{remark}

\subsection{Neumann solutions}\label{Subsec:N-Sol}

Suppose $\Omega\subset \R^n$ is an $H^1$-admissible domain. Given $g\in \Hm$, we call $u\in H^1(\Omega)$ a \emph{weak solution} of the Neumann problem 
\begin{equation}\label{SO}
\begin{cases}
-\Delta u+u &= 0\quad \text{in $\Omega$}\\
\qquad\frac{\partial u}{\partial\nu}|_{\bord} &= g
\end{cases}
\end{equation}
if for all $v\in H^1(\Omega)$, we have
\begin{equation}\label{Neumansol}
\langle u,v\rangle_{H^1(\Omega)}=\langle g, \TTr_i v\rangle_{\Hm,\Hp}.
\end{equation} 
The symbol $\frac{\partial u}{\partial\nu}$ in the formal problem \eqref{SO} stands for the (interior) normal derivative; in the context of weak solutions it is implicitly made rigorous by \eqref{Neumansol}. It is well known and easily seen from the Riesz representation theorem that for any $g\in \Hm$ the Neumann problem \eqref{SO} has a unique weak solution $u_g$; it is an element of $V_1(\Omega)$. We write $\mathscr{N}_1:\Hm\to V_1(\Omega)$, $\mathscr{N}_1g:=u_g$ for the linear operator mapping a given element $g$ of $\Hm$ to the unique weak solution  of \eqref{SO}. By \eqref{DefDualTraceOp} we have 
\[\langle \mathscr{N}_1g,v\rangle_{H^1(\Omega)}=\langle \Tr_i^\ast g, v\rangle_{(H^{1}(\Omega))',H^1(\Omega)},\quad v\in H^1(\Omega),\]
hence $\|\mathscr{N}_1g\|_{H^1(\Omega)}=\|\Tr_i^* g\|_{(H^{1}(\Omega))'}= \|g\|_{\Hm}$, $g\in \Hm$, that is, $\mathscr{N}_1$ is an isometry.

\begin{corollary}
Let $\Omega$ be $H^1$-admissible. The linear operator $\Tr_i\circ\,\Nscr_1: \Hm\to \Hp$ is an isometry and onto. It satisfies, for $g,h\in\Hm$, 
\[\left\langle g, \Tr_i\circ\,\Nscr_1 h\right\rangle_{\Hm,\,\Hp}=\left\langle g, h\right\rangle_{\Hm} = \left\langle h, \Tr_i\circ\,\Nscr_1 g\right\rangle_{\Hm,\,\Hp}.\] 
\end{corollary}

\begin{proof}
 For the first statement, the polarization identity gives, for all $g,h\in\Hm$,
\begin{equation*}
\langle g,h\rangle_{\Hm}=\langle\Nscr_1 g,\Nscr_1 h\rangle_{H^1(\Omega)},
\end{equation*}
and the result follows by \eqref{Neumansol}.
The second statement follows since for all $g,h\in \Hm$ we have 
\begin{multline}
\langle g, \Tr_i\circ\,\Nscr_1 h\rangle_{\Hm,\,\Hp} = \langle g, \Tr_i u_h\rangle_{\Hm,\,\Hp} \notag\\
     = \langle\Tr_i^*g, u_h\rangle_{(H^{1}(\Omega))',H^1(\Omega)} = \langle u_g,u_h\rangle_{H^1(\Omega)}.
\end{multline}
\end{proof}

Now suppose that $\Omega$ is $\dot{H}^1$-admissible. Given $g\in \Hmo$, we call $u\in \dot{H}^1(\Omega)$ a \emph{weak solution in the $\dot{H}^1$-sense} of the Neumann problem 
\begin{equation}\label{SOo}
\begin{cases}
-\Delta u &= 0\quad \text{in $\Omega$}\\
\frac{\partial u}{\partial\nu}|_{\bord}&= g
\end{cases}
\end{equation}
if for all $v\in \dot{H}^1(\Omega)$ we have
\begin{equation}\label{Neumansolo}
\langle u,v\rangle_{\dot{H}^1(\Omega)}=\langle g, \dTTr_i v\rangle_{\Hmo,\,\Hpo}.
\end{equation} 
For any $g\in \Hmo$ the Neumann problem \eqref{SOo} has a unique weak solution $u_g$ in the $\dot{H}^1$-sense, and $u_g$ is an element of $\dot{V}_0(\Omega)$ defined in \eqref{VlOr}. We write $$\dot{\mathscr{N}}_0:\Hmo\to \dot{V}_0(\Omega), \quad \dot{\mathscr{N}}_0g:=u_g,$$ for the linear operator mapping a given element of $\Hmo$ to $u_g$. 
As before, we see that $\dot{\mathscr{N}}_0$ is an isometry: $$\|\dot{\mathscr{N}}_0g\|_{\dot{H}^1(\Omega)}=\|\dTr_i^* g\|_{(\dot{H}^1(\Omega))'}= \|g\|_{\Hmo}, \quad g\in \Hmo.$$

\begin{corollary}
Let $\Omega$ be $\dot{H}^1$-admissible. Then the linear operator $\Tr_i\circ\,\NOt_0: \Hmo\to \Hpo$ is an isometry and onto. It satisfies, for $g,h\in\Hmo$, 
\[\big\langle g, \dTr_i\circ\,\NOt_0 h\big\rangle_{\Hmo,\,\Hpo}=\left\langle g, h\right\rangle_{\Hmo} = \big\langle h, \dTr_i\circ\,\NOt_0 g\big\rangle_{\Hmo,\,\Hpo}.\] 
\end{corollary}

\subsection{Normal derivatives}\label{SS:normal}

Abstract normal derivatives have been defined and studied by various authors in different contexts, see for instance \cite[p. 218]{CONSTANTINESCU-CORNEA-1963}, \cite[Section 3.2]{LEJAN-1978} and \cite{LANCIA-2002}. We formulate a variant of those definitions which suits our purposes. 

Suppose $\Omega\subset\R^n$ is $H^1$-admissible. Let 
\begin{equation}\label{H1D}
H^1_\Delta(\Omega):=\big\{u\in H^1(\Omega)\;\big|\;\Delta u\in L^2(\Omega)\big\};
\end{equation}
here $\Delta u$ is understood in distributional sense. Clearly, $V_1(\Omega)\subset H^1_\Delta(\Omega)$. Given $u\in H^1_\Delta(\Omega)$, there is a unique element $g$ of $\Hm$ such that 
\begin{equation}\label{EqGreenInt}
\langle g,\TTr_i v\rangle_{\Hm,\,\Hp}= \int_\Omega{(\Delta u)v\,\dx}+\int_\Omega \nabla u\cdot\nabla v\,\dx,\quad v\in H^1(\Omega);
\end{equation}
note that the right-hand side defines a bounded linear functional on $H^1(\Omega)$ and recall that  $\TTr_i$ is surjective. We call this element $g$ the \emph{weak interior normal derivative} of $u$ (with respect to $\Omega$) and denote it by $\frac{\partial_i u}{\partial\nu} :=g$. The operator $\frac{\partial_i}{\partial\nu}:H^1_\Delta(\Omega)\to \Hm$ is linear and bounded in the sense that 
\begin{equation}\label{E:boundnormalder}
\big\| \frac{\partial_i u}{\partial\nu}\big\|_{\Hm}\leq \|u\|_{H^1(\Omega)}+\|\Delta u\|_{L^2(\Omega)}.
\end{equation}

We write $\del_{\nu,i}:=\left.\frac{\partial_i}{\partial\nu}\right|_{V_1(\Omega)}$ for the restriction of $\frac{\partial_i}{\partial\nu}$ to $V_1(\Omega)$.

\begin{corollary}\label{C:Di}
Let $\Omega$ be $H^1$-admissible. Then the following assertions hold:
\begin{enumerate}
\item[(i)] Both $\Nscr_1:\Hm\to V_1(\Omega)$ and the operator 
$\del_{\nu,i}: V_1(\Omega)\to \Hm$ are isometries and onto, and  $\del_{\nu,i}=\Nscr^{-1}_1$.
\item[(ii)] For any $u,v\in V_1(\Omega)$ we have
\begin{equation}\label{E:secondGreen}
\big\langle \partial_{\nu,i} u,\Tr_i v\big\rangle_{\Hm,\Hp}=\left\langle u,v\right\rangle_{H^1(\Omega)}=\big\langle  \partial_{\nu,i} v,\Tr_i u\big\rangle_{\Hm,\Hp}.
\end{equation}
\item[(iii)] The dual $\big(\del_{\nu,i}\big)^\ast:\Hp\to V_1'(\Omega)$ of $\del_{\nu,i}$ is an isometry and onto.
\end{enumerate}
\end{corollary}

\begin{proof}
Statement (i) follows using the surjectivity of $\Tr_i$ observed in Theorem \ref{Trisom} (iii) and a comparison of \eqref{Neumansol} and \eqref{EqGreenInt}. Statement (ii) is a special case of (\ref{EqGreenInt}), and (iii) is a consequence of (i).
\end{proof}

Now suppose that $\Omega\subset\R^n$ is $\dot{H}^1$-admissible. We then consider the space 
\[\dot{H}^1_{\Delta}(\Omega):=\{u\in \dot{H}^1(\Omega)\;|\;\Delta u\in L^2(\Omega)\},\]
which contains $\dot{V}_0(\Omega)$. Given $u\in \dot{H}^1_{\Delta}(\Omega)$, there is a unique element $g$ of $\Hmo$ such that 
\begin{equation}\label{EqGreenInto}
\langle g,\dTTr_i v\rangle_{\Hmo,\Hpo}= \int_{\Omega}{(\Delta u)v\,\dx}+\int_{\Omega}{\nabla u\cdot\nabla v\,\dx},\quad v\in \dot{H}^1(\Omega),
\end{equation}
and we call also $ \frac{\dot{\partial}_i u}{\partial\nu}:=g$ the \emph{weak interior normal derivative of $u$ in the $\dot{H}^1$-sense} (with respect to $\Omega$). The operator $ \frac{\dot{\partial}_i }{\partial\nu}:\dot{H}^1_{\Delta}(\Omega)\to \Hmo$ is linear and 
\begin{equation}\label{E:boundnormaldero}
\big\| \frac{\dot{\partial}_i u}{\partial\nu}\big\|_{\Hm}\leq \|u\|_{\dot{H}^1(\Omega)}+\|\Delta u\|_{L^2(\Omega)}.
\end{equation}

We write $\dot{\del}_{\nu,i}:=\frac{\dot{\partial}_i}{\partial\nu}\big|_{\dot{V}_0(\Omega)}$ for the restriction of $\frac{\dot{\partial}_i}{\partial\nu}$ to $\dot{V}_0(\Omega)$.

\begin{remark}\label{R:BprimeB}
Suppose that $\Omega$ is $H^1$-admissible, $\dot{H}^1$-admissible and bounded. Then by Remark \ref{R:trace} (ii), the spaces $\Hp$ and $\Hpo\oplus \R$ can be identified. Each bounded linear functional $g\in\Hmo$ on $\Hpo$ induces a linear functional on $\Hp$ through extension by zero on $\R$. In this sense, an equality of two elements of $\Hmo$, such as the Neumann boundary condition
$ \dot{\partial}_{\nu,i}(\NOt_0g)=g$ for $g\in \Hmo$, may be seen as an equality of two linear functionals on $\Hp$. 
\end{remark}

\begin{corollary}\label{C:Dio}
Let $\Omega$ be $\dot{H}^1$-admissible. Then:
\begin{enumerate}
\item[(i)] Both $\NOt_0:\Hpo\to \dot{V}_0(\Omega)$ and the operator 
$\dot{\partial}_{\nu,i}:\dot{V}_0(\Omega)\to \Hmo$ are isometries and onto, and  $\dot{\partial}_{\nu,i}=\NOt^{-1}_0$. 
\item[(ii)] For any $u,v\in \dot{V}_0(\Omega)$ we have
\begin{equation}\label{E:secondGreeno}
\big\langle  \dot{\partial}_{\nu,i}u,\dTr_i v\big\rangle_{\Hmo,\Hpo}=\left\langle u,v\right\rangle_{\dot{H}^1(\Omega)}=\big\langle  \dot{\partial}_{\nu,i}v,\dTr_i u\big\rangle_{\Hmo,\Hpo}.
\end{equation}
\item[(iii)] The dual $\big(\dot{\partial}_{\nu,i}\big)^\ast:\Hpo\to \dot{V}_0'(\Omega)$ of $\dot{\partial}_{\nu,i}$ is an isometry and onto.
\end{enumerate}
\end{corollary}

\begin{remark}
Weak normal derivatives can also be used to characterize weak solutions of \eqref{DO}: an element $u$ of $H^1(\Omega)$ is a weak solution to \eqref{DO} if and only if $\left\langle u,v\right\rangle_{H^1(\Omega)}=0$ for all $v\in C_c^\infty(\Omega)$ and $\left\langle v,u\right\rangle_{H^1(\Omega)}=\left\langle  \partial_{\nu,i} v,f\right\rangle_{\Hm, \Hp}$ for all $v\in V_1(\Omega)$. Similarly for \eqref{DOo}.
\end{remark}

\begin{remark}\label{R:normalderLip}
Recall Remark \ref{R:boundary_spaces}. Suppose that $\Omega$ is a bounded Lipschitz domain. Then for any $u\in H^1_\Delta(\Omega)$ its weak interior normal derivative $\frac{\partial_i u}{\partial\nu}$ is defined as the unique element of $H^{-1/2}(\partial\Omega)$ satisfying 
(\ref{EqGreenInt}) in place of $g$; in this case the dual pairing $\langle \cdot,\cdot\rangle_{\Hm,\,\Hp}$ is the dual pairing 
 $\left\langle\cdot,\cdot\right\rangle_{H^{-1/2}(\Omega),H^{1/2}(\Omega)}$. See for instance \cite[Section VII.1, Lemma 1]{DUTRAY-LIONS1988}. The situation is similar for $u\in \dot{H}^1_{\Delta}(\Omega)$; in this case $\frac{\dot{\partial}_i u}{\partial\nu}\in \dot{H}^{-1/2}(\partial\Omega)$.
 
If $u\in H^2(\Omega)$, cf. \cite[p. 2]{GIRAULT-1986}, then the distributional partial derivatives $\frac{\partial u}{\partial x_i}$ are in $H^1(\Omega)$ and
\[\frac{\partial_i u}{\partial\nu}=\sum_{i=1}^n\TTr_i\big(\frac{\partial u}{\partial x_i}\big)\nu_i\] 
is an element of $L^2(\partial\Omega,\sigma)$; here $\sigma$ is the surface measure on $\partial\Omega$ and 
$\nu=(\nu_1,...,\nu_n)$ is the $\sigma$-a.e. defined outward unit normal on $\partial\Omega$. See \cite[Section I.1, Lemma 1.4]{GIRAULT-1986}.
\end{remark}

\section{Poincar\'e-Steklov operators for admissible domains}\label{Sec:P-S}

We provide some comments on Poincar\'e-Steklov operators for bounded admissible domains. 

Let $\Omega\subset \R^n$ be a bounded $H^1$-admissible domain. We write $\Delta_D$ for the self-adjoint Dirichlet Laplacian on $L^2(\Omega)$ and denote its spectrum by $\sigma(\Delta_D)$. Recall that $\sigma(\Delta_D)\subset (-\infty,0)$ is pure point with eigenvalues accumulating at minus infinity. 

Given $k\in\R$ and $f\in \Hp$, we call $u\in H^1(\Omega)$ a \emph{weak solution} of the Dirichlet problem 
\begin{equation}\label{DOk}
\begin{cases}
-\Delta u+ku&=0\quad \text{in $\Omega$}\\
\qquad u|_{\partial\Omega}&=f
\end{cases}
\end{equation}
if $\TTr_i u=f$ and $\left\langle u,v\right\rangle_{\dot{H}^1(\Omega)}+k\left\langle u,v\right\rangle_{L^2(\Omega)}=0$ for all $v\in C_c^\infty(\Omega)$. This corresponds to \eqref{DO} with $-\Delta u+k u$ in place of $-\Delta u+u$; problem \eqref{DO} is the special case for $k=1$.
 
If $k\in \R\backslash \sigma(\Delta_D)$, then for any $f\in \Hp$ there is a unique weak solution $u_f\in H_\Delta(\Omega)$ of \eqref{DOk}. For such $k$ one can define a linear operator $\Ascr_k: \Hp \to \Hm$ by 
\[\Ascr_k f:=\ddny{u_f}{i}.\]
This operator is called the \emph{Poincar\'e-Steklov} (or \emph{Dirichlet-to-Neumann}) \emph{operator} associated with $(\Delta-k)$ on $\Omega$. See for instance~\cite{FRIEDLANDER-1991,ARENDT-2012,ARFI-2019,ROZANOVA-PIERRAT-2021} for studies of Poincar\'e-Steklov operators under more restrictive assumptions on~$\Omega$.

\begin{lemma}\label{L:PS}
Let $\Omega$ be a bounded $H^1$-admissible domain. Then the following assertions hold:
\begin{enumerate}
\item[(i)] For any $k\in \R\backslash \sigma(\Delta_D)$, the Poincar\'e-Steklov operator $\Ascr_k:\Hp\to \Hm$ is a bounded linear operator and coincides with its adjoint. It is injective if and only if $k$ is not an eigenvalue of the self-adjoint Neumann Laplacian for $\Omega$.
\item[(ii)] The Poincar\'e-Steklov operator $\Ascr_1:\Hp\to \Hm$ satisfies $\Ascr_1=\partial_{\nu,i} \circ(\Tr_i)^{-1}$. It is an isometry with inverse $\Ascr_1^{-1}:\Hm\to \Hp$ given by $\Ascr_1^{-1}=\Tr_i\circ \Nscr_1$.
\end{enumerate}
\end{lemma}

\begin{proof}
By \eqref{E:boundnormalder} and \eqref{DOk} we have 
\[\left\|\Ascr_kf\right\|_{\Hm}\leq (1+k)\left\|u_f\right\|_{H^1(\Omega)}=(1+k)\left\|f\right\|_{\Hp},\quad f\in \Hp.\] 
From \eqref{EqGreenInt} it is easily seen that 
\[\langle \Ascr_k f_1,  f_2 \rangle_{\Hm,\Hp}= \langle  \Ascr_kf_2,f_1\rangle_{\Hp,\Hm},\]
the special case for $k=1$ was stated in \eqref{E:secondGreen}. The statement on injectivity is clear. Item (ii) follows using Theorem \ref{Trisom} (iii) and Corollary \ref{C:Di} (i). 
\end{proof}

Since we are interested in isometries as in Lemma \ref{L:PS} (ii), which require the use of $k$-dependent norms on $\Hp$ and $\Hpo$, we concentrate on the special cases $k=0,1$ to keep notation simple.

Now suppose that  $\Omega\subset \R^n$ is bounded and $\dot{H}^1$-admissible. Then the linear operator $\dot{\Ascr}_0:\Hpo\to \Hmo$, defined as 
\[\dot{\Ascr}_0:= \dot{\partial}_{\nu,i}\circ (\dTr_i)^{-1},\]
is called the \emph{Poincar\'e-Steklov} (or \emph{Dirichlet-to-Neumann}) \emph{operator in the $\dot{H}^1$-sense} associated with $\Delta$ on $\Omega$. Theorem \ref{Tr0isom} (iii) and Corollary \ref{C:Dio} (i) give the following analog of Lemma~\ref{L:PS}.

\begin{lemma}\label{L:PSo}
Let $\Omega$ be a bounded $\dot{H}^1$-admissible domain. Then the linear operator $\dot{\Ascr}_0:\Hpo\to \Hmo$ coincides with its adjoint. Moreover, it is an isometry with inverse $\dot{\Ascr}_0^{-1}:\Hmo\to \Hpo$ given by $\dot{\Ascr}_0^{-1}=\dTr_i\circ \dot{\Nscr}_0$.
\end{lemma}

Writing $\iota$ to denote the Riesz isomorphism from a Hilbert space to its dual (regardless of that space), we obtain the following commutative diagram~\eqref{Fig-EQ-Diagram}:

\begin{equation}
\label{Fig-EQ-Diagram}
\begin{tikzcd}[row sep=huge, column sep=huge]
V_1(\Omega)=H_0^1(\Omega)^\perp
\arrow[r, shift left=.5ex, "\Tr_i" ] %
\arrow[rdd, shift left=0.5ex, "\frac{\del_i}{\del \nu}"]
\arrow[dd, bend right, shift left=.5ex, "\iota" ]
& \TTr_i(H^1(\Omega))=\Hp 
\arrow[dd, bend left, shift left=.5ex, "\Ascr_1=\iota" ] 
\arrow[l, shift left=.5ex, "\Tr_i^{-1}" ] 
&
\\
&&
\\
V_1'(\Omega)=(H_0^1(\Omega)^\perp)' 
\arrow[r, shift left=.5ex, " (\Tr_i^*)^{-1}" ]
\arrow[uu, bend left, shift left=.5ex, "\iota^{-1}"] 
& 
(\TTr_i(H^1(\Omega)))'=\Hm
\arrow[l, shift left=.5ex, " \Tr_i^*"]
\arrow[uu, bend right, shift left=.5ex, "\Ascr_1^{-1}=\iota^{-1}" ] 
\arrow[luu, shift left=0.5ex, "(\frac{\del_i}{\del \nu})^{-1}"]
& 
\end{tikzcd}
\end{equation}
Note that for any $f_1,f_2\in \Hp$, we have indeed 
\begin{multline}
\left\langle\iota (f_1),f_2\right\rangle_{\Hm,\Hp}=\left\langle f_1,f_2\right\rangle_{\Hp}\notag\\
=\left\langle \Tr_i^{-1}f_1,\Tr_i^{-1}f_2\right\rangle_{H^1(\Omega)}=\big\langle 
 \partial_{\nu,i} \circ(\Tr_i)^{-1}f_1,f_2\big\rangle_{\Hm,\Hp}.
\end{multline}

Obvious adjustments give an analogous commutative diagram involving the Poincar\'e-Steklov operator in the $\dot{H}^1$-sense.

\section{Transmission problems for admissible domains}\label{Sec:TrProb}

We introduce layer potential operators on two-sided admissible domains through weak well-posedness results for transmission problems of the form 
\begin{equation}\label{EqTrProblem}
\begin{cases}
(-\Delta+k) u &= 0 \quad \text{on $\R^n\backslash \partial\Omega$}\\
u_i|_{\partial\Omega}-u_e|_{\partial\Omega}&=f\\
\frac{\partial_iu_i}{\partial\nu}|_{\partial\Omega}-\frac{\partial_eu_e}{\partial \nu}|_{\partial\Omega}&= g.
\end{cases}
\end{equation}
Here $\Omega\subset \R^n$ is a bounded domain, $u_i$ and $u_e$ are the restrictions of the prospective solution $u$ to $\Omega$ and $\R^n\backslash\overline{\Omega}$ respectively, $f$ and $g$ are given data on $\partial\Omega$ and $k=0,1$. Similarly as before, the restrictions $u_i|_{\partial\Omega}$ and $u_e|_{\partial\Omega}$ and the normal derivatives $\frac{\partial_iu_i}{\partial\nu}|_{\partial\Omega}$ and $\frac{\partial_eu_e}{\partial \nu}|_{\partial\Omega}$ will be made precise using traces and abstract normal derivatives, see Subsections \ref{Subsec:DLP}, \ref{Subsec:SLP} and \ref{Subsec:Superposition}.

First, we describe our setup for \eqref{EqTrProblem} and then proceed to well-posedness results, the definitions of layer potentials and statements on some of their properties.

\subsection{Two-sided admissible domains and jumps}\label{Subsec:TSA-Jumps}

We call $\Omega\subset \R^n$ a \emph{two-sided $H^1$-admissible domain} if both $\Omega$ and $\R^n\backslash \overline{\Omega}$ are $H^1$-extension domains, $\partial\Omega=\partial(\R^n\backslash \overline{\Omega})$ and the Lebesgue measure of $\partial \Omega$ is zero. Since $\partial\Omega$ has topological dimension at least $n-1$, \cite[Theorem IV.4]{HUREWICZWALLMAN-1941}, its Hausdorff dimension is at least $n-1$, \cite[Theorem 8.14]{HEINONEN-2001}, and therefore its capacity is positive, \cite[Theorem 5.1.13]{ADAMS-1996}. As before, we do not require $\partial\Omega$ to be Lipschitz or to carry any measure. Moreover, $\partial\Omega$ may consist of parts having different Hausdorff dimensions. In principle it is also possible to extend the results of this and later sections to cases where $\Omega$ or $\R^n\backslash \overline{\Omega}$ has multiple connected components, but we will not address such extensions.

Suppose that $\Omega\subset \R^n$ is a two-sided $H^1$-admissible domain. We discuss \eqref{EqTrProblem} with $k=1$ in terms of the space $H^1(\R^n\backslash \partial\Omega)$. It admits the (natural) orthogonal decomposition 
$H^1(\R^n\backslash \partial\Omega)=H^1(\Omega)\oplus H^1(\R^n\backslash \overline{\Omega})$, 
so that every $u\in H^1(\R^n\backslash \partial\Omega)$ can be written as $u=u_i+u_e$ with uniquely determined $u_i\in H^1(\Omega)$ and $u_e\in H^1(\R^n\backslash \overline{\Omega})$.
We denote the closure of $C_c^\infty(\R^n\backslash\partial\Omega)$ in $H^1(\R^n\backslash \partial\Omega)$ by $H_0^1(\R^n\backslash \partial\Omega)$ and write $V_1(\R^n\backslash \partial\Omega)$ for its orthogonal complement, which is the space of functions that are $1$-harmonic in $\R^n\backslash \partial\Omega$. This gives a second orthogonal decomposition,
\begin{equation}\label{E:orthotwosided}
H^1(\R^n\backslash \partial\Omega)=H_0^1(\R^n\backslash \partial\Omega)\oplus V_1(\R^n\backslash \partial\Omega).
\end{equation}

To distinguish between the two operators, we now write $\TTr_i:H^1(\Omega)\to \Hp$ to denote the boundary trace operator for $\Omega$ and $\TTr_e:H^1(\R^n\backslash \overline{\Omega})\to \Hp$ to denote the boundary trace operator for $\R^n\backslash \overline{\Omega}$; by construction both map onto $\Hp$. 
We refer to $\TTr_i$ as the \emph{interior} trace operator with respect to $\Omega$ and to $\TTr_e$ as the \emph{exterior} one.
In the context of transmission problems we endow $\Hp$ with the Hilbert space norm 
\begin{equation}\label{normTransmi}
\|f\|_{\Hp,t}:=\big(\|f\|_{\Hp, i}^2+\|f\|_{\Hp, e}^2\big)^{1/2},
\end{equation}
where $\|\cdot\|_{\Hp, i}$ denotes the norm defined in \eqref{normTri} with respect to $\Omega$ and $\|\cdot\|_{\Hp,e}$ denotes the norm defined similarly but with $\R^n\backslash \overline{\Omega}$ in place of $\Omega$. 
Since both $\Omega$ and $\R^n\backslash \overline{\Omega}$ are $H^1$-extension domains, the norms $\|\cdot\|_{\Hp, i}$, $\|\cdot\|_{\Hp, e}$ and $\|\cdot\|_{\Hp,t}$ are all equivalent. As a consequence, we may view  $\TTr_i$ and $\TTr_e$ as bounded linear operators from 
$H^1(\R^n\backslash \partial\Omega)$ onto $\Hp$, no matter which norm is used. Only $\|\cdot\|_{\Hp,t}$ will be used in the sequel; we agree to denote it by $\|\cdot\|_{\Hp}$ again.

Given $u\in H^1(\R^n\backslash \partial\Omega)$, we write 
\[\llbracket\TTr u\rrbracket:=\TTr_iu-\TTr_eu\]
for its \emph{jump in trace across $\partial\Omega$}. The map $u\mapsto \llbracket\TTr u\rrbracket$ defines a bounded linear operator $\llbracket\TTr \rrbracket:H^1(\R^n\backslash \partial\Omega)\to \Hp$ and is onto.
In the same spirit, we write $\llbracket\Tr\rrbracket:=\Tr_i-\Tr_e:V_1(\nbord)\to\Hp$, also bounded, linear and onto.

If $\Delta u\in L^2(\Omega)$, then the \emph{interior} normal derivative $\ddny ui$ of $u$ with respect to $\Omega$ is as defined in \eqref{EqGreenInt}.
If $\Delta u\in L^2(\R^n\backslash \overline{\Omega})$, then we define the \emph{exterior} normal derivative $\frac{\partial_e u}{\partial\nu}$ of $u$ with respect to $\Omega$ as minus the interior normal derivative of $u$ with respect to $\R^n\backslash\overline{\Omega}$, that is, the unique element $g\in\Hm$ such that, for all $v\in H^1(\R^n\backslash\overline{\Omega})$,
\begin{equation}\label{E:extnormalder}
\langle g,\TTr_e v\rangle_{\Hm,\,\Hp}= -\int_{\R^n\backslash \overline{\Omega}}{(\Delta u)v\,\dx}-\int_{\R^n\backslash\overline{\Omega}}{\nabla u\cdot\nabla v\,\dx}.
\end{equation}
Now let
\[H^1_\Delta(\R^n\backslash \partial\Omega) :=\{u\in H^1(\R^n\backslash \partial\Omega)\ |\ \Delta u\in L^2(\R^n\backslash \partial\Omega)\}.\]
Clearly this space contains $V_1(\R^n\backslash \partial\Omega)$. For $u \in H^1_\Delta(\R^n\backslash \partial\Omega)$ we write 
\[\Big\llbracket \ddn u\Big\rrbracket:=\ddny{u}{i}-\ddny{u}{e}\]
for the \emph{jump of its normal derivative across $\partial\Omega$}. This defines a linear operator $u\mapsto \left\llbracket \ddn u\right\rrbracket$ from $H^1_{\Delta}(\R^n\backslash \partial\Omega)$ onto $\Hm$, bounded in the sense that  
\begin{equation}\label{E:normalderjumpbd}
\big\|\Big\llbracket\frac{\partial u}{\partial\nu}\Big\rrbracket\big\|_{\Hm}\leq \|u\|_{H^1(\R^n\backslash\partial\Omega)}+\|\Delta u\|_{L^2(\R^n\backslash\partial\Omega)}.
\end{equation}
We write $\llbracket\partial_\nu\rrbracket:=\partial_{\nu,i}-\partial_{\nu,e}:V_1(\nbord)\to\Hm$ for the restriction of $\left\llbracket \ddn u\right\rrbracket$ to $V_1(\nbord)$. 

We call $\Omega\subset \R^n$ a \emph{two-sided $\dot{H}^1$-admissible domain} if both $\Omega$ and $\R^n\backslash \overline{\Omega}$ are $\dot{H}^1$-extension domains, $\Omega$ is bounded, $\partial\Omega=\partial(\R^n\backslash \overline{\Omega})$ and the Lebesgue measure of $\partial\Omega$ is zero.

Suppose that $\Omega\subset \R^n$ is a two-sided $\dot{H}^1$-admissible domain. We discuss \eqref{EqTrProblem} with $k=0$
using the space $\dot{H}^1(\R^n\backslash\partial\Omega)$, which admits the (natural) orthogonal decomposition $\dot{H}^1(\R^n\backslash\partial\Omega)=\dot{H}^1(\Omega)\oplus \dot{H}^1(\R^n\backslash\overline{\Omega})$. We denote the closure of $C_c^\infty(\R^n\backslash\partial\Omega)$ in $\dot{H}^1(\R^n\backslash \partial\Omega)$ by $\dot{H}_0^1(\R^n\backslash \partial\Omega)$ and its orthogonal complement by $\dot{V}_0(\R^n\backslash \partial\Omega)$.

By $\dTTr_i:\dot{H}^1(\Omega)\to\Hpo$ and $\dTTr_e:\dot{H}^1(\R^n\backslash\overline{\Omega})\to\Hpo$, we denote the interior and exterior trace with respect to $\Omega$, defined in an analogous manner. We endow $\Hpo$ with the Hilbert space norm 
\[\|f\|_{\Hpo,t}:=\big(\|f\|_{\Hpo, i}^2+\|f\|_{\Hpo, e}^2\big)^{1/2}\]
and with the summands defined similarly as before. Only $\|\cdot\|_{\Hpo,t}$ will be used in the sequel; we agree to denote it by $\|\cdot\|_{\Hpo}$ again.

Given $u\in \dot{H}^1(\R^n\backslash \partial\Omega)$, we set 
\[\llbracket\dTTr u\rrbracket:=\dTTr_iu-\dTTr_eu\]
for its jump in trace across $\Omega$ in the $\dot{H}^1$-sense. The map $\llbracket\dTTr \rrbracket: \dot{H}^1(\R^n\backslash\partial\Omega)\to \Hpo$, $u\mapsto \llbracket\dTTr u\rrbracket$, defines a bounded linear operator which is onto. We write $\llbracket\dTr\rrbracket:=\dTr_i-\dTr_e:\dot V_0(\nbord)\to\Hpo$, also bounded, linear and onto.

Let $\dot{H}^1_\Delta(\R^n\backslash \partial\Omega)$ denote the space of all $u\in \dot{H}^1(\R^n\backslash \partial\Omega)$ with $\Delta u\in L^2(\R^n\backslash \partial\Omega)$; it contains $\dot{V}_0(\R^n\backslash \partial\Omega)$. Given an element $u\in \dot{H}^1_\Delta(\R^n\backslash \partial\Omega)$, we write $\frac{\dot{\partial}_i u}{\partial\nu}$ for its interior normal derivative in the $\dot{H}^1$-sense with respect to $\Omega$ and $\frac{\dot{\partial}_e u}{\partial\nu}$ for its exterior, again defined as minus the interior with respect to $\R^n\backslash \overline{\Omega}$. We then write 
\[\Big\llbracket \frac{\dot\partial u}{\partial\nu}\Big\rrbracket:= \frac{\dot{\partial}_i u}{\partial\nu}-\frac{\dot{\partial}_e u}{\partial\nu}\]
for the jump of its normal derivative across $\partial\Omega$ in the $\dot{H}^1$-sense.
This defines a linear operator $u\mapsto \llbracket \frac{\dot\partial u}{\partial\nu}\rrbracket$ from $\dot{H}^1_\Delta(\R^n\backslash \partial\Omega)$ onto $\Hmo$, bounded in the sense that 
\[\big\|\Big\llbracket \frac{\dot\partial u}{\partial\nu}\Big\rrbracket\big\|_{\Hmo}\leq \|u\|_{\dot{H}^1(\R^n\backslash\partial \Omega)}+\|\Delta u\|_{L^2(\R^n\backslash\partial\Omega)}.\]
We write $\llbracket\dot\partial_\nu\rrbracket:=\dot\partial_{\nu,i}-\dot\partial_{\nu,e}:\dot V_0(\nbord)\to\Hmo$.

\begin{remark}\label{R:ellipticreg}{\color{black}
By the smoothness of $1$-harmonic and harmonic functions, any element of $V_1(\R^n\backslash \partial\Omega)$ or $\dot{V}_0(\R^n\backslash \partial\Omega)$ has a representative in $C^\infty(\R^n\backslash \partial\Omega)$.}
\end{remark}

\subsection{Subspaces and orthogonality}\label{Subsec:SubSp-Ortho}

Suppose that $\Omega\subset \R^n$ is a two-sided $H^1$-admissible domain. We write 
\[V_{1,\Scal}(\R^n\backslash \partial\Omega):=\{u\in V_1(\R^n\backslash \partial\Omega)\ |\ \llbracket\Tr u\rrbracket=0\}\]
and 
\[V_{1,\Dcal}(\R^n\backslash \partial\Omega):=\{u\in V_1(\R^n\backslash \partial\Omega)\ |\  \llbracket \partial_{\nu} u\rrbracket=0\}.\]
Recall that $\llbracket\TTr \rrbracket:H^1(\R^n\backslash \partial\Omega)\to \Hp$ and that we use the Hilbert norm \eqref{normTransmi} on $\Hp$.

\begin{lemma}\label{L:spaces} Let $\Omega\subset \R^n$ be two-sided $H^1$-admissible. Then the following statements hold:
\begin{enumerate}
\item[(i)]
Each element of $V_{1,\Scal}(\R^n\backslash \partial\Omega)$ has a unique extension to an element of $H^1(\R^n)$.
In this sense, the space $V_{1,\Scal}(\R^n\backslash \partial\Omega)$ is the orthogonal complement of $H_0^1(\R^n\backslash \partial\Omega)$ in $H^1(\R^n)$. We have $\ker~\llbracket\TTr \rrbracket =H^1(\R^n)$. The linear operator $\TTr:H^1(\R^n)\to \Hp$, defined as $\TTr:=\TTr_i|_{H^1(\R^n)}=\TTr_e|_{H^1(\R^n)}$, is bounded with operator norm one. Its kernel is $H_0^1(\R^n\backslash \partial\Omega)$.
\item[(ii)] The space $V_{1,\Scal}(\R^n\backslash \partial\Omega)$ is a closed subspace of $H^1(\R^n\backslash \partial\Omega)$. The linear operator 
\begin{equation}\label{E:tracesurjective}
\Tr:V_{1,\Scal}(\R^n\backslash \partial\Omega)\to \Hp,
\end{equation}
defined as the restriction $\Tr:=\TTr|_{V_{1,\Scal}(\R^n\backslash \partial\Omega)}$, is an isometry and onto. 
\item[(iii)] The space $V_{1,\Dcal}(\R^n\backslash \partial\Omega)$ is a closed subspace  of $H^1(\R^n\backslash \partial\Omega)$. The linear operator 
\begin{equation}\label{E:normaldersurjective}
\partial_{\nu}: V_{1,\Dcal}(\R^n\backslash \partial\Omega)\to \Hm,
\end{equation}
defined as $\partial_{\nu}:=\partial_{\nu,i} =\partial_{\nu,e}$, is bijective and bounded.
\item[(iv)] The space $V_1(\R^n\backslash \partial\Omega)$ admits the orthogonal decomposition 
\[V_1(\R^n\backslash \partial\Omega)=V_{1,\Scal}(\R^n\backslash \partial\Omega)\oplus V_{1,\Dcal}(\R^n\backslash \partial\Omega).\]
\end{enumerate}
\end{lemma}

\begin{proof}
For any $u\in V_{1,\mathcal{S}}(\mathbb{R}^n\setminus \partial\Omega)$ we can find $w\in H^1(\mathbb{R}^n)$ such that $\widetilde{w}|_{\partial\Omega}=
\TTr_iu=\TTr_eu$ in $\Hp$, $w|_\Omega\in V_1(\Omega)$ and $w|_{\mathbb{R}^n\setminus\overline{\Omega}}\in V_1(\mathbb{R}^n\setminus\overline{\Omega})$. This is clear by the solvability of the Dirichlet problems of type \eqref{DO} on $\Omega$ and $\mathbb{R}^n\setminus \overline{\Omega}$, respectively. By the uniqueness of the weak solutions then $u|_\Omega=w|_\Omega$ in $H^1(\Omega)$ and $u|_{\mathbb{R}^n\setminus\overline{\Omega}}=w|_{\mathbb{R}^n\setminus\overline{\Omega}}$ in $H^1(\mathbb{R}^n\setminus\overline{\Omega})$. Combining, we see that $u$ has a unique extension $w\in H^1(\mathbb{R}^n)$ to $\mathbb{R}^n$ and may therefore itself be seen as an element of $H^1(\mathbb{R}^n)$. This gives $V_{1,\mathcal{S}}(\mathbb{R}^n\setminus \partial\Omega)\subset H^1(\mathbb{R}^n)$. Clearly $H_0^1(\mathbb{R}^n\setminus \partial\Omega)\subset H^1(\mathbb{R}^n)$.

Suppose that $u\in H^1(\mathbb{R}^n)$. Then, by definition, $\TTr_i u=\TTr_e u$. If in addition $\left\langle u,v\right\rangle_{H^1(\mathbb{R}^n)}=0$ for all $v\in H^1_0(\mathbb{R}^n\setminus \partial\Omega)$, then $u\in V_{1}(\mathbb{R}^n\setminus \partial\Omega)$ and, consequently, $u\in V_{1,\mathcal{S}}(\mathbb{R}^n\setminus \partial\Omega)$ by the definition of this space.
If instead $u\in V_{1,\mathcal{S}}(\mathbb{R}^n\setminus \partial\Omega)$, then for all $v\in H^1_0(\mathbb{R}^n\setminus \partial\Omega)$ we have 
\begin{align}
\left\langle u,v\right\rangle_{H^1(\mathbb{R}^n)}&=\int_{\Omega}\nabla u\cdot\nabla v\: \dx+ \int_{\Omega} u v\: \dx+\int_{\R^n\backslash\partial\Omega}\nabla u\cdot\nabla v\: \dx + \int_{\R^n\backslash\partial\Omega} u v\: \dx\notag\\
&=\left\langle \partial_{\nu,i}u,\TTr_i v\right\rangle_{\Hm,\Hp}-\left\langle \partial_{\nu,e}u,\TTr_ev\right\rangle_{\Hm,\Hp}\notag\\
&=0,\notag
\end{align}
because $u$ is $1$-harmonic in $\mathbb{R}^n\setminus \partial\Omega$ and $\TTr_i v=\TTr_e v=0$. This proves the orthogonal decomposition $H^1(\R^n)= H_0^1(\mathbb{R}^n\setminus \partial\Omega)\oplus V_{1,\mathcal{S}}(\mathbb{R}^n\setminus \partial\Omega)$. The right-hand side of this decomposition equals $\ker~\llbracket\TTr \rrbracket$, as can be seen by evaluating the condition $\llbracket\TTr u\rrbracket=0$ on either side of \eqref{E:orthotwosided}. The remaining parts of (i) now follow with \eqref{normTransmi}, Theorem \ref{Trisom} (iii) and identity \eqref{identH10}. 

The first claim in (ii) follows from the boundedness of $\llbracket\TTr\rrbracket$, the claims on isometry and surjectivity follow from \eqref{normTransmi} and Theorem \ref{Trisom} (iii). 

The limit $u$ of a convergent sequence of elements of $V_{1,\Dcal}(\R^n\backslash \partial\Omega)$ is in $V_{1}(\R^n\backslash \partial\Omega)$, and since $\llbracket \partial_{\nu}\rrbracket$ is bounded on $V_{1}(\R^n\backslash \partial\Omega)$ by \eqref{E:normalderjumpbd}, it follows that $\llbracket \partial_{\nu}u\rrbracket=0$; this shows the first claim in (iii). The boundedness of $\partial_{\nu}$ follows from \eqref{E:boundnormalder}, its bijectivity is due to the unique solvability of the Neumann problems on $\Omega$ and $\mathbb{R}^n\setminus \overline{\Omega}$.

To prove (iv), note that, given $u\in V_1(\R^n\backslash\partial\Omega)$, we have 
\begin{align}
\big\langle \llbracket \partial_{\nu} u\rrbracket, \Tr v\big\rangle_{\Hm, \Hp}&=\int_{\R^n\backslash\partial\Omega}(\Delta u)v\: \dx+\int_{\R^n\backslash\partial\Omega}\nabla u\cdot\nabla v\: \dx\nonumber\\&=\left\langle u,v\right\rangle_{H^1(\R^n\backslash\partial\Omega)} \label{E:orthotoharmfunctions}
\end{align} 
for any $v\in V_{1,\Scal}(\R^n\backslash \partial\Omega)$. Since $\Tr$ in \eqref{E:tracesurjective} is surjective, it follows that $\left\llbracket \partial_{\nu} u\right\rrbracket=0$ in $\Hm$ if and only if \eqref{E:orthotoharmfunctions} is zero for all $v\in V_{1,\Scal}(\R^n\backslash \partial\Omega)$, and this is the case if and only if $u$ belongs to the orthogonal complement of $V_{1,\Scal}(\R^n\backslash \partial\Omega)$ in $V_1(\R^n\backslash\partial\Omega)$. 
\end{proof}

\begin{remark}\mbox{}
\begin{enumerate}
\item[(i)] If instead of the natural dual space norm based on \eqref{normTransmi}, we use the equivalent Hilbert space norm 
$g\mapsto \big(\|g\|_{\Hm, i}^2+\|g\|_{\Hm, e}^2\big)^{1/2}$ on $\Hm$, where $\|\cdot\|_{\Hm, i}$ and $\|\cdot\|_{\Hm, e}$ are the natural dual space norms based on $\|\cdot\|_{\Hp, i}$ and $\|\cdot\|_{\Hp, e}$, then \eqref{E:normaldersurjective} becomes an isometry.
\item[(ii)] To give an alternative proof of (iv), we could observe that, given $v\in V_1(\R^n\backslash\partial\Omega)$, we have 
\begin{equation}\label{E:orthotoforms}
\big\langle \partial_{\nu} u, \llbracket\Tr v \rrbracket\big\rangle_{\Hm, \Hp}=\left\langle u,v\right\rangle_{H^1(\R^n\backslash\partial\Omega)}
\end{equation} 
for all $u\in V_{1,\Dcal}(\R^n\backslash \partial\Omega)$. Since $\partial_\nu$ in \eqref{E:normaldersurjective} is surjective and $\Hp$ is a Hilbert space, it follows that $\llbracket\Tr v \rrbracket=0$ in $\Hp$ if and only if $v$ belongs to the orthogonal complement of  $V_{1,\Dcal}(\R^n\backslash \partial\Omega)$ in $V_1(\R^n\backslash\partial\Omega)$.
\end{enumerate}
\end{remark}

By $\mathrm{P}_{1,\Scal}$ and $\mathrm{P}_{1,\Dcal}$, we denote the orthogonal projections from $H^1(\R^n\backslash \partial\Omega)$ onto the closed subspaces 
$V_{1,\Scal}(\R^n\backslash \partial\Omega)$ and $V_{1,\Dcal}(\R^n\backslash \partial\Omega)$ respectively. The following is a straightforward consequence of Lemma \ref{L:spaces} (iv).

\begin{corollary}\label{C:orthoV}
Let $\Omega$ be two-sided $H^1$-admissible. If $u,v\in V_1(\R^n\backslash \partial\Omega)$, then 
\begin{multline*}
\left\langle u,v\right\rangle_{H^1(\R^n\backslash\partial\Omega)}= \big\langle\llbracket \partial_{\nu} u\rrbracket, \Tr\circ \mathrm{P}_{1,\Scal}\,v\big\rangle_{\Hm,\Hp}\\
+\big\langle \partial_{\nu} \circ \mathrm{P}_{1,\Dcal}\,v,\left\llbracket\Tr u\right\rrbracket
\big\rangle_{\Hm,\Hp}.
\end{multline*}
\end{corollary}

Now suppose that $\Omega\subset \R^n$ is a two-sided $\dot{H}^1$-admissible domain. We write
\[\dot{V}_{0,\dot{\Scal}}(\R^n\backslash \partial\Omega):=\{u\in \dot{V}_0(\R^n\backslash \partial\Omega)\ |\ \llbracket\dTr u\rrbracket=0\}\]
and  
\[\dot{V}_{0,\dot{\Dcal}}(\R^n\backslash \partial\Omega):=\{u\in \dot{V}_0(\R^n\backslash \partial\Omega)\ |\ \llbracket \dot{\partial}_{\nu} u\rrbracket=0\}.\]
The next lemma is seen similarly as Lemma \ref{L:spaces}.
\begin{lemma}\label{L:spaceso} Let $\Omega\subset \R^n$ be two-sided $\dot{H}^1$-admissible. Then the following statements hold:
\begin{enumerate}
\item[(i)] Each element of $\dot{V}_{0,\dot{\Scal}}(\R^n\backslash \partial\Omega)$ has a unique continuation to an element of $\ker~\llbracket\dTTr \rrbracket$. In this sense, the space $\dot{V}_{0,\dot{\Scal}}(\R^n\backslash \partial\Omega)$ is the orthogonal complement of the closure of $C_c^\infty(\R^n\backslash\partial\Omega)$ in $\ker~\llbracket\dTTr \rrbracket$. This closure is also the kernel of the linear operator $\dTTr:\ker~\llbracket\dTTr \rrbracket\to \Hpo$ defined as $\dTTr:=\dTTr_i=\dTTr_e$, which is bounded with operator norm one.
\item[(ii)] The space $\dot{V}_{0,\dot{\Scal}}(\R^n\backslash \partial\Omega)$ is a closed subspace of $\dot{H}^1(\R^n\backslash \partial\Omega)$.
The linear operator 
\[\dTr:\dot{V}_{0,\dot{\Scal}}(\R^n\backslash \partial\Omega)\to \Hpo,\]
defined as the restriction $\dTr:=\dTTr|_{\dot{V}_{0,\dot{\Scal}}(\R^n\backslash \partial\Omega)}$, is an isometry and onto. 
\item[(iii)] The space $\dot{V}_{0,\dot{\Dcal}}(\R^n\backslash \partial\Omega)$ is a closed subspace of $\dot{H}^1(\R^n\backslash \partial\Omega)$.
The linear operator 
\[\dot\partial_\nu: \dot{V}_{0,\dot{\Dcal}}(\R^n\backslash \partial\Omega)\to \Hm,\]
defined as $\dot\partial_\nu:=\dot{\partial}_{\nu,i}=\dot\partial_{\nu,e}$, is bijective and bounded.
\item[(iv)] The space $\dot{V}_0(\R^n\backslash \partial\Omega)$ admits the orthogonal decomposition 
\[\dot{V}_0(\R^n\backslash \partial\Omega)=\dot{V}_{0,\dot{\Scal}}(\R^n\backslash \partial\Omega)\oplus \dot{V}_{0,\dot{\Dcal}}(\R^n\backslash \partial\Omega).\]
\end{enumerate}
\end{lemma}

\begin{remark}\label{R:readjustedversion}\mbox{}
\begin{enumerate}
\item[(i)] Note that, since constants are ignored by $\dTTr_i$ and $\dTTr_e$, elements of
$\ker~\llbracket\dTTr \rrbracket$ are classes modulo locally constant functions (one constant on $\Omega$ and one on $\mathbb{R}^n\backslash\overline{\Omega}$). The space $\dot{H}^1(\R^n)$, whose elements are classes modulo a single constant, is a proper subspace of
$\ker~\llbracket\dTTr \rrbracket$.
\item[(ii)] Suppose that $\Omega$ is both two-sided $H^1$- and $\dot{H}^1$-admissible and $u\in \ker ~\llbracket\dTTr \rrbracket$. 
Let $w_i\in u_i$ and $w_e\in u_e$ be representatives modulo constants of $u_i$ and $u_e$. 
By the arguments used to show Proposition \ref{P:traceo}, both $\TTr_iw_i$ and $\TTr_ew_e$ are well-defined elements of $\Hp$, and by construction, there is a constant $c\in\R$ such that $\TTr_iw_i-\TTr_ew_e=c$. 
Setting $w':=w_i+w_e+c\mathds{1}_{\R^n\backslash\overline{\Omega}}$ we obtain a representative $w'$ of $u\in \ker ~\llbracket\dTTr \rrbracket$ which satisfies $\llbracket\TTr w'\rrbracket=0$ in $\Hp$. The equivalence class $\overline{u}$ of $w'$ modulo single constants is a uniquely determined element of $\dot{H}^1(\R^n)$, which we call the \emph{zero trace jump readjusted representative} of $u$. 

The zero trace jump readjusted representative $\overline{u}$ of $u\in 
\dot{V}_{0,\dot{\Scal}}(\R^n\backslash \partial\Omega)$ is an element of the orthogonal complement of $C_c^\infty(\R^n\backslash\partial\Omega)$ in $\dot{H}^1(\R^n)$; we denote it by $\overline{V}_{0,\dot{\mathcal{S}}}(\R^n\backslash \partial\Omega)$. Setting $\overline{\Tr}\:\overline{u}:=\dTr u$, we obtain a linear bijection 
\begin{equation}\label{E:readjustetrace}
\overline{\Tr}:\overline{V}_{0,\dot{\mathcal{S}}}(\R^n\backslash \partial\Omega)\to\Hpo.
\end{equation}
\end{enumerate}
\end{remark}

By $\mathrm{P}_{0,\dot{\Scal}}$ and $\mathrm{P}_{0,\dot{\Dcal}}$, we denote the orthogonal projections from $\dot{H}^1(\R^n\backslash \partial\Omega)$ onto the closed subspaces 
$\dot{V}_{0,\dot{\Scal}}(\R^n\backslash \partial\Omega)$ and $\dot{V}_{0,\dot{\Dcal}}(\R^n\backslash \partial\Omega)$ respectively. Lemma \ref{L:spaceso} (iv) now gives the following.

\begin{corollary}\label{C:orthoVo}
Let $\Omega$ be two-sided $\dot{H}^1$-admissible. If $u,v\in \dot{V}_0(\R^n\backslash \partial\Omega)$, then 
\begin{multline*}
\langle u,v\rangle_{\dot{H}^1(\R^n\backslash\partial\Omega)}= \big\langle\llbracket\dot\partial_\nu u\rrbracket, \dTr\circ \mathrm{P}_{0,\dot{\Scal}}\,v\big\rangle_{\Hmo,\Hpo}
\\+\big\langle \dot\partial_\nu\circ \mathrm{P}_{0,\dot{\Dcal}}\,v,\llbracket\dTr u\rrbracket
\big\rangle_{\Hmo,\Hpo}.
\end{multline*}
\end{corollary}

\subsection{Double layer potentials}\label{Subsec:DLP}

Let $\Omega\subset \R^n$ be two-sided $H^1$-admissible.
Given $f\in \Hp$, we call $u\in H^1(\R^n\backslash \partial\Omega)$ a \emph{weak solution} of \eqref{EqTrProblem} with $k=1$ and $g=0$, that is, a weak solution to the problem formally stated as
\begin{equation}\label{E:transmiDL}
\begin{cases}
-\Delta u+ u &= 0 \quad \text{on $\R^n\backslash \partial\Omega$}\\
u_i|_{\partial\Omega}-u_e|_{\partial\Omega}&=-f\\
\frac{\partial_iu_i}{\partial\nu}|_{\partial\Omega}-\frac{\partial_eu_e}{\partial \nu}|_{\partial\Omega}&= 0,
\end{cases}
\end{equation}
if we have $\left\langle u,v\right\rangle_{H^1(\R^n\backslash\partial\Omega)}=0$ for all $v\in C_c^\infty(\R^n\backslash \partial\Omega)\cup V_{1,\Scal}(\R^n\backslash \partial\Omega)$ and $\left\langle v,u\right\rangle_{H^1(\R^n\backslash \bord)}=-\left\langle \partial_\nu v,f\right\rangle_{\Hm,\Hp}$ for all 
$v\in V_{1,\Dcal}(\R^n\backslash \partial\Omega)$. 

\begin{lemma}\label{L:transmiDL}
Let $\Omega\subset \R^n$ be two-sided $H^1$-admissible. For any $f\in \Hp$ there is a unique weak solution $u^f$ of \eqref{E:transmiDL}. It is an element of $V_{1,\Dcal}(\R^n\backslash \partial\Omega)$ and satisfies $\|u^f\|_{H^1(\R^n\backslash \partial\Omega)}\leq \|f\|_{\Hp}$.
\end{lemma}
\begin{proof}
Since, by Lemma \ref{L:spaces} (iii), the linear functional $v\mapsto \left\langle \partial_\nu v,f\right\rangle_{\Hm,\Hp}$ is bounded on the closed subspace $V_{1,\Dcal}(\R^n\backslash \partial\Omega)$ of 
$H^1(\R^n\backslash \partial\Omega)$, the result follows from the Riesz representation theorem.
\end{proof} 

We refer to the bounded linear operator $\Dcal:\Hp\to V_{1,\Dcal}(\R^n\backslash \partial\Omega)$ defined by 
\[\Dcal f:=u^f\]
as the \emph{double layer potential operator associated with the transmission problem for $1-\Delta$ and $\Omega$}.
\begin{corollary}\label{C:DLbijective}
Let $\Omega$ be two-sided $H^1$-admissible. The operator $\Dcal$ is bijective, and its inverse is $\Dcal^{-1}=-\llbracket\Tr \rrbracket$.
\end{corollary}

Now let $\Omega\subset \R^n$ be two-sided $\dot{H}^1$-admissible. Given $f\in \Hpo$, we call $u\in \dot{H}^1(\R^n\backslash \partial\Omega)$ a \emph{weak solution in the $\dot{H}^1$-sense} of \eqref{EqTrProblem} with $k=0$ and $g=0$, that is, 
\begin{equation}\label{E:transmiDLo}
\begin{cases}
-\Delta u &= 0 \quad \text{on $\R^n\backslash \partial\Omega$}\\
u_i|_{\partial\Omega}-u_e|_{\partial\Omega}&=-f\\
\frac{\partial_iu_i}{\partial\nu}|_{\partial\Omega}-\frac{\partial_eu_e}{\partial \nu}|_{\partial\Omega}&= 0,
\end{cases}
\end{equation}
if it satisfies $\left\langle u,v\right\rangle_{\dot{H}^1(\R^n\backslash\partial\Omega)}=0$ for all $v\in C_c^\infty(\R^n\backslash \partial\Omega)\cup \dot{V}_{0,\dot{\Scal}}(\R^n\backslash \partial\Omega)$ and $\left\langle v,u\right\rangle_{\dot{H}^1(\R^n\backslash \bord)}=-\big\langle \dot\partial_\nu v,f\big\rangle_{\Hmo,\Hpo}$ for all 
$v\in \dot{V}_{0,\dot{\Dcal}}(\R^n\backslash \partial\Omega)$. 
\begin{lemma}
Let $\Omega\subset \R^n$ be two-sided $\dot{H}^1$-admissible. For any $f\in \Hpo$, there is a unique weak solution $u^f$ of \eqref{E:transmiDLo} in the $\dot{H}^1$-sense. It is an element of $\dot{V}_{0, \dot{\Dcal}}(\R^n\backslash \partial\Omega)$ and satisfies $\|u^f\|_{\dot{H}^1(\R^n\backslash \partial\Omega)}\leq \|f\|_{\Hpo}$.
\end{lemma}
We refer to the bounded linear operator $\dot{\Dcal}:\Hpo\to \dot{V}_{0,\dot{\Dcal}}(\R^n\backslash \partial\Omega)$ defined by $\dot{\Dcal}f:=u^f$
as the \emph{double layer potential operator associated with the transmission problem for $-\Delta$ and $\Omega$ in the $\dot{H}^1$-sense}.
\begin{corollary}\label{C:DLbijectiveo}
Let $\Omega$ be two-sided $\dot{H}^1$-admissible. The operator $\dot{\Dcal}$ is bijective, and its inverse is $\dot{\Dcal}^{-1}=-\llbracket\dTr \rrbracket$.
\end{corollary}

\begin{remark} A priori, the trace jump condition in \eqref{E:transmiDLo} is only an equality modulo additive constants. However, by
arguments similar to Remarks \ref{R:representative} and \ref{R:readjustedversion}, the situation is auto-improving: If $\Omega$ is both two-sided $H^1$- and $\dot{H}^1$-admissible and $f\in\mathcal{B}(\partial\Omega)$ is given, let $u^{[f]}$ be the unique weak solution in the $\dot{H}^1$-sense of \eqref{E:transmiDLo} with the equivalence class $[f]\in\Hpo$ of $f$ modulo constants in place of $f$.
We can find representatives $w_i\in u^{[f]}_i$ and $w_e\in u^{[f]}_e$ modulo constants and some $c\in\R$ such that $\TTr_iw_i-\TTr_ew_e=-f+c$ in $\Hp$. We set $w':=w_i+w_e+c\mathds{1}_{\R^n\backslash \overline{\Omega}}$. The equivalence class $\overline{u}^{[f]}$ of $w'$ modulo a single constant is uniquely determined.
It satisfies $\Delta \overline{u}^{[f]}=0$ on $\R^n\backslash\partial\Omega$, $\llbracket\Tr \overline{u}^{[f]}\rrbracket=-f$ as an equality in $\Hp$ and 
$\llbracket \dot\partial_\nu\overline{u}^{[f]}\rrbracket=0$, which by Remark \ref{R:BprimeB} may be interpreted as an equality of linear functionals on $\Hp$. One could call $\overline{u}^{[f]}$ the \emph{trace jump readjusted representative of $u^{[f]}$}.
\end{remark}

\subsection{Single layer potentials}\label{Subsec:SLP}

Let $\Omega$ be two-sided $H^1$-admissible. Given $g\in \Hm$, we call $u\in H^1(\R^n\backslash \partial\Omega)$ a \emph{weak solution} of \eqref{EqTrProblem} with $k=1$ and $f=0$, that is, 
\begin{equation}\label{E:transmiSL}
\begin{cases}
-\Delta u+ u &= 0 \quad \text{on $\R^n\backslash \partial\Omega$}\\
u_i|_{\partial\Omega}-u_e|_{\partial\Omega}&=0\\
\frac{\partial_iu_i}{\partial\nu}|_{\partial\Omega}-\frac{\partial_eu_e}{\partial \nu}|_{\partial\Omega}&= g,
\end{cases}
\end{equation}
if $\left\langle u,v\right\rangle_{H^1(\R^n\backslash\partial\Omega)}=0$ for all $v\in V_{1,\Dcal}(\R^n\backslash \partial\Omega)$, as well as $\left\langle u,v\right\rangle_{H^1(\R^n\backslash\partial\Omega)}=\left\langle g, \Tr v\right\rangle_{\Hm,\Hp}$ for all $v\in H^1(\R^n)$.

\begin{lemma}\label{L:transmiSL}
Let $\Omega$ be two-sided $H^1$-admissible. For any $g\in \Hm$, there is a unique weak solution $u_g$ of \eqref{E:transmiSL}. It is in $V_{1,\mathcal{S}}(\R^n\backslash \partial\Omega)$ and satisfies $\|u_g\|_{H^1(\R^n\backslash \partial\Omega)}\leq \|g\|_{\Hm}$.
\end{lemma}

\begin{proof}
By Lemma \ref{L:spaces} (ii) the linear functional $v\mapsto \left\langle g, \Tr v\right\rangle_{\Hm,\Hp}$ is bounded on the closed subspace $V_{1,\mathcal{S}}(\R^n\backslash \partial\Omega)$ of 
$H^1(\R^n\backslash \partial\Omega)$, so the result follows from the Riesz representation theorem once again.
\end{proof} 

We refer to the bounded linear operator $\Scal:\Hm\to V_{1,\mathcal{S}}(\R^n\backslash \partial\Omega)$, defined by 
\[\Scal g:=u_g\]
as the \emph{single layer potential operator associated with the transmission problem for $1-\Delta$ and $\Omega$}.
\begin{corollary}\label{C:SLbijective}
Let $\Omega$ be two-sided $H^1$-admissible. The operator $\Scal$ is bijective, and its inverse is $\Scal^{-1}=\llbracket\partial_\nu\rrbracket$.
\end{corollary} 

Suppose that $\Omega$ is two-sided $\dot{H}^1$-admissible. Given $g\in \Hmo$, we call $u\in \dot{H}^1(\R^n\backslash \partial\Omega)$ a \emph{weak solution in the $\dot{H}^1$-sense} of \eqref{EqTrProblem} with $k=0$ and $f=0$, that is, 
\begin{equation}\label{E:transmiSLo}
\begin{cases}
-\Delta u &= 0 \quad \text{on $\R^n\backslash \partial\Omega$}\\
u_i|_{\partial\Omega}-u_e|_{\partial\Omega}&=0\\
\frac{\partial_iu_i}{\partial\nu}|_{\partial\Omega}-\frac{\partial_eu_e}{\partial \nu}|_{\partial\Omega}&= g,
\end{cases}
\end{equation}
if $\left\langle u,v\right\rangle_{\dot{H}^1(\R^n\backslash\partial\Omega)}=0$ for all $v\in \dot{V}_{0,\dot{\Dcal}}(\R^n\backslash \partial\Omega)$, as well as $\left\langle u,v\right\rangle_{\dot{H}^1(\R^n\backslash\partial\Omega)}=\big\langle g, \dTTr v\big\rangle_{\Hmo,\Hpo}$ for all $v\in \dot{H}^1(\R^n)$.
 
\begin{lemma}\label{L:transmiSLo}
Let $\Omega$ be two-sided $\dot{H}^1$-admissible. For any $g\in \Hmo$, there is a unique weak solution $u_g$ of \eqref{E:transmiSLo} in the $\dot{H}^1$-sense. It is an element of $\dot{V}_{0,\dot{\mathcal{S}}}(\R^n\backslash \partial\Omega)$ and satisfies $\|u_g\|_{\dot{H}^1(\R^n\backslash \partial\Omega)}\leq \|g\|_{\Hmo}$.
\end{lemma}

We refer to the bounded linear operator $\dot{\Scal}:\Hmo\to \dot{V}_{0,\dot{\mathcal{S}}}(\R^n\backslash \partial\Omega)$ defined by $\dot{\Scal}g:=u_g$ as the \emph{single layer potential operator associated with the transmission problem for $-\Delta$ and $\Omega$ in the $\dot{H}^1$-sense.}

\begin{corollary}\label{C:SLbijectiveo}
Let $\Omega$ be two-sided $\dot{H}^1$-admissible. The operator $\dot{\Scal}$ is bijective, and its inverse is $\dot{\Scal}^{-1}=\llbracket\dot\partial_\nu\rrbracket$.
\end{corollary} 

\begin{remark}\label{R:readjustsingle}
Let $\Omega$ be both two-sided $H^1$- and $\dot{H}^1$-admissible, $g\in \Hmo$ and let $u_g$ be as above. Let $\overline{u}_g\in \overline{V}_{0,\dot{\Scal}}(\R^n\backslash \partial\Omega)$ be the zero trace jump readjusted representative of $u_g$ as in Remark \ref{R:readjustedversion}. Since it is uniquely determined, $\overline{\Scal}g:=\overline{u}_g$ defines a bounded linear map 
\begin{equation}\label{E:readjustsingle}
\overline{\Scal}:\Hmo\to \overline{V}_{0,\dot{\Scal}}(\R^n\backslash \partial\Omega);
\end{equation}
it is a zero trace jump readjusted variant of $\dot{\Scal}$.
\end{remark}

\subsection{Superposition}\label{Subsec:Superposition}

Superposition gives well-posedness for \eqref{EqTrProblem}. Let $\Omega$ be $H^1$-admissible. Given $f\in \Hp$ and $g\in \Hm$, we call $u\in H^1(\R^n\backslash \partial\Omega)$ a \emph{weak solution} for \eqref{EqTrProblem} with $k=1$ if $\left\langle u,v\right\rangle_{H^1(\R^n\backslash \partial\Omega)}=0$ for all $v\in C_c^\infty(\R^n\backslash\partial\Omega)$, $\left\langle u,v\right\rangle_{H^1(\R^n\backslash \partial\Omega)}=\left\langle g,\Tr v\right\rangle_{\Hm,\Hp}$ for all $v\in V_{1,\Scal}(\R^n\backslash \partial\Omega)$, as well as $\left\langle u,v\right\rangle_{H^1(\R^n\backslash \partial\Omega)}=\left\langle \partial_\nu v, f\right\rangle_{\Hm,\Hp}$ for all $v\in V_{1,\Dcal}(\R^n\backslash \partial\Omega)$.

\begin{corollary}\label{C:superpos}
Let $\Omega$ be $H^1$-admissible. For any $f\in \Hp$ and $g\in \Hm$, the unique weak solution $u$ of \eqref{EqTrProblem} with $k=1$ is $u=\Scal g-\Dcal f$. It is an element of $V_1(\R^n\backslash \partial\Omega)$ and satisfies $$\|u\|_{H^1(\R^n\backslash \partial\Omega)}\leq \|f\|_{\Hp}+\|g\|_{\Hm}.$$
\end{corollary}

If $\Omega$ is $\dot{H}^1$-admissible, and if $f\in \Hpo$ and $g\in \Hmo$, then we call $u\in \dot{H}^1(\R^n\backslash \partial\Omega)$ a \emph{weak solution} for \eqref{EqTrProblem} with $k=0$ in the $\dot{H}^1$-sense if $\left\langle u,v\right\rangle_{\dot{H}^1(\R^n\backslash \partial\Omega)}=0$ for all $v\in C_c^\infty(\R^n\backslash\partial\Omega)$, $\left\langle u,v\right\rangle_{\dot{H}^1(\R^n\backslash \partial\Omega)}=\big\langle g,\dTr v\big\rangle_{\Hmo,\Hpo}$ for all $v\in \dot{V}_{0,\dot{\Scal}}(\R^n\backslash \partial\Omega)$ and we have $\left\langle u,v\right\rangle_{\dot{H}^1(\R^n\backslash \partial\Omega)}=\langle \dot\partial_\nu v, f\rangle_{\Hmo,\Hpo}$ for all $v\in \dot{V}_{0,\dot{\Dcal}}(\R^n\backslash \partial\Omega)$.

\begin{corollary}\label{C:superposo}
Let $\Omega$ be two-sided $\dot{H}^1$-admissible. For any $f\in \Hpo$ and $g\in \Hmo$, the unique weak solution $u$ in the $\dot{H}^1$-sense of \eqref{EqTrProblem} with $k=0$ is $u=\dot{\Scal}g-\dot{\Dcal}f$. It is an element of $\dot{V}_0(\R^n\backslash \partial\Omega)$ and satisfies $$\|u\|_{\dot{H}^1(\R^n\backslash \partial\Omega)}\leq \|f\|_{\Hpo}+\|g\|_{\Hmo}.$$
\end{corollary}

\section{Resolvent representations}\label{Sec:Resolvent}

We give resolvent representations of the layer potential operators defined in the variational sense in Section~\ref{Sec:TrProb}. This allows to recover the classical integral formulas when the domain $\Omega$ is a Lipschitz domain and the boundary is endowed with the surface measure.

\subsection{Representations of single layer potentials}\label{Subsec:Reso-SLP}

Let $u\mapsto \Gcal u=\big((1+|\xi|^2)^{-1}\hat{u}\big)^\vee$ be the Bessel potential operator of order $2$, where $u\mapsto \hat{u}$ denotes the Fourier transform on tempered distributions and $u\mapsto \check{u}$ its inverse. It is well known that $\Gcal=(I-\Delta)^{-1}$ in this distributional sense, that $\Gcal$ is bounded on $L^2(\R^n)$ and that $\Gcal$ acts as an isometric isomorphism from $H^{-1}(\R^n)$ onto $H^1(\R^n)$.  

If $\Omega\subset \R^n$ is two-sided $H^1$-admissible, then $\Gcal$, viewed on $L^2(\R^n)$, is the resolvent operator uniquely associated with the symmetric bilinear form \eqref{E:sp} when endowed with the domain $H^1(\R^n)=H^1_0(\R^n\backslash \partial\Omega)\oplus V_{1,\Scal}(\R^n\backslash \partial\Omega)$.

Using a similar agreement as in Remark \ref{R:extendbyzero} together with Lemma \ref{L:spaces} (ii), the adjoint $\Tr^\ast:\Hm\to V_{1,\Scal}'(\R^n\backslash \partial\Omega)$ of the restricted trace operator $\Tr$ as in Lemma \ref{L:spaces} (ii)
can be viewed as a bounded linear operator $\Tr^\ast:\Hm\to H^{-1}(\R^n)$; it is characterized by 
\begin{equation}\label{E:char}
\left\langle g, \TTr v\right\rangle_{\Hm,\Hp}=\left\langle \Tr^\ast g,v\right\rangle_{H^{-1}(\R^n),H^1(\R^n)},\quad v\in H^1(\R^n),\ g\in \Hm.
\end{equation}
We obtain the following representation for the single layer potential operator.

\begin{lemma}\label{L:repSL}
Let $\Omega$ be two-sided $H^1$-admissible. Then $\Scal=\Gcal\circ \Tr^\ast$.
\end{lemma}
\begin{proof}
Let $g\in \Hm$. Then $\Tr^\ast g\in H^{-1}(\R^n)$ and consequently $\Gcal\circ \Tr^\ast g\in H^1(\R^n)$. Since
$\Gcal:H^{-1}(\R^n)\to H^1(\R^n)$ is a Riesz isometry, we have 
\[\left\langle \Gcal\circ \Tr^\ast g,v\right\rangle_{H^1(\R^n)}=\left\langle \Tr^\ast g,v\right\rangle_{H^{-1}(\R^n),H^1(\R^n)}\]
for all $v\in H^1(\R^n)$, and the definition of $\Scal$ gives
\begin{equation}\label{E:compareSG}
\left\langle g, \TTr v\right\rangle_{\Hm,\Hp}=\left\langle \Scal g,v\right\rangle_{H^1(\R^n)}.
\end{equation}
Combining with \eqref{E:char}, the lemma follows.
\end{proof}

Let $G$ be the Bessel kernel of order two, that is, the fundamental solution for $\Delta-1$ on $\R^n$.  It is well known that $\Gcal w=G\ast w$ for any $w\in C_c^\infty(\mathbb{R}^n)$. For a nonnegative Radon measure $\nu$ on $\partial\Omega$ the convolution 
\[G\ast \nu(x)=\int_{\partial\Omega}G(x-y)\nu(\dy),\quad x\in \R^n,\]
is a lower semicontinuous function taking values in $[0,+\infty]$. For finite $\nu$ it is
finite at all points $x\in \R^n\backslash \partial\Omega$. For a finite signed Radon measure $\nu$ on $\partial\Omega$ the convolution 
\[G\ast\nu(x)=G\ast \nu^+(x)-G\ast \nu^-(x), \quad x\in \R^n\backslash\partial\Omega,\] 
is a Borel function $G\ast\nu$ on $\R^n\setminus\partial\Omega$.

We say that a nonnegative Radon measure $\nu$ on $\partial\Omega$ has \emph{finite energy} if there is a constant $c>0$ such that 
\begin{equation}\label{E:finiteenergy}
\int_{\partial\Omega}|v|\: {\rm d}\nu\leq c\:\|v\|_{H^1(\R^n)},\quad v\in H^1(\R^n)\cap C_c(\R^n). 
\end{equation}
Given a finite signed Radon measure $\nu$ on $\partial\Omega$, we say that it has finite energy if its total variation measure $\nu^++\nu^-$ has finite energy.

\begin{remark}
It is well known that, by the Riesz representation theorem for measures, the cone of nonnegative elements of $H^{-1}(\R^n)$ is in one-to-one correspondence with the cone of nonnegative Radon measures of finite energy on $\R^n$. This can be seen using  \cite[Chapter 6, Exercise 4]{RUDIN-1991}; a variant of the argument is provided in \cite[Proposition 9.2.1]{BOULEAU-HIRSCH-1991}.
\end{remark}

\begin{proposition}\label{P:repS}
Let $\Omega$ be two-sided $H^1$-admissible and let $\nu$ be a nonnegative Radon measure on $\partial\Omega$ of finite energy or a finite signed Radon measure on $\partial\Omega$ of finite energy. Then sets of zero capacity have zero $\nu$-measure, and $\nu$ defines an element of $\Hm$ by
\begin{equation}\label{E:nudef}
\left\langle \nu,f\right\rangle_{\Hm,\Hp}:=\int_{\partial\Omega} f{\rm d}\nu,\quad f\in \Hp.
\end{equation}
Moreover, $\Scal\nu(x)=G\ast \nu(x)$, $x\in\R^n\backslash\partial\Omega$.
\end{proposition}

\begin{proof} It suffices to prove the result for a nonnegative Radon measure $\nu$ of finite energy. The first claim is shown in \cite[Lemma 2.2.3]{FOT94}. 
Estimate \eqref{E:finiteenergy} extends to all $v\in H^1(\R^n)$ and gives $|\int_{\partial\Omega}\TTr v \:{\rm d}\nu|\leq c\|v\|_{H^1(\R^n)}$; here $\TTr$ is as in Lemma \ref{L:spaces} (i). Optimizing over $v$ gives
\[\Big|\int_{\partial\Omega} f\:d\nu\Big|\leq c\:\|f\|_{\Hp}, \quad f\in \Hp,\]
with a (different) constant $c>0$. Consequently, $\nu \in \Hm$. One can follow \cite[Theorem 2.2.2]{FOT94} to see that
\begin{equation}\label{E:measureident}
\left\langle \nu,\TTr v\right\rangle_{\Hm,\Hp}=\int_{\partial\Omega}\TTr v\: {\rm d}\nu=\left\langle G\ast \nu,v\right\rangle_{H^1(\R^n)},\quad v\in H^1(\R^n).
\end{equation}
Together with \eqref{E:compareSG} and Remark \ref{R:ellipticreg} this gives the last claim in Proposition \ref{P:repS}.
\end{proof}

\begin{remark}{\color{black}
Let $\partial\Omega$ be compact and $\nu$ a signed Radon measure with $\supp\nu=\partial\Omega$; then $\nu$ is obviously finite.
For $n=1,2$ \cite[formula (1.2.11)]{ADAMS-1996} shows that $G\ast \nu$ is bounded and continuous on all of $\R^n$ and, as a consequence, the energy of $\nu$ is finite.
If there are constants $n-2<d<n$ and $c>0$ such that $\nu^\pm(B(x,r))\leq c\:r^d$ for all $x\in \partial\Omega$, $0<r<1$, then these facts are also true for $n\geq 3$.
An easy proof follows by \cite[formula (1.2.12)]{ADAMS-1996} and \cite[p. 109]{MATTILA-1995}, combined with arguments similar to \cite[Lemma 1 in Section 3.4.5]{TRIEBEL-1992}.}
\end{remark}

\begin{corollary}\label{cor-L2mu}
Let $\Omega$ be two-sided $H^1$-admissible and let $\mu$ be a nonnegative Radon measure on $\partial\Omega$ with the property that all sets of zero capacity are $\mu$-null sets. If there is some $c>0$ such that 
\begin{equation}\label{E:traceineq}
\|\TTr v\|_{L^2(\partial\Omega,\mu)}\leq c\:
\|v\|_{H^1(\mathbb{R}^n)},\quad v\in H^1(\mathbb{R}^n),
\end{equation}
then for any $g\in L^1(\partial\Omega,\mu)\cap L^2(\partial\Omega,\mu)$, the finite signed Radon measure $(g\cdot\mu)(dy):=g(y)\mu(dy)$ is of finite energy and
\begin{equation}\label{E:classSL}
\Scal (g\cdot\mu)(x)=\int_{\partial\Omega}G(x-y)g(y)\mu(\dy),\quad x\in \R^n\backslash\partial\Omega.
\end{equation}
\end{corollary}

\begin{proof}
Given $g\in L^2(\partial\Omega,\mu)$, we have
\[\int_{\partial\Omega}|\TTr v||g|d\mu\leq \|g\|_{L^2(\partial\Omega,\mu)}\|\TTr v\|_{L^2(\partial\Omega,\mu)}\leq c\:\|g\|_{L^2(\partial\Omega,\mu)}\|v\|_{H^1(\mathbb{R}^n)}\]
for all $v\in H^1(\mathbb{R}^n)$. This shows that the total variation measure $|g(y)|\mu(dy)$ of $g(y)\mu(dy)$ is of finite energy. Formula (\ref{E:classSL}) follows using the preceding proposition.
\end{proof}

\begin{remark}\mbox{}\label{R:strictcaseSL}
If the two-sided $H^1$-admissible domain $\Omega$ is a bounded Lipschitz domain and $\mu=\sigma$ is the surface measure on $\partial\Omega$, then for any $g\in L^2(\partial\Omega,\sigma)$, the right-hand side of \eqref{E:classSL} is a weak solution of \eqref{E:transmiSL}, and the equality \eqref{E:classSL} can alternatively be derived from the uniqueness in Lemma \ref{L:transmiSL}.
\end{remark}

The Riesz potential operator $u\mapsto \Ical u=(|\xi|^{-2}\hat{u})^\vee$ of order $2$ can be considered on the space of tempered distributions modulo polynomials. Viewed in this way, $\Ical=(-\Delta)^{-1}$. It is well known that $\Ical$ is an isometric isomorphism from $\dot{H}^{-1}(\R^n)$ onto $\dot{H}^1(\R^n)$. See  \cite[Section 6.2.1]{GRAFAKOS-2009}; further related details can be found in \cite[Sections 25.1 and 25.2]{SKM93}.

Suppose that $\Omega$ is both two-sided $H^1$- and $\dot{H}^1$-admissible and recall the operators $\overline{\Tr}$ and $\overline{\Scal}$ as in
\eqref{E:readjustetrace} and \eqref{E:readjustsingle}, respectively. Similarly as before, we may view the dual $\overline{\Tr}^\ast$ of $\overline{\Tr}$
as a bounded linear operator $\overline{\Tr}^\ast:\Hmo\to \dot{H}^{-1}(\R^n)$. 

\begin{lemma}
Let $\Omega$ be both a two-sided $H^1$- and $\dot{H}^1$-admissible domain. Then $\overline{\Scal}=\Ical\circ \overline{\Tr}^\ast$.
\end{lemma}

\begin{proof}
Given $g\in \Hmo$, we have $\Ical\circ \overline{\Tr}^\ast g\in \dot{H}^1(\R^n)$, and for all $v\in \dot{H}^1(\R^n)$ then
\begin{multline}\label{E:compareSGo}
\big\langle \Ical\circ \overline{\Tr}^\ast g,v\big\rangle_{\dot{H}^1(\R^n)}
=\big\langle \overline{\Tr}^\ast g,v\big\rangle_{\dot{H}^{-1}(\R^n),\dot{H}^1(\R^n)}
=\big\langle g,\dTTr v\big\rangle_{\Hmo,\Hpo}\\
=\big\langle \dot{\Scal} g,v\big\rangle_{\dot{H}^1(\R^n\backslash\partial\Omega)}= \big\langle \overline{\Scal} g,v\big\rangle_{\dot{H}^1(\R^n)}.
\end{multline}
\end{proof}

Now assume that $n\geq 2$. Let $K$ be the Green's function, that is, the fundamental solution for $\Delta$ on $\R^n$. We have 
$\Ical w=K\ast w$ for any $w\in C_c^\infty(\mathbb{R}^n)$. Since for a two-sided $\dot{H}^1$-admissible domain $\Omega$ the boundary $\partial\Omega$ is compact, any nonnegative Radon measure $\nu$
on $\partial\Omega$ is compactly supported and finite. For such a measure $\nu$ the convolution
\[K\ast \nu(x)=\int_{\partial\Omega}K(x-y)\nu(\dy),\quad x\in\R^n,\]
is a lower semicontinuous function. In the case $n\geq 3$, it takes values in $[0,+\infty]$, in the case $n=2$ in $(-\infty,+\infty]$. In either case it is finite on $\R^n\backslash\partial\Omega$. Given a signed Radon measure $\nu$ on $\partial\Omega$ (then automatically finite), the convolution $K\ast \nu=K\ast \nu^+-K\ast \nu^-$ is a Borel function on $\R^n\backslash\partial\Omega$. We call a signed measure $\nu$ \emph{centered} if $\nu(\R^n)=0$.

\begin{proposition}
Let $\Omega\subset \mathbb{R}^n$, $n\geq 2$, be both a two-sided $H^1$- and $\dot{H}^1$-admissible domain. Let $\nu$ be a centered finite signed Radon measure on $\partial\Omega$ of finite energy. Then sets of zero capacity have zero $\nu$-measure, and $\nu$ defines an element of $\Hmo$ by 
\begin{equation}\label{E:nudefo}
\left\langle \nu, f\right\rangle_{\Hmo,\Hpo}:=\int_{\partial\Omega}f\:{\rm d}\nu,\quad f\in \Hpo,
\end{equation}
with the integral defined using an arbitrary representative of $f$ modulo constants. Moreover, $K\ast \nu$ is a representative modulo constants of $\overline{\Scal}\nu$.
\end{proposition}

\begin{remark}
The assumption that $\nu$ is centered ensures that $\nu$ is an element of $\Hmo$. In the case $n=2$ it is also responsible for the correct decay behaviour of $K\ast\nu$ and its gradient at infinity, see \cite[proof of Lemma 2.5]{AMMARI-2004} and \cite[p. 351, formula (14)]{VLADIMIROV-1971}. 
\end{remark}

\begin{proof}
By Proposition \ref{P:repS}, the measure $\nu$ charges no set of zero capacity, and since $\nu$ is centered, the right hand side of \eqref{E:nudefo} is linear in $f\in \Hpo$. Let $B$ be an open ball containing $\partial\Omega$, let $\chi\in C_c^\infty(B)$ be such that $0\leq\chi\leq 1$ and $\chi\equiv 1$ on $\overline{\Omega}$.
Given $f\in \Hpo$, let $v\in \dot{H}^1(\R^n)$ be such that $\dTTr v=f$ and let $w$ be a representative modulo constants of $v$ such that $\int_B \chi w\:{\rm d}x=0$. Poincar\'e's inequality for $B$ gives $\|\chi w\|_{H^1(\R^n)}\leq c\:\|\chi\|_{C^1(B)}\|v\|_{\dot{H}^1(\Omega)}$ with $c>0$ independent of $v$. Similarly as before, \eqref{E:finiteenergy} extends to all of $H^1(\mathbb{R}^n)$ (with the trace in the integral), and using this fact, we find that $| \int_{\partial\Omega} f\:d\nu|\leq c\:\|f\|_{\Hpo}$, $f\in \Hpo$, where $c>0$ is another constant. This shows that $\nu\in \Hmo$. For $n\geq 3$ we can apply \cite[Theorem 2.2.5]{FOT94} to find that 
\begin{equation}\label{E:desiredfordot}
\big\langle \nu, \dTTr v\big\rangle_{\Hmo,\Hpo}=\int_{\partial\Omega}\dTTr v\: {\rm d}\nu=\left\langle K\ast \nu,v\right\rangle_{\dot{H}^1(\R^n)},\quad v\in \dot{H}^1(\R^n),
\end{equation}
which, together with \eqref{E:compareSGo}, gives the last claim. For $n=2$ we can use the more general argument that for  
$K\ast \nu$, viewed as a distribution, we have 
\[\Big|\sum_{i=1}^n \frac{\partial}{\partial x_i}(K\ast \nu)\Big(\frac{\partial v}{\partial x_i}\Big)\Big|=|(\Delta (K\ast \nu))(v)|=|\nu(v)|=|\int_{\partial\Omega}v\: {\rm d}\nu|\leq c\|v\|_{\dot{H}^1(\mathbb{R}^n)}\] 
for all $v\in C_c^\infty(\mathbb{R}^n)$. Clearly $\sum_{i=1}^n \frac{\partial}{\partial x_i}(K\ast \nu)(\eta_i)=0$ for all divergence free vector fields $\eta=(\eta_1,...,\eta_n)\in C_c^\infty(\mathbb{R}^n,\mathbb{R}^n)$. Therefore the density of $C_c^\infty(\mathbb{R}^n)$ in $\dot{H}^1(\mathbb{R}^n)$ now implies that $\nabla(K\ast \nu)\in L^2(\mathbb{R}^n,\mathbb{R}^n)$, hence $K\ast \nu\in \dot{H}^1(\mathbb{R}^n)$. Also the validity of \eqref{E:desiredfordot} can now be seen using this density.

\end{proof}

\begin{corollary}\label{Cor:kernelSLPd}
Let $\Omega\subset \mathbb{R}^n$, $n\geq 2$, be both two-sided $H^1$- and $\dot{H}^1$-admissible and let $\mu$ be a nonnegative Radon measure on $\partial\Omega$ such that \eqref{E:traceineq} holds. Then $\mu$ is finite and of finite energy. For any element $g$ of 
\[L^2_0(\partial\Omega,\mu)=\bigg\{g\in L^2(\partial\Omega,\mu)\;\bigg|\; \int_{\partial\Omega} g\,{\rm d}\mu=0\bigg\},\] 
the centered signed Radon measure $(g\cdot\mu)(dy) =g(y)\mu(dy)$ is finite and of finite energy. The function 
\begin{equation}\label{E:classicSL}
x\mapsto \int_{\partial\Omega}K(x-y)g(y)\mu(\dy)
\end{equation}
is a representative modulo constants of the function $\overline{\Scal}(g\cdot \mu)$ on $\R^n\backslash\partial\Omega$.
\end{corollary}

\begin{remark}\label{R:strictcaseSLdot}
Suppose $\Omega$ is a bounded Lipschitz domain in $\mathbb{R}^n$, $n\geq 2$. If $\mu=\sigma$ is the surface measure and $g\in L^2_0(\partial\Omega,\sigma)$, then \eqref{E:classicSL} is the classical single layer potential~\cite{AMMARI-2004, VERCHOTA-1984} of $g$. {\color{black}For bounded $g$, it provides a representative of $\overline{\Scal}(g\cdot \sigma)$ which is continuous on all of $\mathbb{R}^n$.}
\end{remark}

\subsection{Representations of double layer potentials}\label{Subsec:Reso-DLP}

Given a two-sided $H^1$-admissible domain $\Omega\subset \R^n$, let $\Rcal$ denote the unique bounded linear operator  from $L^2(\R^n\backslash \partial\Omega)$ into $H^1_0(\R^n\backslash \partial\Omega)\oplus V_{1,\Dcal}(\R^n\backslash \partial\Omega)$ such that
\[\left\langle \Rcal u,v\right\rangle_{H^1(\R^n\backslash\partial\Omega)}=\left\langle u,v\right\rangle_{L^2(\R^n\backslash \partial\Omega)}\]
for all $u\in L^2(\R^n\backslash \partial\Omega)$ and  $v\in H^1_0(\R^n\backslash \partial\Omega)\oplus V_{1,\Dcal}(\R^n\backslash \partial\Omega)$.
This operator is symmetric on $L^2(\R^n\backslash \partial\Omega)$ and symmetric with respect to $\left\langle\cdot,\cdot\right\rangle_{H^1(\R^n\backslash\partial\Omega)}$. 
It extends to an isometric isomorphism from $(H^1_0(\R^n\backslash \partial\Omega)\oplus V_{1,\Dcal}(\R^n\backslash \partial\Omega))'$ onto $H^1_0(\R^n\backslash \partial\Omega)\oplus V_{1,\Dcal}(\R^n\backslash \partial\Omega)$. The adjoint $(\partial_{\nu})^\ast$ of $\partial_{\nu} $ as in Lemma \ref{L:spaces} (iii) maps from $\Hp$ into $(V_{1,\Dcal}(\R^n\backslash \partial\Omega))'$. 
In the spirit of Remark \ref{R:extendbyzero}, it can be viewed as a bounded linear operator from $\Hp$ into $(H^1(\R^n\backslash \partial\Omega))'$ if for each $f\in \Hp$, the functional $(\partial_{\nu})^\ast f$ is silently extended by zero to a bounded linear functional on all of $H^1(\R^n\backslash \partial\Omega)$. The operator $\Dcal$ can be represented in terms of those operators.

\begin{lemma}\label{L:repDL}
Let $\Omega$ be two-sided $H^1$-admissible. Then $\Dcal=-\Rcal\circ (\partial_{\nu})^\ast$.
\end{lemma}

\begin{proof}
Given $f\in \Hp$ and $\varphi\in H^1_0(\R^n\backslash \partial\Omega)\oplus V_{1,\Dcal}(\R^n\backslash \partial\Omega)$ Corollary \ref{C:orthoV} gives
\begin{multline}
-\left\langle \Dcal f,\varphi\right\rangle_{(H^1)',H^1}=-\left\langle \Rcal\circ \Dcal f,\varphi\right\rangle_{H^1}=-\left\langle \Dcal f,\Rcal\varphi\right\rangle_{H^1}\notag\\
=\big\langle (\partial_\nu)\circ \mathrm{P}_{1,\Dcal}\circ \Rcal \varphi,f\big\rangle_{\Hm,\Hp}=\big\langle (\partial_{\nu})^\ast f, \mathrm{P}_{1,\Dcal}\circ \Rcal \varphi\big\rangle_{(H^1)',H^1}\notag\\
=\big\langle (\partial_{\nu})^\ast f,\Rcal \varphi\big\rangle_{(H^1)',H^1}=\big\langle \Rcal \circ (\partial_{\nu})^\ast f,\varphi\big\rangle_{(H^1)',H^1},
\end{multline}
where we use $H^1$ to abbreviate $H^1(\nbord)$.
\end{proof}

\begin{proposition}\label{P:repDL}
Let $\Omega$ be a bounded Lipschitz domain. Then $\Hp$ equals $H^{1/2}(\partial\Omega)$, and for any $f\in H^{1/2}(\partial\Omega)$, we have
\begin{equation}\label{E:classDL}
\Dcal f(x)=-\int_{\partial\Omega}\frac{\partial}{\partial\nu_y} G(x-y)f(y)\:\sigma(\dy),\quad x\in\R^n\backslash\partial\Omega;
\end{equation}
here $\frac{\partial}{\partial\nu_y}$ denotes the classical interior normal derivative and $\sigma$ the surface measure on $\partial\Omega$.
\end{proposition}

\begin{proof}
The classical double layer potential on the right hand side of \eqref{E:classDL} is known to solve \eqref{E:transmiDL} in the weak sense, so the uniqueness part of Lemma \ref{L:transmiDL} implies \eqref{E:classDL}.
\end{proof}

Let $\Omega\subset \R^n$ be a two-sided $\dot{H}^1$-admissible domain. Let $\dot{\Rcal}$ denote the unique bounded linear operator mapping $(\dot{H}^1_0(\R^n\backslash \partial\Omega)\oplus \dot{V}_{0,\Dcal}(\R^n\backslash \partial\Omega))'$ into $\dot{H}^1_0(\R^n\backslash \partial\Omega)\oplus \dot{V}_{0,\Dcal}(\R^n\backslash \partial\Omega)$ such that
\[\big\langle \dot{\Rcal}u,v\big\rangle_{\dot{H}^1(\R^n\backslash\partial\Omega)}=\left\langle u,v\right\rangle_{(\dot{H}^1(\R^n\backslash\partial\Omega))',\dot{H}^1(\R^n\backslash\partial\Omega)}\]
for all $u\in (\dot{H}^1_0(\R^n\backslash \partial\Omega)\oplus \dot{V}_{0,\Dcal}(\R^n\backslash \partial\Omega))'$ and $v\in \dot{H}^1_0(\R^n\backslash \partial\Omega)\oplus \dot{V}_{0,\Dcal}(\R^n\backslash \partial\Omega)$. With a similar interpretation as before, we obtain the following.

\begin{lemma}\label{L:repDLo}
Let $\Omega$ be two-sided $\dot{H}^1$-admissible. Then $\dot{\Dcal}=-\dot{\Rcal}\circ (\dot\partial_\nu)^\ast$.
\end{lemma}

\begin{proposition}\label{P:repDLo}
Let $\Omega\subset \mathbb{R}^n$, $n\geq 2$, be a bounded Lipschitz domain. Then for any $f\in \dot{H}^{1/2}(\partial\Omega)$, the trace jump readjusted representative of $\dot{\Dcal}f$ differs only by an additive constant from $-\int_{\partial\Omega}\frac{\partial}{\partial\nu_y} K(\cdot -y)f(y)\:\mu(\dy)$ on $\R^n\backslash\partial\Omega$. 
\end{proposition}

\begin{remark}
Proposition~\ref{P:repDLo} states that, in the case of a bounded Lipschitz domain $\Omega\subset \mathbb{R}^n$, $n\geq 2$, endowed with the surface measure on the boundary, the trace jump readjusted variant of the double layer potential operator $\dot\Dcal$ can be represented using the classical formula~\cite{AMMARI-2004, VERCHOTA-1984}, up to an additive constant.
\end{remark}

\section{Neumann-Poincar\'e operators for admissible domains}\label{Sec:Neumann-Poincaré}

For smooth or Lipschitz domains Neumann-Poincar\'e operators are well understood \cite{KHAVINSON-2006, ANDO-2021,KANG-2022}, and their spectral properties are known to reflect the regularity of $\partial\Omega$: They are compact in the $C^1$-case \cite{FABES-1978} but have a continuous non-real spectrum if $\bord$ has a corner~\cite{ANDO-2021, CARLEMAN-1916}. Refined mapping properties of Neumann-Poincar\'e operators for bounded Lipschitz domains were studied in \cite{mitrea_generalization_2008}.

We define Neumann-Poincar\'e operators $\Kcal$ and $\dot{\Kcal}$ on the trace spaces $\Hp$ and $\Hpo$ respectively. As in the Lipschitz case \cite{STEINBACH-2001}, they are bounded operators and satisfy the well-known jump relations, see Theorems~\ref{T:K} and~\ref{T:Kdot}. Following~\cite{STEINBACH-2001}, we define boundary layer potential operators associated with $(\Delta-1)$ respectively $\Delta$ in \eqref{E:VW} respectively \eqref{E:VWo} below. Using those operators, we obtain generalizations of the classical Calderón projectors, see Theorems~\ref{NP-Sym} and~\ref{ThCalderonProjo}.

As in~\cite{STEINBACH-2001}, the boundary layer potentials give rise to equivalent Hilbert space norms on the trace spaces and their duals. Those norms make the single and double layer potentials isometries and the operators $(\pm\frac{1}{2}I+\Kcal)$ coercive contractions, see Lemma~\ref{L:coercive} and Theorem~\ref{T:makeiso}.

\subsection{Neumann-Poincar\'e operators}\label{Subsec:N-P}

Suppose that $\Omega\subset \R^n$ is a two-sided $H^1$-admissible domain. We refer to the bounded linear operator $\Kcal:\Hp\to \Hp$, defined by
\[\Kcal:=\frac12(\Tr_i+\Tr_e)\circ \Dcal,\]
as the \emph{Neumann-Poincar\'e operator} for \eqref{E:transmiDL}; here $\Tr_i:=\TTr_i|_{V_1(\Omega)}$ and $\Tr_e:=\TTr_e|_{V_1(\R^n\backslash\overline\Omega)}$. By $\Kcal^\ast:\Hm\to\Hm$, we denote its dual. 

In the sequel, we use the symbol $I$ to denote the identity operator; the space on which it acts will be clear from the context.

The following identities generalize results well-known in the Lipschitz case, see \cite[Theorem 2.4]{AMMARI-2004} or \cite{VERCHOTA-1984}.

\begin{theorem}\label{T:K}
Let $\Omega$ be two-sided $H^1$-admissible. Then
\begin{enumerate}
\item[(i)]  $\Tr_i\circ \Dcal=-\frac12 I+\Kcal$ and $\Tr_e\circ \Dcal=\frac12 I+\Kcal$.
\item[(ii)]  $\partial_{\nu,i}\circ \Scal=\frac12 I+\Kcal^\ast$ and $\partial_{\nu,e}\circ \Scal=-\frac12 I+\Kcal^\ast$. In particular, 
\begin{equation}\label{E:Kast}
\Kcal^\ast=\frac12(\partial_{\nu,i}+\partial_{\nu,e})\circ \Scal.
\end{equation}
\end{enumerate}
\end{theorem}

Recall the definitions \eqref{EqGreenInt} and \eqref{E:extnormalder} of the interior and exterior normal derivatives. They give the Gauss-Green formulas
\begin{equation}\label{E:GGi}
\left\langle \partial_{\nu,i} u, \TTr_i v\right\rangle_{\Hm,\Hp}=\int_\Omega uv\:\dx+\int_\Omega \nabla u\cdot\nabla v\:\dx
\end{equation}
for all $u\in V_1(\Omega)$ and $v\in H^1(\Omega)$, and 
\begin{equation}\label{E:GGe}
\left\langle \partial_{\nu,e} u, \TTr_e v\right\rangle_{\Hm,\Hp}=-\int_{\mathbb{R}^n\setminus\overline{\Omega}} uv\:\dx-\int_{\mathbb{R}^n\setminus\overline{\Omega}} \nabla u\cdot\nabla v\:\dx
\end{equation}
for all $u\in V_1(\mathbb{R}^n\setminus\overline{\Omega})$ and $v\in H^1(\mathbb{R}^n\setminus\overline{\Omega})$. We prove Theorem \ref{T:K}.

\begin{proof} Statement (i)  follows from the definitions of $\Dcal$. To see (ii), note that for any $f\in \Hp$ and $g\in \Hm$ we have 
\begin{align*}
0&=\left\langle \Dcal f,\Scal g\right\rangle_{H^1(\R^n\backslash \partial\Omega)}\\&=\langle\partial_{\nu,i}\Scal g,\Tr_i\:\Dcal f\rangle_{\Hm,\Hp}-  \langle\partial_{\nu,e}\Scal g,\Tr_e\:\Dcal f\rangle_{\Hm,\Hp}
\end{align*}
by Lemma \ref{L:spaces} (iv), \eqref{E:GGi} and \eqref{E:GGe}, and that adding this zero to 
\[\langle\partial_{\nu,i}\Scal g,\Tr_e\:\Dcal f\rangle_{\Hm,\Hp}-  \langle\partial_{\nu,e}\Scal g,\Tr_i\:\Dcal f\rangle_{\Hm,\Hp}\]
respectively subtracting it gives 
\begin{multline}
2\left\langle g, \Kcal f\right\rangle_{\Hm,\Hp}=\langle \llbracket\partial_\nu\Scal g\rrbracket,\Tr_i\Dcal f+\Tr_e\Dcal f\rangle_{\Hm,\Hp}\notag\\
=\langle \partial_{\nu,i}\Scal g+\partial_{\nu,e}\Scal g,-\left\llbracket\Tr\Dcal f\right\rrbracket\rangle_{\Hm,\Hp}=\langle (\partial_{\nu,i}+\partial_{\nu,e})\circ \Scal g,f\rangle_{\Hm,\Hp}.
\end{multline}
This yields \eqref{E:Kast}, and using the definition of $\Scal$, item (ii) follows.
\end{proof}

For a two-sided $\dot{H}^1$-admissible domain $\Omega\subset \R^n$, similar observations can be made. We refer to the bounded linear operator $\dot{\Kcal}:\Hpo\to \Hpo$, defined by
\[\dot{\Kcal}:=\frac12(\dTr_i+\dTr_e)\circ \dot{\Dcal},\]
as the \emph{Neumann-Poincar\'e operator} for \eqref{E:transmiDLo}. By $\dot{\Kcal}^\ast:\Hmo\to\Hmo$, we denote its dual.

\begin{theorem}\label{T:Kdot}
Let $\Omega$ be two-sided $\dot{H}^1$-admissible. Then
\begin{enumerate}
\item[(i)]  $\dTr_i\circ \dot{\Dcal}=-\frac12 I+\dot{\Kcal}$ and $\dTr_e\circ \dot{\Dcal}=\frac12 I+\dot{\Kcal}$.
\item[(ii)]  $ \dot{\partial}_{\nu,i}\circ \dot{\Scal}=\frac12 I+\dot{\Kcal}^\ast$ and $\dot\partial_{\nu,e}\circ \dot{\Scal}=-\frac12 I+\dot{\Kcal}^\ast$.
In particular, 
\[\dot{\Kcal}^\ast=\frac12(\dot{\partial}_{\nu,i}+\dot\partial_{\nu,e})\circ \dot{\Scal}=\frac12( \dot{\partial}_{\nu,i}+\dot\partial_{\nu,e})\circ \overline{\Scal},\]
where $\overline{\Scal}$ is as in Remark \ref{R:readjustsingle}.
\end{enumerate}
\end{theorem}

\begin{remark}\label{Rem:SameK*}
If $\Omega$ is a bounded Lipschitz domain in $\mathbb{R}^n$, {\color{black} $n\geq 2$}, and $\bord$ is endowed with the surface measure $\sigma$, then by Corollary ~\ref{Cor:kernelSLPd} and Remark~\ref{R:strictcaseSLdot} the operator $\overline{\Scal}$ is the usual single layer potential operator, up to an additive constant.
As pointed out in Remark \ref{R:normalderLip} for $\dot{\partial}_{\nu,i}$, also the $\dot{H}^{-1/2}(\partial\Omega)$-valued linear operators $\dot{\partial}_{\nu,i}$ and $\dot{\partial}_{\nu,e}$ are the usual weak normal derivatives.
As a consequence, the above definition of the operator $\dot\Kcal^*$ coincides with the usual one for bounded Lipschitz domains, see ~\cite[Theorem 2.4]{AMMARI-2004} {\color{black} (there denoted by $\mathcal{K}_D^\ast$)} or ~\cite{VERCHOTA-1984}  {\color{black} (there denoted by $K^\ast$)}. 
\end{remark}

\subsection{Boundary layer potentials and Calder\'on projectors}\label{Subsec:CaldProj}

Let $\Omega$ be two-sided $H^1$-admissible. By Lemma \ref{L:spaces}, Corollary \ref{C:DLbijective} and Corollary \ref{C:SLbijective} the \emph{boundary single layer potential operator} $\Vcal:\Hm\to \Hp$ and the \emph{\color{black}hypersingular operator} $\Wcal:\Hp\to \Hm$, defined by  
\begin{equation}\label{E:VW}
\Vcal:=\Tr\circ\, \Scal\quad \text{respectively}\quad \Wcal:=-\partial_\nu\circ \Dcal,
\end{equation}
are bounded linear bijections with bounded inverses.

\begin{remark}
In a classical context, the trace $\Vcal g$ on $\partial\Omega$ of the single layer potential $\Scal g$ can be defined as a pointwise limit, while the (co-)normal derivative $\Wcal f$ of the double layer potential $\Dcal f$ is a hypersingular integral, which, for smooth enough $\partial \Omega$ and $f$, can be expressed as a {\color{black}Hadamard finite part integral}. See for instance \cite[Section 2]{STEINBACH-2001} for details; note that our notation differs slightly from the one used there.
\end{remark}

Now let ${\mathcal{M}}:\Hp\times\Hm\to \Hp\times\Hm$ be the linear operator defined by 
\begin{equation*}
{\mathcal{M}}\;:=\;
\begin{pmatrix}
-\Kcal & \Vcal\\[1.5ex]
\Wcal & \Kcal^* 
\end{pmatrix}
\end{equation*}
and set
\begin{equation*}
{\mathcal{C}}_i:=\frac{1}{2}I+{\mathcal{M}}\qquad\mbox{and}\qquad{\mathcal{C}}_e:=\frac{1}{2}I-{\mathcal{M}}.
\end{equation*}
One refers to the linear operators ${\mathcal{C}}_i$ and ${\mathcal{C}}_e$ as the \emph{interior} respectively \emph{exterior Calder\'on projector} for $\Omega$. We obtain generalizations of known symmetrization formulas referred to as \emph{Calder\'on relations}, see \cite[Lemma 1.2.4]{HSIAO-2008}, \cite[Theorem 3.1.3]{NEDELEC-2001} or \cite[Proposition 5.1]{STEINBACH-2001}.

\begin{theorem}\label{NP-Sym}
Let $\Omega$ be two-sided $H^1$-admissible. The operators $\mathcal{C}_i$ and $\mathcal{C}_e$ are continuous projectors and satisfy ${\mathcal{C}}_i+{\mathcal{C}}_e=I$. Moreover, we have $\mathcal{M}^2=\frac{1}{4}I$, that is, 
\begin{equation}
\begin{cases}
\Kcal\Vcal=\Vcal\Kcal^*,\\[2ex]
\Wcal\Kcal=\Kcal^*\Wcal,
\end{cases}
\qquad\mbox{and}\qquad
\begin{cases}
\displaystyle\Kcal^2 +\Vcal\Wcal= \frac{1}{4}I,\\
\displaystyle(\Kcal^*)^2 +\Wcal\Vcal = \frac{1}{4}I.
\end{cases}
\end{equation}
\end{theorem}

\begin{proof} Continuity and the first identity are clear. To see that $\mathcal{C}_i^2=\mathcal{C}_i$, suppose that $f\in\Hp$ and $g\in\Hm$ are given, and that $u=\Scal g-\Dcal f$ is the unique weak solution to \eqref{EqTrProblem} with $k=1$, with $\llbracket \partial_{\nu} u\rrbracket=g$ and with $\llbracket\Tr u\rrbracket=f$ in place of $\llbracket\Tr u\rrbracket=-f$. Then \eqref{E:VW} and Theorem \ref{T:K} give
\begin{equation}\label{E:evalCi}
\mathcal{C}_i
\begin{pmatrix}
f\\[1ex] g
\end{pmatrix}
=
\begin{pmatrix}
\displaystyle\Vcal g-\left(-\frac{1}{2}I+\Kcal\right)f\\[1ex]
\displaystyle\left(\frac{1}{2}I+\Kcal^*\right)g+\Wcal f
\end{pmatrix}
=
\begin{pmatrix}
\Tr\Scal g-\Tr_i\Dcal f\\[1ex] \displaystyle \partial_{\nu,i}\Scal g-\displaystyle\partial_\nu\Dcal f
\end{pmatrix}
=
\begin{pmatrix}
\Tr_i u\\[1ex] \displaystyle\partial_{\nu,i}u
\end{pmatrix}.
\end{equation}
Now let $v$ be the unique weak solution to problem \eqref{EqTrProblem} with $k=1$, with $\llbracket\Tr u\rrbracket=\Tr_i u$ and $\llbracket \partial_{\nu} u\rrbracket=\partial_{\nu,i}u$. Since $u\mathds{1}_\Omega$ is also a weak solution to this problem, it follows that 
$v=u\mathds{1}_\Omega$ and consequently
\begin{equation*}
\mathcal{C}_i^2
\begin{pmatrix}
 f\\[1ex] g
\end{pmatrix}
=\mathcal{C}_i
\begin{pmatrix}
\Tr_i u\\[1ex] \displaystyle\partial_{\nu,i}u
\end{pmatrix}
=
\begin{pmatrix}
\Tr_i v\\[1ex] \displaystyle\partial_{\nu,i}v
\end{pmatrix}
=
\begin{pmatrix}
\Tr_i u\\[1ex] \displaystyle\partial_{\nu,i}u
\end{pmatrix}
=
\mathcal{C}_i
\begin{pmatrix}
f\\[1ex] g
\end{pmatrix}.
\end{equation*}
In the same way, we prove $\mathcal{C}_e^2=\mathcal{C}_e$. The remaining identities then follow.
\end{proof}

\begin{remark}
The unique weak solution $u=\Scal g-\Dcal f$ of \eqref{EqTrProblem} with $k=1$ and boundary data $f\in\Hp$ and  
$g\in \Hm$ is zero on $\R^n\backslash \overline{\Omega}$ if and only if $\Tr_e u=0$ and $\partial_{\nu,e}u=0$. By \eqref{E:evalCi} this is the case if and only if $\mathcal{C}_i\binom{f}{g}=\binom{f}{g}$, which happens if and only if $(f,g)$ is an element of the graph of the Poincar\'e-Steklov operator $\Ascr_1$ for $\Omega$. As in \cite[Section 5]{STEINBACH-2001}, an evaluation of \eqref{E:evalCi} gives
\begin{equation}\label{E:PSidentity}
\Ascr_1=\mathcal{V}^{-1}\Big(\frac12 I+\Kcal\Big)=\mathcal{W}+\Big(\frac12 I+\Kcal^\ast\Big)\mathcal{V}^{-1}\Big(\frac12 I+\Kcal\Big).
\end{equation}
\end{remark}

Now suppose that $\Omega$ is two-sided $\dot{H}^1$-admissible. By Lemma \ref{L:spaceso}, Corollary \ref{C:DLbijectiveo} and Corollary \ref{C:SLbijectiveo}, the operators $\dot{\Vcal}:\Hmo\to \Hpo$ and $\dot{\Wcal}:\Hpo\to \Hmo$, defined by  
\begin{equation}\label{E:VWo}
\dot{\Vcal}:=\dTr\circ \dot{\Scal}\quad \text{and}\quad \dot{\Wcal}:=-\dot\partial_\nu\circ \dot{\Dcal},
\end{equation}
are bounded linear bijections with bounded inverses. Let us define the operator $\dot{\mathcal{M}}:\Hpo\times\Hmo\to \Hpo\times\Hmo$ by 
\[\dot{\mathcal{M}}\;:=\;
\begin{pmatrix}
-\dot{\Kcal} & \dot{\Vcal}\\[1.5ex]
\dot{\Wcal} & \dot{\Kcal}^*
\end{pmatrix}
\]
and set 
\begin{equation*}
\dot{\mathcal{C}}_i:=\frac{1}{2}I+\dot{\mathcal{M}}\qquad\mbox{and}\qquad\dot{\mathcal{C}}_e:=\frac{1}{2}I-\dot{\mathcal{M}}.
\end{equation*}

The same arguments as before give the following counterpart of Theorem \ref{NP-Sym}.

\begin{theorem}\label{ThCalderonProjo}
Let $\Omega$ be two-sided $\dot{H}^1$-admissible. The operators $\dot{\mathcal{C}}_i$ and $\dot{\mathcal{C}}_e$ are continuous projectors and satisfy $\dot{\mathcal{C}}_i+\dot{\mathcal{C}}_e=I$. In addition, $\dot{\mathcal{M}}^2=\frac14 I$, that is,
\begin{equation}
\begin{cases}
\dot{\Kcal}\dot{\Vcal}=\dot{\Vcal}\dot{\Kcal}^*,\\[2ex]
\dot{\Wcal}\dot{\Kcal}=\dot{\Kcal}^*\dot{\Wcal},
\end{cases}
\qquad\mbox{and}\qquad
\begin{cases}
\displaystyle\dot{\Kcal}^2 +\dot{\Vcal}\dot{\Wcal}= \frac{1}{4}I,\\
\displaystyle(\dot{\Kcal}^*)^2 +\dot{\Wcal}\dot{\Vcal} = \frac{1}{4}I.
\end{cases}
\end{equation}
\end{theorem}

\subsection{Invertibility and isometries}\label{Subsec:IsomInvert}

Using the operators in \eqref{E:VW} as metrics, we can introduce 
\[\|\cdot\|^2_{\Hm, \Vcal}:=\langle \cdot,\Vcal\cdot\rangle_{\Hm,\Hp}, \qquad \|\cdot\|^2_{\Hp,\Vcal^{-1}}:=\langle\Vcal^{-1}\cdot,\cdot\rangle_{\Hm,\Hp},\]
where $\Vcal^{-1}$ denotes the inverse of $\Vcal$, and 
\[\|\cdot\|^2_{\Hp,\Wcal}:=\langle\Wcal\cdot,\cdot\rangle_{\Hm,\Hp}.\]
These quadratic forms give equivalent Hilbert space norms on $\Hm$ and $\Hp$, respectively, as shown in the following lemma.

\begin{lemma}\label{L:coercive}
Let $\Omega$ be two-sided $H^1$-admissible. 
\begin{enumerate}
\item[(i)] The operator $\Scal$ is an isometry with respect to $\|\cdot\|_{\Hm,\Vcal}$, and 
 $\Dcal$ is an isometry with respect to $\|\cdot\|_{\Hp,\Wcal}$.
\item[(ii)] There are constants $\alpha,\beta>0$ such that  
\begin{gather*}
\frac{1}{\alpha}\|g\|^2_{\Hm}\le\|g\|^2_{\Hm,\Vcal}\le \alpha\|g\|^2_{\Hm},\quad g\in \Hm,\\
\frac{1}{\alpha}\|f\|^2_{\Hp}\le\|f\|^2_{\Hp,\Vcal^{-1}}\le \alpha\|f\|^2_{\Hp},\quad f\in \Hp,
\end{gather*}
and 
\[\frac{1}{\beta}\|f\|^2_{\Hp}\le\|f\|^2_{\Hp, \Wcal}\le \beta\|f\|^2_{\Hp},\quad f\in \Hp.\]
\end{enumerate}
\end{lemma}

\begin{proof}
Using \eqref{E:VW} together with Corollaries  \ref{C:orthoV}, \ref{C:DLbijective} and Corollary \ref{C:SLbijective}, statement (i) follows. Statement (ii) follows using (i), together with the boundedness of the operators $\Scal$ and $\Dcal$ and their inverses.
\end{proof}

As in the classical case, one can observe invertibility properties depending on a spectral parameter $\lambda\in \C$. The following observation for the operators is similar to \cite[Lemma 2.4]{AMMARI-2004}.
\begin{lemma}\label{L:Kastinjective}
Let $\Omega$ be two-sided $H^1$-admissible and $\lambda\in\R$ with $|\lambda|\geq \frac12$. Then $\lambda I+\Kcal^\ast:\Hm\to \Hm$ is injective.
\end{lemma}

\begin{proof}
Assume that there is some nonzero $g\in \ker(\lambda I+\Kcal^\ast)$. Then $\Scal g\neq 0$ by Corollary \ref{C:SLbijective}, and since $\Scal g$ is $1$-harmonic in $\R^n\backslash \partial\Omega$, both $\Scal g|_\Omega$ and $\Scal g|_{\R^n\backslash\overline{\Omega}}$ must be nonzero, otherwise the null jump in trace would yield $\Scal g=0$. By Theorem \ref{T:K}, together with \eqref{E:GGi} and \eqref{E:GGe}, this implies that 
\[A:=\Big(\frac12-\lambda\Big)\left\langle g,\Vcal g\right\rangle_{\Hm,\Hp}=\Big\langle \Big(\frac12 I+\Kcal^\ast\Big)g,\Tr\Scal g\Big\rangle_{\Hm,\Hp}>0 \]
and 
\[B:=\Big(\frac12+\lambda\Big)\left\langle g,\Vcal g\right\rangle_{\Hm,\Hp}=-\Big\langle \Big(-\frac12 I+\Kcal^\ast\Big)g,\Tr\Scal g\Big\rangle_{\Hm,\Hp}>0. \]
For $|\lambda|=\frac12$ this is impossible. For $|\lambda|>\frac12$ we can use the fact that $\left\langle g,\Vcal g\right\rangle_{\Hm,\Hp}=\|g\|^2_{\Hm, \Vcal}>0$ by Lemma \ref{L:coercive}. This gives 
$\lambda=\frac12\frac{B-A}{B+A}\in \big[-\frac12,\frac12\big]$,
which again is impossible.
\end{proof}

For the operators $\pm\frac{1}{2}I+\Kcal$ we have an analog of \cite[Theorem 5.1]{STEINBACH-2001}.

\begin{theorem}\label{T:makeiso}
Let $\Omega$ be two-sided $H^1$-admissible. For any $f\in \Hp$ we have
\[(1-c)\|f\|_{\Hp,\Vcal^{-1}}\le \Big\|\Big(\pm\frac{1}{2}I+\Kcal\Big)f\Big\|_{\Hp,\Vcal^{-1}}\le c\|f\|_{\Hp,\Vcal^{-1}},\]
where 
\[c=\frac{1}{2}+\sqrt{\frac{1}{4}-\frac{1}{\alpha\beta}}<1\]
with constants $\alpha$ and $\beta$ as in Lemma \ref{L:coercive} chosen large enough so that $\alpha\beta>4$. In particular, 
the operators $\pm\frac{1}{2}I+\Kcal:\Hp\to\Hp$ are isomorphisms.
\end{theorem}

\begin{proof}
We can proceed similarly as in \cite[p. 744]{STEINBACH-2001}; note that by \eqref{E:PSidentity}, 
\begin{align}
\Big\|\Big(\frac12 I+\Kcal\Big)f\Big\|_{\Hp,\Vcal^{-1}}^2&=\Big\langle \Vcal^{-1}\Big(\frac12 I+\Kcal\Big)f,\Big(\frac12 I+\Kcal\Big)f\Big\rangle_{\Hm,\Hp}\notag\\
&=\left\langle\Ascr_1 f,f\right\rangle_{\Hm,\Hp}-\left\langle \Wcal f,f\right\rangle_{\Hp,\Hm}\notag
\end{align}
for any $f\in\Hp$; the first summand on the right-hand side is
\[\left\langle\Vcal\Ascr_1f,f\right\rangle_{\Hp,\Vcal^{-1}}\leq \Big\|\Big(\frac12 I+\Kcal\Big)f\Big\|_{\Hp,\Vcal^{-1}}\left\|f\right\|_{\Hp,\Vcal^{-1}},\]
and by Lemma \ref{L:coercive}, the second is bounded below by $\frac{1}{\alpha\beta}\left\|f\right\|_{\Hp,\Vcal^{-1}}^2$.
\end{proof}

\begin{remark}
The contractivity of $(\pm\frac{1}{2}I+\Kcal)$ ensures the convergence of the associated Neumann series \cite{NEUMANN-1870}: Recovering the jump in trace $f\in \Hp$ of a transmission solution for $(\Delta-1)$ with zero jump in normal derivative from its exterior trace $\varphi\in \Hp$ amounts to solving the boundary integral equation of the second kind $-\frac{1}{2}f-\Kcal f=\varphi$
in $\Hp$, and its unique solution $f$ is given by the Neumann series
$f=\sum_{\ell=0}^{+\infty} \left(\frac{1}{2}I+\Kcal\right)^\ell \varphi$ 
which converges in $\Hp$.
Similarly, one can invert the operator $-\frac{1}{2}I+\Kcal^*$ on $\Hm$ to recover the jump in normal derivative of a solution with no jump in trace from its exterior normal derivative, see \cite[Theorem 3.2, p.741]{STEINBACH-2001}.
\end{remark}

Based on \eqref{E:VWo}, one can similarly introduce equivalent Hilbert norms $\left\|\cdot\right\|_{\Hmo, \dot{\Vcal}}$, $\|\cdot\|_{\Hpo,\dot{\Vcal}^{-1}}$ and $\|\cdot\|_{\Hpo,\dot{\Wcal}}$ on $\Hmo$ and $\Hpo$, and proceed as before to obtain the following.

\begin{theorem}\label{T:makeisoo}
Let $\Omega$ be two-sided $\dot{H}^1$-admissible. Then counterparts of Lemma \ref{L:coercive} (i) and (ii),
Lemma \ref{L:Kastinjective} and Theorem \ref{T:makeiso} hold for $\dot{\Vcal}$, $\dot{\Wcal}$, $\dot{\Scal}$, $\dot{\Dcal}$, $\dot{\Kcal}^\ast$ and $\dot{\Kcal}$.
\end{theorem}

A perturbation argument gives certain spectral properties of the Neumann-Poincaré operators in the spirit of well-known results in the Lipschitz case \cite{CHANG-2008, ESCAURIAZA-1992,FABES-1978, MITREA-2002, PERFEKT-2017,VERCHOTA-1984}.

\begin{theorem}\label{ThSpecNP}
Let $\Omega$ be two-sided $H^1$-admissible. For $\lambda\in\C$, if $|\lambda-\frac12|\geq 1$ or $|\lambda+\frac12|\geq 1$, then
the operators $\lambda I+\Kcal$ and $\lambda I+\Kcal^*$ are invertible on $\Hp$ and $\Hm$ respectively. The complex spectra of $\Kcal$ and $\Kcal^*$ lie in the intersection of the open disks with radius $1$ and centers $\pm\frac12$, the real spectra of $\Kcal$ and $\Kcal^*$ are included in $(-\frac12,\frac12)$.
\end{theorem}

\begin{proof}
For any $\lambda\neq-\frac{1}{2}$ we have 
$\lambda I+\Kcal=\left(\lambda+\frac{1}{2}\right)I+\left(-\frac{1}{2}I+\Kcal\right)$.
By Theorem~\ref{T:makeiso}, $\|-\frac{1}{2}I+\Kcal\|< 1$. Hence, by~\cite[Theorem 1.2.9]{DAVIES-2007}, $\lambda I+\Kcal$ is invertible if $|\lambda+\frac{1}{2}|\ge1$. Proceeding in the same way with $\frac{1}{2}I+\Kcal$ instead, we find that $\lambda I+\Kcal$ is invertible if $|\lambda-\frac{1}{2}|\ge1$. By Banach's closed range theorem and Lemma~\ref{L:Kastinjective}, the same holds for $\Kcal^*$, see~\cite[Theorem 8.1.5]{DAVIES-2007}.
\end{proof}

\begin{figure}[h]
\centering
\begin{tikzpicture}
\draw [->] (-3.7,0) -- (3.7,0) ;
\draw [->] (0,-2.7) -- (0,2.7) ;
\draw (0,3) node {$\mathrm{Im}(\lambda)$} ;
\draw (4.3,0) node {$\mathrm{Re}(\lambda)$} ;
\draw (1,0) node[below right] {$\frac12$} ;
\draw (1,-0.1) -- (1,0.1) ;
\draw (-1,0) node[below left] {$-\frac12$} ;
\draw (-1,-0.1) -- (-1,0.1) ;
\draw (3,0) node[below right] {$\frac32$} ;
\draw (3,-0.1) -- (3,0.1) ;
\draw (-3,0) node[below left] {$-\frac32$} ;
\draw (-3,-0.1) -- (-3,0.1) ;
\draw (-0.4,{sqrt(3)+0.45}) node {$i\frac{\sqrt3}2$} ;
\draw (-0.1,{sqrt(3)}) -- (0.1,{sqrt(3)}) ;
\draw (-0.5,{-sqrt(3)-0.4}) node {$-i\frac{\sqrt3}2$} ;
\draw (-0.1,{-sqrt(3)}) -- (0.1,{-sqrt(3)}) ;
\fill [color=gray, opacity=0.2] plot [domain=-0.5:0, samples=300, scale=2] (\x,{sqrt(1-(\x-0.5)^2)})
-- plot [domain=0:0.5, samples=100, scale=2] (\x,{sqrt(1-(\x+0.5)^2)})
-- plot [domain=0.5:0, samples=100, scale=2] (\x,{-(sqrt(1-(\x+0.5)^2))}) 
-- plot [domain=0:-0.5, samples=100, scale=2] (\x,{-sqrt(1-(\x-0.5)^2)}) -- cycle ;
\draw [densely dotted] (-1,0) circle (2) ;
\draw [densely dotted] (1,0) circle (2) ;
\end{tikzpicture}
\caption{Graphic representation of the values of $\lambda$ for which Theorems~\ref{ThSpecNP} and~\ref{ThSpecNPo} apply, which are all $\lambda\in\C$ outside the gray area (or on its boundary). The dotted lines correspond to the circles of center $\pm\frac12$ and radius $1$.}
\label{fig:specK}
\end{figure}

Analogous arguments give a parallel result in the homogeneous case.

\begin{theorem}\label{ThSpecNPo}
Let $\Omega$ be two-sided $\dot{H}^1$-admissible. For $\lambda\in\C$, if $|\lambda-\frac12|\geq 1$ or $|\lambda+\frac12|\geq 1$, then
the operators $\lambda I+\dot\Kcal$ and $\lambda I+\dot\Kcal^*$ are invertible on $\Hpo$ and $\Hmo$ respectively. The complex spectra of $\dot\Kcal$ and $\dot\Kcal^*$ lie in the intersection of the open disks with radius $1$ and centers $\pm\frac12$, 
the real spectra of $\dot\Kcal$ and $\dot\Kcal^*$ are included in $(-\frac12,\frac12)$.
\end{theorem}

The values of $\lambda\in\C$ for which Theorems~\ref{ThSpecNP} and~\ref{ThSpecNPo} hold are represented in Figure~\ref{fig:specK}.

\section{Applications to imaging}\label{Sec:Imaging}


We generalize three results from \cite{AMMARI-2004}. There they were shown for Lipschitz domains; here we establish them for two-sided admissible domains $\Omega$. Theorem~\ref{ThImagRepr} is a boundary representation formula for the unique weak solution to a specific transmission problem for the Laplacian and generalizes~\cite[Theorem 2.17]{AMMARI-2004}, see also~\cite{KANG-1996,KANG-1999}. Theorem~\ref{monotonous} is a uniqueness results for subdomain identification through a single boundary measurement in the monotone case, partly generalizing \cite[Theorem A.7, p.220]{AMMARI-2004}, see also
\cite{BELLOUT-1988}. Finally, we prove a similar result for disks in $\R^2$ in Theorem~\ref{ThImagDisks}, generalizing~\cite[Theorem A]{KANG-1996} to the case of a surrounding extension domain and with boundary data understood in the sense of the dual trace space $\Hmo$ instead of $L^2_0(\del \Omega)$.

\subsection{Representation formula}\label{SubsecImReprF}

Let $\Omega$ and $D$ be two two-sided $\dot{H}^1$-admissible domains in $\R^n$ such that $D\subset\subset\Omega$, see Figure~\ref{Fig:Imaging}. We write $\dTTr_{\bord,i}$ for the interior trace operator with respect to $\Omega$.

Let $\chi\in C_c^\infty(\Omega)$ be such that $0\leq \chi\leq 1$ and $\chi\equiv 1$ on a neighbourhood of $\overline{D}$. Given $k\in(0,1)\cup(1,+\infty)$ and $g\in \Hmo$, we call $u\in \dot{H}^1(\Omega)$ a \emph{weak solution} of the Neumann problem 
\begin{equation}\label{incl}
\begin{cases}
\displaystyle\nabla\cdot\Big(\big(1+(k-1)\mathds{1}_D\big)\nabla u\Big)=0 &\mbox{on }\Omega,\\
\displaystyle \frac{\partial_iu}{\partial \nu}\Big|_{\bord}=g
\end{cases}
\end{equation}
\emph{in the $\dot{H}^1$-sense} if 
\begin{equation}\label{E:weakreconstruct}
\int_\Omega{\big(1+(k-1)\mathds{1}_D\big)\nabla u\cdot\nabla v\,\dx}=0
\end{equation}
for all $v\in C_c^\infty(\Omega)$ and $\left\langle u,(1-\chi)v\right\rangle_{\dot{H}^1(\Omega)}=\big\langle g,\dTTr_{\bord,i} v\big\rangle_{\Hmo,\Hpo}$ for all $v\in \dot{H}^1(\Omega)$. 

Since $\min(1,k)\leq 1+(k-1)\mathds{1}_D\leq \max(1,k)$, the existence of a unique weak solution $u$ of \eqref{E:weakreconstruct} is clear from the Riesz representation theorem and by the following remark, which also explains 
that the unique weak solution of \eqref{E:weakreconstruct} does not depend on the choice of $\chi$.

\begin{remark}\label{R:justify}\mbox{}
\begin{enumerate}
\item[(i)] We can extend the definition \eqref{EqGreenInto} of the weak interior normal derivative $\frac{\dot{\partial}_i}{\partial \nu}|_{\partial\Omega}$ with respect to $\Omega$ to the space
\[\dot{H}^1_{\Delta}(\Omega\backslash \overline{D}):=\{u\in \dot{H}^1(\Omega\backslash \overline{D})\;|\;\Delta u\in L^2(\Omega\backslash \overline{D})\}.\]
Given $u\in \dot{H}^1_{\Delta}(\Omega\backslash\overline{D})$, there is a unique element $g\in \Hmo$ such that 
\begin{equation}\label{EqGreenext}
\big\langle g,\dTTr_{\bord,i} v\big\rangle_{\Hmo,\Hpo}=\int_\Omega (\Delta u) (1-\chi)v\: {\rm d}x+\int_\Omega \nabla u \nabla ((1-\chi)v)\: {\rm d}x
\end{equation} 
for all $v\in \dot{H}^1(\Omega)$.
We obviously have $\dTTr v=\dTTr (1-\chi)v$, and $(1-\chi)v$ vanishes on a neighbourhood of $\overline{D}$. The right-hand side of \eqref{EqGreenext} is bounded by 
\[\|\chi\|_{C^1(\Omega)}\Big(\|\Delta u\|_{L^2(\Omega\backslash \overline{D})}+\|u\|_{\dot{H}^1(\Omega\backslash \overline{D})}\Big)\|v\|_{\dot{H}^1(\Omega\backslash \overline{D})}.\]
It is not difficult to see that $g$ does not depend on the particular choice of $\chi$. Since \eqref{EqGreenext} extends \eqref{EqGreenInto}, we could denote $g$ again by $\frac{\dot{\partial}_i u}{\partial \nu}$. 

\item[(ii)] We agree to write $\dot{V}_0(\Omega\backslash \overline{D})$ for the orthogonal complement of $C_c^\infty(\Omega\backslash \overline{D})$ in $\dot{H}^1(\Omega\backslash \overline{D})$. (It is not hard to show that under the stated assumptions also $\Omega\backslash \overline{D}$ is $\dot{H}^1$-admissible, so that this agreement is consistent with our former notation.) The restriction $\dot{\partial}_{\nu,i}|_{\partial\Omega}$ of $\frac{\dot{\partial}_i }{\partial \nu}$ to  $\dot{V}_0(\Omega\backslash \overline{D})\subset \dot{H}^1_{\Delta}(\Omega\backslash \overline{D})$ is a bounded linear operator. By \eqref{E:weakreconstruct} a weak solution $u$ of \eqref{incl} is an element of $\dot{V}_0(\Omega\backslash \overline{D})$. If $u\in \dot{V}_0(\Omega\backslash \overline{D})$ and $\chi'$ is another function with the properties specified for $\chi$, we have 
$\left\langle u,(\chi-\chi')v\right\rangle_{\dot{H}^1(\Omega)}=0$ for all $v\in \dot{H}^1(\Omega)$ by orthogonality. Therefore the Neumann boundary condition does not depend on the choice of $\chi$. By \eqref{EqGreenext} the weak solution $u$ of \eqref{incl} is the unique element of $\dot{V}_0(\Omega\backslash \overline{D})$ such that $\dot{\partial}_{\nu,i}u|_{\partial\Omega}=g$ in $\Hmo$, as desired.
\end{enumerate}
\end{remark}

\begin{figure}[h]
\centering
\begin{tikzpicture}
\draw (0,0) node {\includegraphics[scale=0.6]{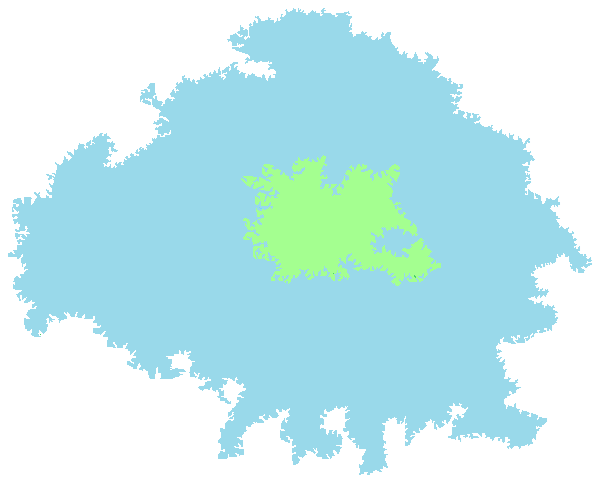}};
\draw (0.3,0.2) node{$D$};
\draw (1,-1.5) node{$\Omega$};
\end{tikzpicture}
\caption{An illustration of the imaging setting with an inclusion $D$ inside a domain $\Omega$, both two-sided $\dot H^1$-admissible. The purpose of the imaging problem is to identify the inclusion $D$ based on measurements on $\bord$.}
\label{Fig:Imaging}
\end{figure}

We write $\dot{\Scal}_{\partial \Omega}$ and $\dot{\Scal}_{\partial D}$ for the single layer potential operators with respect to $\Omega$ and $D$,  $\dot{\Dcal}_{\partial\Omega}$ for the double layer operator with respect to $\Omega$ and $\dot{\Kcal}_{\partial D}$ for the Neumann-Poincaré operator with respect to $D$, all in the $\dot{H}^1$-sense. By 
$\dot{\partial}_{\nu,i}u|_{\partial D}$ and $\dot{\partial}_{\nu,e}u|_{\partial D}$ we denote the interior and exterior normal derivative in the $\dot{H}^1$-sense with respect to $D$ of $u$ of $u\in \dot{V}_0(D)$ respectively $u\in \dot{V}_0(\Omega\backslash \overline{D})$; the extension of the exterior normal derivative to this latter space can be justified similarly as in Remark \ref{R:justify}. We write $\big\llbracket\dot{\partial}_{\nu}u\big\rrbracket_{\partial D}$ for their difference, and $\dot{\partial}_\nu u|_{\partial D}$ for their common value if they coincide.

The following representation formula for the weak solution $u$ to~\eqref{incl} is a generalization of \cite[Theorem 2.17]{AMMARI-2004}.

\begin{theorem}\label{ThImagRepr}
Let $\Omega$ and $D$ be two-sided $\dot{H}^1$-admissible and such that $D\subset\subset\Omega$, and let $k\in(0,1)\cup(1,+\infty)$. Given $g\in \Hmo$, let $u\in \dot{H}^1(\Omega)$ be the unique weak solution of \eqref{incl} in the $\dot{H}^1$-sense, let $f:=\dTTr_{\bord,i} u$ and 
$H=\dot{\Scal}_{\bord}\,g-\dot{\Dcal}_{\bord} f$. Then we have
\begin{equation}\label{rep}
u=H-\dot{\Scal}_{\partial D}\varphi,
\end{equation}
seen as an equality in $\dot{H}^1(\Omega)$, where $\varphi$ is the unique element of $\dot{\mathcal{B}}'(\partial D)$ such that 
\begin{equation}\label{phi}
\left(\frac{k+1}{2(k-1)}I+\dot{\Kcal}_{\partial D}^*\right)\varphi=\displaystyle\dot{\partial}_{\nu}H\big|_{\partial D}.
\end{equation}
The decomposition \eqref{rep} is the unique decomposition of $u$ into an element $H$ of $\dot{V}_0(\Omega)$ and a single layer potential with respect to $\partial D$, $-\dot{\Scal}_{\partial D}\varphi$. The function $H-\dot{\Scal}_{\partial D}\varphi$ is constant on $\R^n\backslash\overline{\Omega}$.
\end{theorem}

The function $H|_\Omega$ is the harmonic part of $u$ in $\Omega$, while $(-\dot{\Scal}_{\partial D}\varphi)|_\Omega$ is its refraction part.
The existence and uniqueness of $\varphi\in\dot{\mathcal{B}}'(\partial D)$ solving \eqref{phi} follow from Theorem \ref{ThSpecNPo}.

\begin{proof} Consider the bilinear form 
\[Q(w,v)=k\int_D\nabla w\cdot\nabla v\:\dx+\int_{\Omega\backslash\overline{D}}\nabla w\cdot\nabla v\:\dx,\quad w,v\in \dot{H}^1(\mathbb{R}^n\setminus\partial\Omega).\]
As before, let $\chi\in C_c^\infty(\Omega)$ be such that $0\leq \chi\leq 1$ and $\chi\equiv 1$ on a neighbourhood of $\overline{D}$.
Similarly as in Subsection \ref{Subsec:Superposition}, we call an element $w$ of $\dot{H}^1(\R^n\backslash \partial\Omega)$ a weak solution in the $\dot{H}^1$-sense of the transmission problem formally stated as
\begin{equation}\label{inclRn}
\begin{cases}
\displaystyle\nabla\cdot\Big(\big(1+(k-1)\mathds{1}_D\big)\nabla w\Big)=0 &\mbox{on }\nbord\\
w_i|_{\partial\Omega}-w_e|_{\partial\Omega}=f\\
\frac{\partial_iw_i}{\partial\nu}|_{\partial\Omega}-\frac{\partial_ew_e}{\partial \nu}|_{\partial\Omega}= g,
\end{cases}
\end{equation}
if it satisfies $Q(w,v)=0$ for all $v\in C_c^\infty(\R^n\backslash\partial\Omega)$, $\left\langle w,(1-\chi)v\right\rangle_{\dot{H}^1(\R^n\backslash \partial\Omega)}=\big\langle g,\dTr v\big\rangle_{\Hmo,\Hpo}$ for all $v\in \dot{V}_{0,\dot{\Scal}}(\R^n\backslash \partial\Omega)$ and finally $\left\langle w,v\right\rangle_{\dot{H}^1(\R^n\backslash \partial\Omega)}=\langle \dot\partial_\nu v, f\rangle_{\Hmo,\Hpo}$ for all $v\in \dot{V}_{0,\dot{\Dcal}}(\R^n\backslash \partial\Omega)$. As in Section \ref{Subsec:TSA-Jumps} the notations $w_i$ and $w_e$ in \eqref{inclRn} stand for the parts of the prospective weak solution $w$ on $\Omega$ respectively $\mathbb{R}^n\backslash \overline{\Omega}$.

Since the bilinear form $Q$ is comparable to $\left\langle \cdot,\cdot\right\rangle_{\dot{H}^1(\mathbb{R}^n\backslash \partial\Omega)}$ on $\dot{H}^1(\mathbb{R}^n\backslash \partial\Omega)$, there is a unique weak solution $w$ of \eqref{inclRn} in the $\dot{H}^1$-sense.

The function $w$, defined by $w:=u$ on $\Omega$ and $w:=0$ on $\R^n\backslash \overline{\Omega}$ is a weak solution of \eqref{inclRn} in the $\dot{H}^1$-sense. We now show that also $w':=H-\dot{\Scal}_{\partial D}\varphi$ is a weak solution of \eqref{inclRn} in the $\dot{H}^1$-sense. For $v\in C_c^\infty(\mathbb{R}^n\backslash \overline{\Omega})$ clearly $Q(w',v)=0$. For $v\in C_c^\infty(\Omega)$ we can, basically following the arguments in \cite[Lemma 3.3]{KANG-1996}, use the fact that $w'\in \dot{V}_0(D)\cap \dot{V}_0(\Omega\backslash\overline{D})$ and the corresponding analogs of \eqref{E:GGi} and \eqref{E:GGe} with $D$ in place of $\Omega$ to see that 
\[Q(w',v)=\big\langle k\:\dot{\partial}_{\nu,i} w'\big|_{\partial D}-\dot{\partial}_{\nu,e} w'\big|_{\partial D},\dTTr_{\partial D}v\big\rangle_{\!\B'(\partial D),\,\B(\partial D)}=0;\]
note that
\[k\:\dot{\partial}_{\nu,i} w'\big|_{\partial D}-\dot{\partial}_{\nu,e} w'\big|_{\partial D}=(k-1)\dot{\partial}_{\nu}H|_{\partial D}-\Big(\frac{k+1}{2}I+(k-1)\dot{\Kcal}^\ast_{\partial D}\Big)\varphi=0\]
by  Theorem \ref{T:Kdot} (ii) and \eqref{phi}. Given $v\in \dot{V}_{0,\dot{\Scal}}(\mathbb{R}^n\backslash\partial\Omega)$, we have 
\[\left\langle w',(1-\chi)v\right\rangle_{\dot{H}^1(\mathbb{R}^n\backslash\partial\Omega)}
= \big\langle \llbracket \dot{\partial}_\nu w'\rrbracket_{\partial \Omega},\dTr_{\bord} v\big\rangle_{\Hmo,\Hpo}
= \big\langle g,\dTr_{\bord} v\big\rangle_{\Hmo,\Hpo}\]
by analogs of \eqref{E:GGi} and \eqref{E:GGe} and the definition of $w'$. Given $v\in \dot{V}_{0,\dot{\Dcal}}(\mathbb{R}^n\backslash\partial\Omega)$, we similarly find that 
\[\left\langle w',v\right\rangle_{\dot{H}^1(\mathbb{R}^n\backslash\partial\Omega)}\notag\\
=\big\langle \dot{\partial}_\nu v|_{\partial\Omega}, \llbracket \dTr w'\rrbracket\big\rangle_{\Hmo,\Hpo}=\big\langle \dot{\partial}_\nu v|_{\partial\Omega}, f\big\rangle_{\Hmo,\Hpo}.\]
Consequently $w'$ is a weak solution of \eqref{inclRn} in the $\dot{H}^1$-sense, and therefore $w=w'$ in $\dot{H}^1(\R^n\backslash \partial\Omega)$ by uniqueness, which shows \eqref{rep} and the last claim in the theorem.

If \eqref{rep} holds with $H'\in \dot{V}_0(\Omega)$ and $\varphi'\in\dot{\mathcal{B}}'(\partial D)$ in place of $H$ and $\varphi$, then $H-\dot{\Scal}_{\partial D}\varphi=H'-\dot{\Scal}_{\partial D}\varphi'$, consequently $\dot{\Scal}_{\partial D}(\varphi-\varphi')\in \dot{V}_0(\Omega)$ and 
\[\big\llbracket \dot{\partial}_\nu\dot{\Scal}_{\partial D}(\varphi-\varphi')\big\rrbracket_{\partial D}=0.\] 
But this implies that $\varphi'=\varphi$ and therefore $H'=H$.

\end{proof}

\subsection{Subdomain identification}\label{SubsubDomInd}

We give a generalization of a theorem on the identification of \enquote{monotone} inclusions through a single boundary measurement. For Lipschitz domains, it can be found as a part of \cite[Theorem A.7]{AMMARI-2004}; the formulation below works for two-sided $\dot{H}^1$-admissible domains $\Omega$. 

\begin{theorem}\label{monotonous}
Let $\Omega$ be $\dot{H}^1$-admissible, let $D_1$ and $D_2$ be two-sided $\dot{H}^1$-admissible domains such that $D_1\subset D_2\subset\subset \Omega$ and let $k\in(0,1)\cup(1,+\infty)$. Suppose that $g\in\Hmo$ is nonzero and that $u_1$ and $u_2$ are the unique weak solutions in the $\dot{H}^1$-sense of \eqref{incl} with $D_1$ and $D_2$ in place of $D$ respectively. Then $\dTTr_{\bord,i} u_1=\dTTr_{\bord,i} u_2$ implies $D_1=D_2$.
\end{theorem}

\begin{proof} Suppose that $D_1\subsetneq D_2$. Since 
\[\int_\Omega{\big(1+(k-1)\mathds{1}_{D_1}\big)\nabla u_1 \cdot\nabla v\,\dx}=\int_\Omega{\big(1+(k-1)\mathds{1}_{D_2}\big)\nabla u_2 \cdot\nabla v\,\dx}\]
for all $v\in \dot{H}^1(\Omega)$, it follows that 
\begin{equation}\label{E:startingpoint}
\int_\Omega{\big(1+(k-1)\mathds{1}_{D_1}\big)\nabla(u_1-u_2)\cdot\nabla v\,\dx}=(k-1)\int_{D_2\backslash D_1}{\nabla u_2\cdot\nabla v\,\dx}.
\end{equation}
Testing with $v=u_1-u_2$ gives
\begin{align}\label{img2}
\int_\Omega \big(1+(k-1)\mathds{1}_{D_1}\big)|\nabla(u_1-u_2)|^2\,\dx & +(k-1)\int_{D_2\backslash D_1} |\nabla u_2|^2\,\dx\notag\\
&=(k-1)\int_{D_2\backslash D_1} \nabla u_2 \cdot\nabla u_1\,\dx,
\end{align}
and testing with $v=u_1$ shows that the right-hand side of \eqref{img2} equals
\begin{equation}\label{img3}
\int_\Omega{\big(1+(k-1)\mathds{1}_{D_1}\big)\nabla(u_1-u_2)\cdot\nabla u_1\,\dx}.
\end{equation}
By \eqref{E:weakreconstruct} the function $u_1$ is harmonic in $\Omega$ (that is, orthogonal to $C_c^\infty(\Omega)$) with respect to the equivalent scalar product
\[(u,w)\mapsto \int_\Omega\big(1+(k-1)\mathds{1}_{D_1}\big)\nabla u\cdot\nabla w\,\dx\]
on $\dot{H}^1(\Omega)$. Therefore, and since $u_1-u_2\in \ker \dTTr_{\partial\Omega,i}$, an analog of Theorem \ref{Tr0isom} (i) shows that \eqref{img3} is zero and therefore also \eqref{img2}. 

Suppose that $k>1$.
Since the first summand on the left-hand side of \eqref{img2} must be zero, we have $u_1=u_2$ in $\dot{H}^1(\Omega)$.
Since also the second summand must be zero, we have $u_2|_{D_2\setminus\overline{D}_1}=0$ in $\dot{H}^1(D_2\setminus\overline{D}_1)$. However, this would mean that each representative up to constants of $u_2|_{\Omega\setminus \overline{D}_1}\in \dot{V}_0(\Omega\setminus \overline{D}_1)$ would be constant on the open set $D_2\backslash\overline{D}_1$. Since $D_1$ is two-sided $\dot{H}^1$-admissible, we have $D_2\cap \partial D_1=D_2\cap \partial(\R^n\backslash \overline{D}_1)=D_2\cap \partial(D_2\backslash \overline{D}_1)$. Together with the fact that $D_2\backslash D_1$ is nonempty, this shows that $D_2\backslash \overline{D}_1$ cannot be empty. Consequently, by its harmonicity in $\Omega\setminus \overline{D}_1$ (in the sense of \eqref{VlOr} and \eqref{DOo}),
each representative up to constants of $u_2|_{\Omega\setminus \overline{D}_1}$ would have to be constant on all of $\Omega\setminus \overline{D}_1$, contradicting the boundary condition in \eqref{incl} with nonzero $g$. Consequently $D_1=D_2$ in this case. 

For $k\in(0,1)$, the result follows by the same arguments, but with (\ref{E:startingpoint}) replaced by 
\[\int_\Omega{\big(1+(k-1)\mathds{1}_{D_2}\big)\nabla(u_2-u_1)\cdot\nabla v\,\dx}=(1-k)\int_{D_2\backslash D_1}{\nabla u_1\cdot\nabla v\,\dx}.\]
\end{proof}

The representation formula in Theorem~\ref{ThImagRepr} can be simplified if $n=2$ in the case of an open disk $D$. Note that, although the geometry of the inclusion is assumed to be smooth here, the geometry of the larger domain can still be irregular. For that matter, the boundary data is still understood in the sense of $\dot\B$ and $\dot\B'$, and the operators are defined in the variational sense.

\begin{corollary}\label{corollary3}
Let $\Omega\subset \mathbb{R}^2$ be two-sided $\dot{H}^1$-admissible, $D$ be an open disk such that $D\subset\subset\Omega$, and let $k\in(0,1)\cup(1,+\infty)$. Given $g\in \Hmo$, let $u\in \dot{H}^1(\Omega)$ be the unique weak solution to \eqref{incl} in the $\dot{H}^1$-sense, let $f:=\dTTr_{\bord,i} u$ and 
$H=\dot{\Scal}_{\bord}\,g-\dot{\Dcal}_{\bord} f$. Then we have 
\begin{equation}\label{rep2}
u=H-\frac{2(k-1)}{k+1}\,\dot{\Scal}_{\partial D}\!\big(\dot{\partial}_\nu H |_{\partial D}\big),
\end{equation}
seen as an equality in $\dot{H}^1(\Omega)$.
\end{corollary}

\begin{proof}
As it was pointed out in Remark~\ref{Rem:SameK*}, if $D$ is a disk, then $\dot\Kcal^*_{\partial D}$ can be represented using the usual kernel formula~\cite{mitrea_generalization_2008,VERCHOTA-1984}.  By~\cite[Section 4]{KANG-1996}, $\dot{\Kcal}^*_{\partial D}\varphi=0$ for all $\varphi\in L^2_0(\partial D)$. Recall that, up to norm equivalence, $\dot\B'(\partial D)$ equals $\dot{H}^{-\frac12}(\partial D)$, cf. Remark \ref{R:boundary_spaces}. Since $L^2_0(\partial D)$ is dense in this space and $\dot\Kcal^*_{\partial D}$ is bounded on $\dot\B'(\partial D)$, we have 
\begin{equation}\label{E:Kastzero}
\dot{\Kcal}^*_{\partial D}\varphi=0,\qquad \varphi\in \dot\B'(\partial D).
\end{equation}
Combining this with \eqref{phi}, formula \eqref{rep2} follows.
\end{proof}

Using Theorem \ref{monotonous} and Corollary \ref{corollary3}, the identification of a general disk-shaped inclusion when $\Omega$ is a two-sided extension domain of $\R^2$ follows by slight variations of the arguments in~\cite[Theorem A]{KANG-1996}.

\begin{theorem}\label{ThImagDisks}
Let $\Omega\subset \mathbb{R}^2$ be two-sided $\dot{H}^1$-admissible, $D_1$ and $D_2$ be open disks such that $D_1, D_2\subset\subset \Omega$ and
let $k\in(0,1)\cup(1,+\infty)$. Suppose that $g\in\Hmo$ is nonzero and that $u_1$ and $u_2$ are the unique weak solutions in the $\dot{H}^1$-sense of \eqref{incl} with $D_1$ and $D_2$ in place of $D$ respectively. Then $\dTTr_{\bord,i} u_1=\dTTr_{\bord,i} u_2$ implies $D_1=D_2$.
\end{theorem}
	
\begin{proof}
Assume $D_1\cap D_2\not\in\{D_1,D_2\}$, otherwise Theorem \ref{monotonous} yields the result. Let us denote $f:=\dTTr_{\bord,i}u_1=\dTTr_{\bord,i}u_2$. The representation formula \eqref{rep2} states that for $p=1,2$ we have 
\begin{equation*}
u_p=H-\frac{2(k-1)}{k+1}\,\dot{\Scal}_{\partial D_p}\!\big(\dot{\partial}_\nu H|_{\partial D_p}\big)
\end{equation*}
 in $\dot H^1(\Omega)$ with $H=\dot{\Scal}_{\bord}\,g-\dot{\Dcal}_{\bord} f \in \dot H^1(\R^2\backslash \partial \Omega)$, as before. Consequently 
\[\dTTr_{\partial\Omega,i} \dot{\Scal}_{\partial D_1}\!\big(\dot{\partial}_\nu H|_{\partial D_1}\big)=\dTTr_{\partial\Omega,i} \dot{\Scal}_{\partial D_2}\!\big(\dot{\partial}_\nu H|_{\partial D_2}\big)\]
in $\Hmo$, and therefore 
\[\dot{\Scal}_{\partial D_1}\!\big(\dot{\partial}_\nu H|_{\partial D_1}\big)=\dot{\Scal}_{\partial D_2}\!\big(\dot{\partial}_\nu H|_{\partial D_2}\big) \quad \text{in $\dot{H}^1(\R^2\backslash \overline{\Omega})$}\]
by the uniqueness of the exterior Dirichlet problem. The equality remains true if we replace the single layer potentials $\dot{\Scal}_{\partial D_p}$ by their zero trace jump readjusted variants $\overline{\Scal}_{\partial D_p}$ as in \eqref{E:readjustsingle}. Let $s_p\in \overline{\Scal}_{\partial D_p}\!\big(\dot{\partial}_\nu H|_{\partial D_p}\big)$ be representatives modulo constants such that $s_1=s_2$ on $\mathbb{R}^2\backslash \overline{\Omega}$; 
by the regularity of $H$ we may assume $s_1$ and $s_2$ are continuous on $\mathbb{R}^2$, cf. Remark \ref{R:strictcaseSLdot}. Their harmonicity on $\mathbb{R}^2\backslash \overline{D}_1$ respectively $\mathbb{R}^2\backslash \overline{D}_2$ and their continuity
then imply that $s_1=s_2$ on $\mathbb{R}^2\backslash (D_1\cup D_2)$.

If $D_1\cap D_2=\emptyset$, then $s_2$ is a harmonic extension of $s_1$ to $D_1$ with the same trace on $\partial D_1$, hence $s_1=s_2$ in $D_1$. This implies that $s_1$ itself must be harmonic on all of $\mathbb{R}^2$ and therefore constant. But then $\dot{\partial}_\nu H=0$, so that $H$ is constant, hence $\dot{\Scal}_{\bord}\,g=0$ in $\dot H^1(\R^2\backslash \partial \Omega)$, which contradicts the nonzero boundary condition in \eqref{incl}.

Now suppose that $D_1\cap D_2\neq \emptyset$. By \eqref{E:Kastzero} we have 
\[\partial_{\nu,i}\big|_{\partial D_p}\circ \dot{\Scal}_{\partial D_p}\!\big(\dot{\partial}_\nu H|_{\partial D_p}\big)=\frac{1}{2}\dot{\partial}_\nu H|_{\partial D_p},\quad p=1,2,\]
and therefore, by the uniqueness of the Neumann problem,
\[\dot{\Scal}_{\partial D_p}\!\big(\dot{\partial}_\nu H|_{\partial D_p}\big)=\frac{1}{2}H\quad\text{in $\dot{H}^1(D_p)$},\quad p=1,2.\]
But this implies that $s_1$ equals $s_2$ plus a constant on $D_1\cap D_2$, and since $s_1=s_2$ outside $D_1\cup D_2$, this constant must be zero by continuity. By harmonicity then $s_1=s_2$ on $D_1\cup D_2$. Using $s_2$ as an extension of $s_1$, we find that $s_1$ must be harmonic on $\mathbb{R}^2$, which gives the same contradiction as before.
\end{proof}

\appendix

\section{Background proofs for Section \ref{SecBoundCase}}\label{PSecBoundCase}

We collect some background on the potential theoretic notions used in Section \ref{SecBoundCase}. 

The \emph{capacity} $\cpct(U)$ of an open set $U\subset\mathbb{R}^n$ is defined by
\[\cpct(U):=\inf\big\lbrace \|u\|_{H^1(\mathbb{R}^n)}^2: u\in H^1(\mathbb{R}^n),\ u\geq 1\ \text{a.e. on $U$}\big\rbrace\]
with the agreement that $\inf \emptyset=+\infty$. The capacity $\cpct(A)$ of a general set $A\subset \mathbb{R}^n$ is defined by
\[\cpct(A):=\inf\big\lbrace \cpct(U):\ A\subset U,\ \text{$U$ open}\big\rbrace.\]
See \cite[Section 2.3]{CHEN-FUKUSHIMA-2012}, \cite[Section 2.1]{FOT94}  or \cite{ADAMS-1996, MAZ'JA-1985}. 

A set of zero capacity has zero Lebesgue measure. For $n=1$ all nonempty sets have positive capacity.
A property which holds outside a set of zero capacity is said to hold \emph{quasi everywhere}, or short, \emph{q.e.} 

An extended real valued function $v$ defined q.e. on $\mathbb{R}^n$ is \emph{quasi continuous} if for any $\varepsilon>0$ there is an open set $G\subset \mathbb{R}^n$ such that $\cpct(G)<\varepsilon$ and $v$ is continuous on $\mathbb{R}^n\setminus G$. Each element $u$ of $H^1(\R^n)$ has a quasi continuous representative $\widetilde{u}$, see for instance \cite[Theorem 2.3.4]{CHEN-FUKUSHIMA-2012}, \cite[Theorem 2.1.3]{FOT94} or \cite{ADAMS-1996, MAZ'JA-1985}. Two quasi continuous representatives of the same element $u$ of $H^1(\mathbb{R}^n)$ agree q.e. on $\mathbb{R}^n$, see \cite[p. 71]{FOT94} or \cite[Theorem 6.1.4]{ADAMS-1996}. 

For any open set $\Omega\subset \R^n$, the bilinear form \eqref{E:spdot}, endowed with the domain $H^1(\Omega)$, is a Dirichlet form \cite[Definition 1.1.2]{CHEN-FUKUSHIMA-2012}, see also \cite[Section 1.1]{FOT94}. By $H^1_e(\Omega)$ we denote its extended Dirichlet space \cite[Definition 1.1.4]{CHEN-FUKUSHIMA-2012}, see also \cite[Section 1.5]{FOT94}. For $\Omega=\R^n$ it is well known that 
\begin{equation}\label{E:geq3}
\text{if $n\geq 3$, then $(H_e^1(\R^n),\left\langle\cdot,\cdot\right\rangle_{\dot{H}^1(\R^n)})\cong (\dot{H}^1(\R^n),\left\langle\cdot,\cdot\right\rangle_{\dot{H}^1(\R^n)})$}
\end{equation}
in the sense that the vector spaces are isomorphic and the isomorphism is a Hilbert space isometry, and 
\begin{equation}\label{E:eq2}
\text{if $n\leq 2$, then $(H_e^1(\R^n),\left\langle\cdot,\cdot\right\rangle_{\dot{H}^1(\R^n)})=(\dot{H}^1(\R^n)\oplus \R,\left\langle\cdot,\cdot\right\rangle_{\dot{H}^1(\R^n)})$},
\end{equation}
seen as an equality of vector spaces endowed with bilinear forms. Proofs of \eqref{E:geq3} and \eqref{E:eq2} can for instance be found in \cite[Theorems 2.2.12 and 2.2.13]{CHEN-FUKUSHIMA-2012}.

{\color{black} Also for each $u\in \dot{H}^1(\R^n)$ we can find a quasi continuous representative $\widetilde{u}\in u$; the difference between two quasi continuous representatives of the same $u\in \dot{H}^1(\R^n)$ is constant q.e. on $\mathbb{R}^n$. This follows from \cite[Theorem 2.3.4]{CHEN-FUKUSHIMA-2012} and the identities \eqref{E:geq3} and \eqref{E:eq2}. }

We provide arguments and references for Theorems \ref{Trisom} (i) and \ref{Tr0isom} (i). We start with Theorem \ref{Trisom} (i).

\begin{proof}[Proof of Theorem \ref{Trisom} (i)] We consider the quadratic form defined by \eqref{E:spdot} with $\Omega=\R^n$ and the domain $H^1(\R^n)$. For (i) it suffices to note that by \cite[Corollary 2.3.1 and Example 2.3.1]{FOT94} we have 
\begin{equation}\label{identH10}
H_0^1(\R^n\backslash\partial\Omega)=\{w\in H^1(\R^n)\ |\ \text{$\widetilde{w}=0$ q.e. on $\partial\Omega$}\},
\end{equation}
where $\widetilde{w}$ is a quasi continuous version of $w$. Now $u\in H^1(\Omega)$ has an extension to an element of $H_0^1(\R^n\backslash\partial\Omega)$ if and only if $u\in H^1_0(\Omega)$, and it has an extension to an element of the right-hand side of \eqref{identH10} if and only if $\TTr u=0$. 
\end{proof}

A proof of Theorem \ref{Tr0isom} (i) is given in \cite[Example 2.3.2]{FOT94}, only in a slightly different language. We sketch how to adapt it to our formulation.

\begin{proof}[Proof of Theorem \ref{Tr0isom} (i)] An element $u\in \dot{H}^1(\Omega)$ is in $\ker \dTTr$ if and only if it is the restriction to $\Omega$ of an element of 
\begin{equation}\label{E:globaltracezero}
\{w\in \dot{H}^1(\R^n)\ |\ \text{$\widetilde{w}=0$ q.e. on $\partial\Omega$}\},
\end{equation}
where $\widetilde{w}=0$ is understood representative wise and modulo constants. For $n\geq 3$ the space \eqref{E:globaltracezero} may be identified with 
\begin{equation}\label{E:globalextendedtracezero}
\{v\in H_e^1(\R^n)\ |\ \text{$\widetilde{v}=0$ q.e. on $\partial\Omega$}\}
\end{equation}
under the isometry in \eqref{E:geq3}. By \cite[Theorem 2.3.3 and Example 2.3.2]{FOT94} the space \eqref{E:globalextendedtracezero} coincides with $H_{0,e}^1(\R^n\backslash\partial\Omega)$, the extended Dirichlet space of \eqref{E:spdot}, endowed with the smaller domain $H_0^1(\R^n\backslash\partial\Omega)$. The complement of $H_{0,e}^1(\R^n\backslash\partial\Omega)$ in $H_e^1(\R^n)$ is isometric to 
\begin{equation}\label{E:globalcomplement}
\bigg\{w\in\dot{H}^1(\R^n)\ \bigg|\ \text{$\int_{\R^n}\nabla w\cdot\nabla v~\dx=0$ for all $v\in C_c^\infty(\R^n\backslash \partial\Omega)$}\bigg\},
\end{equation}
and since $\R^n\backslash \partial\Omega$ is the disjoint union of $\Omega$ and $\R^n\backslash \overline{\Omega}$, 
the restriction of this space to $\Omega$ is $\dot{V}_0(\Omega)$. For $n=2$ it is shown in \cite[Example 2.3.2]{FOT94} that the space \eqref{E:globalextendedtracezero} still coincides with $H_{0,e}^1(\R^n\backslash\partial\Omega)$, and  (simplifications of) the same arguments give this coincidence also for $n=1$. By \eqref{E:eq2}, the space \eqref{E:globalextendedtracezero} may then be regarded as a closed subspace of $\dot{H}^1(\R^n)$, and the complement of this closed subspace is \eqref{E:globalcomplement}.
\end{proof}

\def\refname{References}
\bibliographystyle{siam}
\bibliography{biblio}

\end{document}